\documentclass{imsart}
\RequirePackage{amsthm,amsmath,amsfonts,amssymb}
\RequirePackage[numbers]{natbib}
\RequirePackage[colorlinks,citecolor=blue,urlcolor=blue]{hyperref}
\RequirePackage{graphicx}

\usepackage[bbgreekl]{mathbbol}
\usepackage{multirow}
  \usepackage{todonotes}
\usepackage{ulem}

\usepackage{fancybox}
\usepackage{tikz}
\usepackage{tikz-cd}
\usetikzlibrary{arrows.meta,shapes,positioning,bending,automata,backgrounds,petri}

\usepackage{pgfplots}
\pgfplotsset{compat=1.18}
\usepackage{wrapfig}   
\usepackage{xstring}

\arxiv{2402.11154}

\startlocaldefs

\theoremstyle{plain}
\newtheorem{thm}{Theorem}[section]
\newtheorem{lem}[thm]{Lemma}
\newtheorem{prop}[thm]{Proposition}
\newtheorem{cor}[thm]{Corollary}

\newtheoremstyle{sltheorem}
{}                
{}                
{\slshape}        
{}                
{\bfseries}       
{.}               
{ }               
{\thmname{#1}-\thmnumber{#2}\thmnote{(#3)}}               
\theoremstyle{sltheorem}
\newtheorem{assum}{HD}
\newtheorem{assump}{HD} 

\newtheorem{assumc}{HC}
\newtheorem{assumcp}{HC} 

\theoremstyle{definition}
\newtheorem{defn}{Definition}
\newtheorem{xmpl}{Example}

\newcommand{\us}[2]{\underset{#1}{\underbrace{#2}}}
\newcommand{\R}{{\mathbb R}}
\newcommand{\Z}{{\mathbb Z}}
\newcommand{\N}{{\mathbb N}}
\newcommand{\X}{{\mathbb X}}
\newcommand{\bX}{\pmb{\X}}
\newcommand{\ttau}{\bbtau}

\newcommand{\refcol}[1]{
\StrGobbleLeft{#1}{1}[\tempstr]
\StrLeft{\tempstr}{1}[\tempstr]
{\tiny\IfStrEqCase{\tempstr}{%
{1}{\color{red}}%
{2}{\color{blue}}}%
[\color{purple}]#1}
}

\newcommand{\tro}[1]{}
\newcommand{\rito}[3]{}
\newcommand{\srito}[3]{}

\endlocaldefs

\begin{document}
\begin{frontmatter}
\title{Representation and Characterization of Quasistationary Distributions for Markov Chains}
\runtitle{QSDs for Markov Chains}

\begin{aug}
\author[A]{\fnms{Iddo}~\snm{Ben-Ari}\ead[label=e1]{iddo.ben-ari@uconn.edu}} 
\and
\author[B]{\fnms{Ningwei}~\snm{Jiang}\ead[label=e2]{ningwei.jiang@uconn.edu}}
\address[A]{Department of Mathematics, University of Connecticut\printead[presep={,\ }]{e1}}

\address[B]{Department of Mathematics, University of Connecticut\printead[presep={,\ }]{e2}}

\runauthor{I. Ben-Ari, N. Jiang}
\end{aug}

\begin{abstract}
This work provides a complete description of Quasistationary Distributions (QSDs) for Markov chains with a unique absorbing state and an irreducible set of non-absorbing states. As is well-known, every QSD has an associated absorption parameter describing the exponential tail of the absorption time with the QSD as the initial distribution. Our analysis of existence and representation of QSDs corresponding to a given parameter is according to whether the moment generating function of the absorption time starting from any non-absorbing state evaluated at the parameter is finite or infinite, the  {\it finite} or {\it infinite} moment generating function regimes, respectively. For absorption parameters in the finite regime, all QSDs are in the convex cone of a Martin entry boundary associated with the parameter. The infinite regime corresponds to at most one absorption parameter value. In this regime, when a QSD exists, it is unique and can be represented by a renewal-type formula. Multiple applications are presented, including revisiting some of the main classical results in the area. 
\end{abstract} 


\end{frontmatter}



\tableofcontents
\section{Introduction}
\subsection{Assumptions and Definitions}
Let $\Z_+=\{0,1,2,\dots\}$ denote the set of nonnegative integers, and $\N=\{1,2,3,\dots\}$ denote the set of natural numbers. We also write $\R_+$ for the set of nonnegative real numbers. 

Consider a Markov chain ${\bf X}=(X_n: n\in \Z_+)$ on a state space which is a disjoint union of the countable set $S$ and the singleton $\{\Delta\}$\tro{R2:M1}. Let $p$ denote the transition function for ${\bf X}$. As usual, we write $P_\mu$ and $E_{\mu}$ for the probability and expectation associated with ${\bf X}$ under the initial distribution $\mu$, with $P_x$ and $E_x$ serving as shorthand for the respective notions when the initial distribution is the  point-mass distribution at $x\in S \cup\{\Delta\}$, $\delta_x$. For $x\in S \cup \{\Delta\}$, denote the return and hitting time of $x$: 
\begin{equation}
\tau_x = \inf\{n\in\N:X_n = x\},
\end{equation}
and 
\begin{equation}
    \label{eq:hitting0}
    ^0\tau_x=\inf\{n \in \Z_+: X_n =x\},
\end{equation}
respectively, where here and henceforth $\Z$ is the set of integers, $\Z_+$ is the set of nonnegative integers and  $\N=\{1,2,3,\dots\}$ is the set of positive integers.  In order to define a QSD we introduce the following hypotheses: 
\begin{assum}\label{as:reg:uab} $\tau_\Delta < \infty$ $P_x$-a.s. for some $x\in S$. \end{assum}  
\begin{assum}\label{as:reg:Sirr} The restriction of $p$ to $S$ is irreducible.   
\end{assum}
As a result, $\Delta$ is a unique absorbing state. We therefore refer to $\tau_\Delta$ as the absorption time.  Moreover, for  all $x\in S$ and $n\in\Z_+$, $P_x (\tau_\Delta> n) >0$.  
\begin{defn}[QSD]
\label{def:QSD}
Let {\bf HD-\ref{as:reg:uab},\ref{as:reg:Sirr}}  hold.  A probability measure $\nu$ on $S$ is a Quasistationary Distribution (QSD) if 
\begin{equation}
    \label{eq:QSD_defn}
    P_\nu (X_n \in \cdot~ |~\tau_\Delta >n) = \nu(\cdot)\quad
\end{equation}  
for all $n\in\N$.
\end{defn} 
Note that the irreducibility hypothesis  {\bf HD-\ref{as:reg:Sirr}} guarantees that if $\nu$ is a QSD then  $\nu(j)>0$ for all $j\in S$. Proposition \ref{prop:eigen_equivalence} below (or a direct calculation) show that $\nu$ is a QSD if and only if \eqref{eq:QSD_defn} holds for $n=1$. As is well-known, \cite[Theorem 2]{Vere-Jones_Limit}  $\nu$ is a QSD if and only if it is a ``Quasi limiting distribution" in the sense that there exists some probability measure $\mu$ on $S$ so that  
\begin{equation}
     \label{eq:QLD}
     \lim_{n\to\infty} P_\mu(X_n \in \cdot ~|~ \tau_\Delta >n)=\nu. 
    \end{equation}
The idea of a limiting conditional distribution traces back to as early as 1931 by Wright \cite{WrightFirstQSD} in the discussion of gene frequencies in finite populations. Later, Bartlett introduced the notion of ``quasi stationarity" \cite{Bartlett1957} and coined the term ``quasi-stationary distribution" in the context of a birth and death process \cite{Bulmer1961}. Yaglom \cite{Yaglom1947} was the first who showed explicitly that the limit \tro{R1:G1}\eqref{eq:QLD} holds for a non-degenerate subcritical branching process starting from any deterministic initial population, with a limit independent of the initial population. To this day, quasi-limiting distributions corresponding to a determinstic initial distributions are called Yaglom limits.

The definition of a QSD immediately leads to the following characterization of QSD, \cite[Theorem 1]{Coolen} or \cite{Pierrebook}.\tro{R2:M2}
\begin{prop}
\label{prop:eigen_equivalence}
Let {\bf HD-\ref{as:reg:uab},\ref{as:reg:Sirr}} hold. 
A probability measure $\nu$ on $S$ is a QSD if and only if there exists some $\lambda\in(0,\infty) $ such that 
\begin{equation}
    \label{eq:eigen_description}
    \sum_{i\in S } \nu(i)p(i,j) = e^{-\lambda} \nu(j),~j\in S.
\end{equation}
In this case, \begin{equation} 
\label{eq:QSD_zeta_tail}
P_\nu(\tau_\Delta>n)=e^{-\lambda n},~n\in\Z_+.
\end{equation} 
\end{prop}
We comment that a measure on $S$ (not necessarily finite) satisfying \eqref{eq:eigen_description} is often referred to as an $e^{-\lambda}$-invariant measure. 

Restating \eqref{eq:QSD_zeta_tail}, if $\nu$ is a QSD, then under $P_\nu$,  $\tau_\Delta$ is geometric with parameter $1-e^{-\lambda}$. For this reason, in what follows we refer to $\lambda$ as the absorption parameter associated with the QSD $\nu$. 

For completeness, here is the proof of the proposition. 
\begin{proof}
    Assume first that  $\nu$ is a QSD.  For $j \in S$ by Definition \ref{def:QSD} $P_\nu(X_1=j|\tau_\Delta>1)=\nu(j)$. But 
    $$P_\nu (X_1=j | \tau_\Delta> 1) = \frac{P_\nu (X_1= j) }{ P_\nu (\tau_\Delta > 1)}= \rho^{-1}\sum_{i\in S}\nu (i) p (i,j),$$ 
    where $\rho= P_\nu (\tau_\Delta >1)\in (0,1]$.  This gives 
    $$  \sum_{i\in S} \nu (i) p(i,j) =\rho \nu (j).$$
    Consequently,  for every $n\in\Z_+$ and $j\in S$,   $(\nu p^n)(j) = \rho ^n \nu(j)$, and by summing both sides over $j\in S$, we obtain $P_\nu (\tau_\Delta > n) = \rho^n$. Hypothesis {\bf HD-\ref{as:reg:uab}}, then guarantees that $\rho\in (0,1)$ and is therefore equal to $e^{-\lambda}$ for some $\lambda \in (0,\infty)$. 

    Conversely, if \eqref{eq:eigen_description} holds, then for all $n\in\Z_+$ and $j\in S$, $(\nu p^n)(j)  = e^{-\lambda n} \nu(j)$, equivalently $P_\nu(X_n =j, \tau_\Delta > n) = e^{-\lambda n } \nu(j)$ and summing over $j\in S$ gives $P_\nu (\tau_\Delta >n ) = e^{-\lambda n}$, and the result follows. 
\end{proof}
As \eqref{eq:QSD_zeta_tail} necessitates the finiteness of some exponential moments of $\tau_\Delta$, when discussing QSDs we can replace {\bf HD-\ref{as:reg:uab}} with the following stronger hypothesis: 
\begin{assump}\label{as:reg:moments}
  There exists  $\beta>0$ such that $E_x [\exp(\beta\tau_\Delta)]<\infty$  for some  $x\in S$.
 \end{assump}
 We note that under the irreducibility asasumption  {\bf HD-\ref{as:reg:Sirr}}, the assumption on existence of finite exponential moments {\bf HD-\ref{as:reg:moments}} is equivalent to $E_x [ \exp(\beta \tau_\Delta)]<\infty$ for all $x\in S$.   Yet,  {\bf HD-\ref{as:reg:moments}} is not a sufficient condition, as the following example shows: 
 \begin{xmpl}
\label{xmpl:ind_zeta}
Let $\mathfrak p$ be any irreducible transition function on  $S$.  Fix $\lambda_{cr}>0$, and define a  transition function $p$ on $S\cup \{\Delta\}$ by letting $p(x,y) = e^{-\lambda_{cr}} {\mathfrak p}(x,y)$ when $x,y\in S$,  $p(x,\Delta) = 1-e^{-\lambda_{cr}}$ for $x\in S$, and $p(\Delta,\Delta)=1$.  In terms of sample paths, $\mathfrak p$ and $p$ are related as follows.  Let ${\bf Y}$ be a MC corresponding to $\mathfrak p$, and  let $\tau_\Delta$ be a geometric random variable with parameter $1-e^{-\lambda_{cr}}$, independent of ${\bf Y}$. Now for $t\in \Z_+$ set 
$$X_n = Y_n {\bf 1}_{\{\tau_\Delta>n\}}.$$
Then ${\bf X}=(X_n:n\in \Z_+)$ is a MC with transition function  $p$, and \eqref{eq:QSD_zeta_tail} holds with $\lambda=\lambda_{cr}$ and any distribution $\nu$ on $S$. 
However, for every probability measure $\nu$  on $S$,   $P_\nu( X_n \in \cdot ~| ~\tau_\Delta > n) = P_{\nu} (Y_n \in \cdot)$ and so $\nu$ is a QSD for $p$ if and only if it is a stationary distribution for $\mathfrak p$, and this holds if and only if $\mathfrak p$ is positive recurrent.
\end{xmpl} 
Our work rests on the following definition:  
 \begin{defn}
 \label{def:lambda_Delta_def}
 Let {\bf HD-\ref{as:reg:moments},\ref{as:reg:Sirr}} hold. 
 \begin{enumerate} 
 \item The critical absorption parameter $\lambda_{cr}$ is defined as 
 \begin{equation} 
 \label{eq:lambda_cr_def}
 \lambda_{cr} = \sup \{\lambda>0: E_x[e^{\lambda \tau_\Delta}]<\infty\mbox{ for some }x\in S\},
\end{equation}
\item A parameter $\lambda \in (0,\lambda_{cr}]$ is in the finite MGF regime if $E_x [ \exp (\lambda \tau_\Delta)] < \infty$ for some $x\in S$.
\item The critical absorption parameter $\lambda_{cr}$ is in the infinite MGF regime if 
$E_x [ \exp (\lambda_{cr} \tau_\Delta)]=\infty$ for some $x\in S$. 
\end{enumerate} 
\tro{R1:M3}
\end{defn} 
\tro{R1:P6}As the results in the next sections demonstrate, the analysis leading to the characterization of QSDs for a given absorption parameter rests entirely on whether the parameter is in the finite or infinite regime. It should be noted that in general $\lambda_{cr}$ may be in the infinite or finite regime. The ``hub and two spokes" process disucss in Section \ref{sec:B_D} provides a simple example of both behaviors, depending on choice of parameter. Notable cases where $\lambda_{cr}$ is in the infinite MGF regime include the case when $S$ is finite, Proposition \ref{prop:Sfinite} and subcritical branching, Section \ref{sec:subcritical_branching}.
\subsection{Prelimiary Results}
\begin{cor}~
Let {\bf HD-\ref{as:reg:moments},\ref{as:reg:Sirr}} hold. Then 
\label{cor:minimal}
\begin{enumerate} 
\item $\lambda_{cr} \in (0,\infty)$.
\item For every $\lambda \in (0,\lambda_{cr}]$, 
\begin{equation}
\label{eq:summation} E_x [\exp(\lambda \tau_\Delta)] = e^{\lambda} + (e^\lambda -1) \sum_{m=1}^\infty e^{\lambda m} P_x (\tau_\Delta>m).
\end{equation}
\item If $\nu$ is a QSD with absorption parameter $\lambda$, then $\lambda\leq \lambda_{cr}$. 
\end{enumerate} 
\end{cor} 
\begin{proof} 
\tro{R1:M1} Clearly {\bf HD-\ref{as:reg:moments}} implies $\lambda_{cr}>0$. We will show that $\lambda_{cr}<\infty$ after we prove \eqref{eq:summation}. Note that by  monotone convergence  it is sufficient to prove \eqref{eq:summation} for$\lambda \in (0,\lambda_{cr})$. We therefore assume $\lambda \in (0,\lambda_{cr})$ and use summation by parts. Clearly, 
$$E_x [\exp (\lambda \tau_\Delta) ] = \lim_{M\to\infty} \sum_{m=1}^M e^{\lambda m} [P_x (\tau_\Delta> m-1) - P_x (\tau_\Delta>m)].$$
Summation by parts gives 
\begin{align*}  \sum_{m=1}^M e^{\lambda m} [P_x (\tau_\Delta> m-1) - P_x (\tau_\Delta>m)] &=
 e^{\lambda}- e^{\lambda M} P_x (\tau_\Delta>M) + (e^{\lambda}-1)\sum_{m=1}^{M-1}e^{\lambda m} P_x (\tau_\Delta>m).
\end{align*}
Let $\lambda' \in (\lambda,\lambda_{cr})$. Then Markov's inequality gives 
$$e^{\lambda M} P_x (\tau_\Delta>M) \le e^{\lambda M} e^{-\lambda' M} E_x [\exp (\lambda' \tau_\Delta)]\underset{M\to\infty}{\to} 0,$$
and therefore \eqref{eq:summation} holds for $\lambda\in (0,\lambda_{cr})$. 

The irreducibility hypothesis {\bf HD-\ref{as:reg:Sirr}} implies that for every $x\in S$ there exists some $n\in\N$ such that  $p^n(x,x)>0$, hence - by induction - for any $k\in\N$,  $P_x (\tau_\Delta > kn ) \ge (p^n (x,x))^k$. Pick  $\lambda \in (0,\lambda_{cr})$.  As the lefthand side of \eqref{eq:summation} is finite, the series on the righthand side converges, and so 

$$\sum_{m=1}^\infty e^{\lambda m} P_x (\tau_\Delta>m)\ge \sum_{k=1}^\infty e^{\lambda nk} P_x (\tau_\Delta>nk) \ge \sum_{k=1}^\infty( e^{\lambda n} p^n (x,x))^k.$$

The righthand side is a geometric series. Its convergence implies $e^{\lambda n} p^n(x,x)<1$, which implies   $e^{\lambda} \le (p^n (x,x))^{-1/n}<\infty$, and therefore $\lambda_{cr} \le - \frac{1}{n} \ln p^n (x,x)<\infty$. With this we completed the proof of the first and the second claims. 

To prove the third claim, suppose that  $\nu$ is a QSD with absorption parameter $\lambda$, then  \eqref{eq:QSD_zeta_tail} implies that for any $\epsilon\in (0,1)$, $E_\nu [\exp (\lambda (1-\epsilon) \tau_\Delta)]<\infty$, and therefore $\lambda (1-\epsilon)\le \lambda_{cr}$, which implies $\lambda\le \lambda_{cr}$. 
\end{proof}
As a result, if $\nu$ is a QSD with absorption parameter $\lambda$ then  
$$E_\nu [\tau_\Delta ] \overset{\eqref{eq:QSD_zeta_tail}}{=} \frac{1}{1-e^{-\lambda}}\overset{\mbox{Cor. }\ref{cor:minimal}}{\ge} \frac{1}{1-e^{-\lambda_{cr}}}.$$
For this reason, a QSD with absorption parameter $\lambda_{cr}$ is called a minimal QSD: it minimizes the expected absorption time among all initial distributions that are QSDs. 

The importance of the minimal QSD is apparent from the following simple observation identifying minimal QSDs as Yaglom limits.  
\tro{R1:P2}\begin{prop}
\label{prop:limit}
Let {\bf HD-\ref{as:reg:moments},\ref{as:reg:Sirr}} hold. 
Suppose that for some $x\in S$ and for each $j\in S$, $\nu(j)= \lim_{n\to\infty} P_x (X_n =j| \tau_\Delta > n)$ exists, and $\nu$ is a probability measure on $S$. Then $\nu$ is  a minimal QSD. 
\end{prop}

\begin{proof}
\tro{R1:M2}
 As $\nu$ is quasi-limiting, it is a QSD with some absorption parameter $\lambda\in (0,\infty)$ and therefore $P_\nu (\tau > n)= e^{-\lambda n}$.  For $n\in\N$, let  $c_n = P_x (\tau_\Delta>n)$. We have 
\begin{align*} \frac{c_{n+m}}{c_n} &= \frac{P_x(\tau_\Delta > n+m )}{P_x (\tau_\Delta >n)}\\
&=\frac{E_x [ P_{X(n)}(\tau_\Delta >m),\tau_\Delta>n]}{P_x(\tau_\Delta > n)}\\
& =E_i [ P_{X(n)}(\tau_\Delta >m)|\tau_\Delta>n].
\end{align*} 
The last expression is of the form $E_x [ f(X_n) | \tau_\Delta>n]$, where $f$ is the bounded function $f(k) = P_k (\tau_\Delta>m)$. As pointwise convergence of a sequence of probability measures on (the countable set) $S$ to a probability measure on $S$ is equivalent to weak convergence, and since $f$ is bounded, it follows from the assumption that $\lim_{n\to\infty} E_x [f(X_n)| \tau_\Delta>n] =\sum_{k}\nu(k) f(k) = P_\nu (\tau_\Delta>m)=e^{-\lambda m}$. That is, 
$$ \lim_{n\to\infty} \frac{c_{n+m}}{c_n}= e^{-\lambda m}.$$
By taking $m=1$, it follows from the ratio test that the radius of convergence  of the power series $\sum_{n=1}^\infty r^n c_n$, is equal to $e^{\lambda}$. Plugging this into the summation by parts formula \eqref{eq:summation} gives $\lambda =\lambda_{cr}$.  
\end{proof}
We close the section with one more result on exponential moments of return and absorption times which is key to the analysis of the infinite MGF regime. This result is slightly weaker than \cite[Theorem C]{verejones}, which \tro{R1:G2} can be proved in the same manner. \tro{R1:P3}
 \begin{prop}
\label{prop:weird_values}
Let {\bf HD-\ref{as:reg:moments},\ref{as:reg:Sirr}} hold and let  $\lambda\in (0, \lambda_{cr}]$. Then for every $x\in S$,  
\begin{equation} 
\label{eq:theone}
E_x [\exp (\lambda \tau_x),\tau_x<\tau_\Delta]\le 1
\end{equation} Moreover,
\begin{enumerate} 
\item If $E_x [ \exp (\lambda \tau_\Delta) ] < \infty$ then the  inequality is strict; 
\item If $E_x [\exp(\lambda_{cr} \tau_\Delta)]=\infty$ and $E_x[\exp(\lambda_{cr} \tau_\Delta),\tau_\Delta<\tau_x]<\infty$ for some $x\in S$, then $E_x [\exp(\lambda_{cr} \tau_x),\tau_x<\tau_\Delta]=1$ for all $x\in S$. 
\end{enumerate}
\end{prop} 
\begin{proof}
Pick $\lambda<\lambda_{cr}$. Then  $E_x [\exp (\lambda \tau_\Delta)]<\infty$. \tro{R1:P4}The strong Markov property gives
$$ E_x [ \exp (\lambda \tau_\Delta)]= E_x [\exp (\lambda \tau_x),\tau_x <\tau_\Delta]E_x [\exp (\lambda \tau_\Delta)]+ E_x [\exp (\lambda \tau_\Delta),\tau_\Delta<\tau_x].$$ 
As the left-hand side is finite and the second term on the right-hand side is strictly larger than zero, the first statement holds, with all terms on the right-hand side finite. Moreover, we can write 
\begin{equation} 
\label{eq:alphazeta_broken}
E_x [\exp (\lambda \tau_\Delta) ]= \frac{E_x[\exp (\lambda \tau_\Delta),\tau_\Delta<\tau_x]}{1-E_x[\exp (\lambda \tau_x),\tau_x<\tau_\Delta]}.
\end{equation}
\tro{R1:P5}  Both assertions now follow from the monotone convergence theorem by letting $\lambda \nearrow \lambda_{cr}$. 
\end{proof}

\section{Results: Discrete-Time}
\label{sec:disc_time}
\tro{R2:M4}
In this section we present our results in the discrete-time setting.
\begin{itemize} 
\item In Section \ref{sec:results_infi_disc} we present the results on QSDs in the infinite MGF regime. The main result of this section is the characterization of QSDs in the infinite MGF regime, Theorem \ref{th:nu_recurr}. The proofs of the results in this section are given in Section \ref{sec:pf_infi_disc}.
\item In section \ref{sec:Rrecurr}  we discuss the connection between the infinite MGF regime and the notion of $R$-recurrence, often used in analysis of QSDs.
\item In Section \ref{sec:results_finite_disc} we present the results in the finite MGF regime. The main results are the characterization of QSDs in this regime through a Martin boundary constructed to represent QSDs,  Theorem \ref{thm:martin} and conditions on existence / non-existence of QSDs in this regime, Theorem \ref{thm:QSD_tightness}.  The proofs are given in Sections \ref{sec:pf_finite_1} and \ref{sec:pf_finite_2}. 
\item Finally, in Section \ref{sec:aux} we collect  several results that are not specific to the finite or infinite MGF regime and are of indepedent interst. 
\end{itemize}
Throughout Section \ref{sec:disc_time} we assume without referenece that  hypotheses {\bf HD-\ref{as:reg:moments},\ref{as:reg:Sirr}} hold. 
\subsection{Infinite MGF Regime}
\label{sec:results_infi_disc}
The main result of this section the following necessary and sufficient condition for the existence and uniqueness of a minimal QSD in the infinite MGF regime. 
\begin{thm}
\label{th:nu_recurr}
Suppose $\lambda_{cr}$ is in the infinite MGF regime. Then 
\begin{enumerate}
    \item There exists a minimal QSD if and only if for some $x\in S$\tro{R1:P11}\tro{R1:P2}
\begin{equation} 
\label{eq:finite_stopped} 
E_x[\exp(\lambda_{cr}(\tau_\Delta \wedge \tau_x))]<\infty.
\end{equation}
In this case, the minimal QSD is unique and is given by the formula \tro{R1:P13}
\begin{equation}
\label{eq:nu_recurr}
 \nu_{cr}(x) = \frac{e^{\lambda_{cr}}-1}{E_x[\exp(\lambda_{cr} \tau_\Delta),\tau_\Delta<\tau_x]},~x\in S. 
\end{equation} 
\item If, in addition to \eqref{eq:finite_stopped}, 
 \begin{equation} 
 \label{eq:positive_recurrent} 
 E_x [ \exp (\lambda_{cr} \tau_x)\tau_x,\tau_x < \tau_\Delta]<\infty\mbox{ for some }x\in S,
 \end{equation}
 and $p$ is aperiodic, then $\nu_{cr}$ is the Quasi-Limiting Distribution \eqref{eq:QLD} for every finitely supported initial distribution. 
\end{enumerate}
\end{thm}
Some comments are in place: 
\begin{enumerate}
    \item Proposition \ref{prop:zeta_taux} lists several sufficient conditions for \eqref{eq:finite_stopped}. One such notable case is when  $\{x\in S:p(x,\Delta)>0\}$ is a finite set.
    \item The formula \eqref{eq:nu_recurr} for $\nu_{cr}$ is analogous to the formula for the stationary measure for a positive recurrent Markov chain, expressing it as the reciprocal of the expected return time. To see the connection recall that under the given assumptions Proposition \ref{prop:weird_values}-3 gives
    \begin{equation}
    \label{eq:nu_recurr_alt}\nu_{cr}(x) = \frac{ e^{\lambda_{cr}}-1}{E_x [ \exp (\lambda_{cr} ( \tau_\Delta \wedge \tau_x))]-1 },
    \end{equation}
    and the heuristic limit along a sequence of transition functions converging to a positive recurrent transition function on $S$ with $\lambda_{cr} \searrow 1$ is $\frac{1}{E_x [ \tau_x ] }$.   Proposition \ref{prop:mu_x}-1 provides an alternative representation for $\nu_{cr}$. 
    \item When $S$ is finite:  $\lambda_{cr}$ is in the infinite MGF regime and there exists a unique  QSD which is also minimal, see   Proposition \ref{prop:Sfinite} and the discussion that follows.  This particular case does not require the apparatus developed here and can be obtained directly from the Perron-Frobenius theorem. For example,  the fact that the Perron root is the unique eigenvalue with magnitude equal to the spectral radius and a simple eigenvalue  implies that $\lambda_{cr}$ is in the infinite MGF regime.   Nevertheless, this special case should be viewed as a motivating example for the infinite MGF regime and for Theorem \ref{thm:coming_infinty} below. 
    \item In light of the last comment,  Theorem \ref{th:nu_recurr} can be viewed as a countable state-space version of Perron-Frobenius, tailored to QSDs. The classical extension of Perron-Frobenius is through $R$-recurrence, a notion we expand on in Section \ref{sec:Rrecurr} below.  The recent paper \cite{glynn2018PF} is a generalization of Perron-Frobenius in a way that is more aligned with our work: 
    \begin{enumerate} 
    \item It proves the existence of both a right- and left- (not necessarily normalizable) eigenvectors corresponding to the eigenvalue $e^{-\lambda_{cr}}$ under a basic assumption: equality in \eqref{eq:theone} (by Proposition \ref{prop:weird_values} this implies $\lambda_{cr}$ is in the infinite MGF regime).
    \item It lists \eqref{eq:positive_recurrent} as a sufficient condition for uniqueness up to constant multiples (we only use this condition for convergence). 
    \item The expressions we provide in Proposition \ref{prop:mu_x} and from which we obtain \eqref{eq:nu_recurr} are essentially the same as those for the right- and left- eigenvectors presented in the paper (though we do not work under the key assumption in the paper). 
    \end{enumerate}
    However, the paper does not study QSDs and does not provide any conditions under which the left-eigenvector is normalizable, whereas Theorem \ref{th:nu_recurr} identifies \eqref{eq:finite_stopped} as necessary and sufficient.
\end{enumerate} 
The next result is a condition equivalent to  \eqref{eq:finite_stopped} which may be easier to verify directly in some cases. For $K\subsetneq S$, define the hitting time 
\begin{equation}
    \label{eq:hitK}
    \tau_K=\inf\{n\in \N:X_n\in K\}.
\end{equation} 
\begin{prop}
\label{prop:K2single}
Suppose that $K\subsetneq S$ is nonempty and finite and that for some $x \not \in K$,
\begin{equation}
    \label{eq:getK}
 E_x [ \exp (\lambda_{cr} ( \tau_\Delta\wedge \tau_K))] < \infty.
\end{equation}
Then there exists $z\in K$ so that 
$$E_{z} [ \exp (\lambda_{cr}  (\tau_\Delta \wedge \tau_z))]<\infty.$$
\end{prop}

\tro{R1:P15} The next result is a discrete-time version of the main result in \cite{Villemonais2014}, descrbing the case of  ``small-world'' chains, also described by the authors as processes coming fast from infinity (though the results are applicable to finite state spaces). See Theorem \ref{thm:coming_infinty_cts}, our continuous-time version of that result, followed by a discussion on the differences between our approach and the approach by the authors of the cited work.    
\begin{thm} \tro{R2:M11}
\label{thm:coming_infinty}
Suppose that there exists some $\bar \lambda>0$ and a nonempty finite $K\subsetneq S$ such that both following conditions hold:  
\begin{align} 
\label{eq:blowlambda0}
&E_x [ \exp (\bar\lambda \tau_\Delta)]=\infty\mbox{ for some }x \in S;\\
\label{eq:bd_arrival} 
&\sup_{x\not\in K} E_x [\exp (\bar\lambda (\tau_\Delta\wedge \tau_K))]<\infty. 
\end{align}
Then 
\begin{enumerate} 
\item  $\lambda_{cr}\in (0,\bar \lambda]$,  $\lambda_{cr}$  is in the infinite regime, and \eqref{eq:finite_stopped}, \eqref{eq:positive_recurrent}  and \eqref{eq:getK} hold. 
\item There exists a unique QSD. This QSD is minimal and is given by \eqref{eq:nu_recurr}. 
\end{enumerate} 
If, in addition,  $p$ is aperiodic and there exists some $x_0 \in S$ and  $n_0\in\N$ such that 
\begin{equation}
\label{eq:lowerh}
\inf_{x\in S} \frac{P_x (\tau_\Delta>\tau_{x_0})}{P_x (\tau_\Delta >n_0)}>0,
\end{equation}
then \eqref{eq:QLD} holds for any initial distribution $\mu$.  
\end{thm}
\tro{R1:P16}When $S$ is finite the existence of a unique QSD which is also minimal as well as convergence are obtained through the Perron-Frobenius theorem, as was first studied in detail in \cite{DoSe1965dis}\tro{R2:M5}. Thus, the next result simply records the fact that the finite state-space case is basically a very special case of the infinite MGF regime. 
\begin{prop}
\label{prop:Sfinite}
    Suppose $S$ is finite and has at least two elements.\tro{R1:P17}  Then \eqref{eq:blowlambda0},  \eqref{eq:bd_arrival} and  \eqref{eq:lowerh} hold with $\bar \lambda = \lambda_{cr}$ and $K=S-\{x_0\}$, where $x_0\in S$ is any element maximizing $S\ni x\to p(x,\Delta)$. 
\end{prop}
Thus in the finite state case,   \eqref{eq:nu_recurr} is a stochastic representation to the left-Perron eigenvector. Recent work do exactly the same, but with a different flavor:    \cite{cerf1} gives a representation for this eigenvector through the auxiliary process obtained by normalizing the rows of $p$, and \cite{cerf2} provides a representation through a multi-type branching process. All three representations are equivalent and can be obtained from each other through a routine change of measure. For example, the probability of a path $i_0,i_1,\dots, i_n\in S$ under the auxiliary row-normalized process is equal to $\prod_{j=0}^{n-1} p(i_j,i_{j+1}) / \prod_{j=0}^{n-1} f(i_j)$, where $f(i) = \sum_{j\in S} P(i,j)$, an identity which leads to identifying the first formula on the second page of \cite{cerf1} as equivalent to \eqref{eq:nu_recurr}. 
\subsection{Connection with $R$-recurrence}
\label{sec:Rrecurr}
\tro{R2:M3}\tro{R1:P1} \tro{R1:P7} A prominent classical approach to obtaining existence, uniquness and convergence is  through both the standard and infinite dimensional versions of the Perron-Frobenius theorem  \cite{verejones} for $R$-positive Markov chains \cite[Lemma 1]{Seneta1966} with additional conditions that imply that the ``left" eigenvector is normalizable, or, equivalently, the existence of a finite $1/R$-invariant measure. We will now briefly review this approach and establish connections with the infinite MGF regime. 

Let $R\ge 1$. The restriction of  $p$ to $S$  is called $R$-recurrent if and only if \begin{equation} 
\label{eq:power_series} 
\sum_{n=1}^\infty R^n p^n (i,i)=\infty
\end{equation} 
for some (equivalently all) $i\in S$. Otherwise, $p$ is called $R$-transient.   If $p$ is $R$-recurrent, then it called $R$-positive if for some (equivalently all) $i$, $\limsup_{n\to\infty} R^n p^n(i,i) > 0$. It is $R$-null otherwise.  The radius of convergence of the power series \eqref{eq:power_series} is also known in the literature as the convergence parameter for the restriction of $p$,  which we denote here by $R_{cr}$.  Clearly, $p$ is $R$-transient for all $R<R_{cr}$, and $p$ may be $R_{cr}$-transient or $R_{cr}$-null or $R_{cr}$-positive. 

Suppose that $R\ge 1$ and $\lambda\ge 0$ is given by $e^{\lambda}=R$. If $p$ is $R$-recurrent then  (e.g. Corollary \ref{cor:minimal}-2 below), $E_x [\exp (\lambda \tau_\Delta)]=\infty$ for all $x\in S$. As a result, $\lambda \ge \lambda_{cr}$, and in particular, 
\begin{equation}
 \label{eq:R_recurr}
 e^{\lambda_{cr}}\le R_{cr}.
\end{equation} 
\cite[Theorem 2]{Coolen} is the combination of this inequality and Corollary \ref{cor:minimal}-3.  The inequality may be strict. Here is a simple example: 
\begin{xmpl}
\label{xmple:R_cr}
Consider Example \ref{xmpl:ind_zeta}, with $S=\Z$, $\mathfrak p(x,x\pm 1)= \frac{1\pm \epsilon}{2}$ for some $\epsilon\in (0,1)$. Note that in this case $\lambda_{cr}$ is predetermined, and that since $\mathfrak p$ is transient,  there are no QSDs for $p$. We will now show that 
$$R_{cr}= \frac{e^{\lambda_{cr}}}{\sqrt{1-\epsilon^2}}>e^{\lambda_{cr}}.$$

From the construction, for every $i,j\in \Z$, $p(i,j) = e^{-\lambda_{cr}}\mathfrak p (i,j)$, and therefore 
$$ \sum_{n=1}^n R^n p^n (i,i) = \sum_{n=1}^\infty R^n e^{-\lambda_{cr} n} \mathfrak p^n (i,i).$$
Let $\mathfrak p_0$ denote the transition function on $\Z$ corresponding to the simple symmetric random walk (that is $\epsilon=0$). Then $\mathfrak p^n(i,i) = (1-\epsilon^2)^{n/2}\mathfrak p_0^n (i,i)$,
which leads to 
$$\sum_{n=1}^n R^n p^n (i,i) = \sum_{n=1}^\infty e^{-\lambda_{cr} n} R^n  (1-\epsilon^2)^{n/2} \mathfrak p_0^n (i,i).$$
As $\mathfrak p_0$ is recurrent, the righthand side is convergent if and only if $e^{-\lambda_{cr} } R \sqrt{1-\epsilon^2}<1$, and the result follows. 
\end{xmpl}

We now discuss sufficient conditions for equality in  \eqref{eq:R_recurr}. We begin with a trivial one:  If $p$ has a QSD with absorption parameter $\ln R$ for some $R>0$, Corollarly \ref{cor:minimal}-3 implies $R\le e^{\lambda_{cr}}$. In particular, the existence of a QSD with absorption paramter $R_{cr}$ implies $R_{cr}=e^{\lambda_{cr}}$.

For additional conditions we first obtain a stochastic representation to the power series in \eqref{eq:power_series}. Pick $\lambda>0$ so that $e^{\lambda} < R_{cr}$. The strong Markov property gives  
$$ \sum_{n=0}^\infty e^{\lambda n} p^n (i,i) = 1+ E_i [ \exp (\lambda \tau_i), \tau_i < \tau_\Delta]  \sum_{n=0}^\infty e^{\lambda n}p^n(i,i),$$ 
that is  $\sum_{n=0}^\infty e^{\lambda n} p^n (i,i) = \frac{1}{1-E_i [ \exp (\lambda \tau_i),\tau_i < \tau_\Delta]}$ (note that this also leads to the identification of  $R_{cr}$ as the infimum over $R>1$ satsifying $E_i [ R^{\tau_i},\tau_i < \tau_\Delta]\ge 1$).   Through monotone convergence the identity extends to $\lambda = \lambda_{cr}$, with  the convention $\frac10=\infty$:  
\begin{equation} 
\label{eq:visit_times}
\sum_{n=0}^\infty e^{\lambda_{cr}  n} p^n (i,i)=\frac{1}{1-E_i [ \exp (\lambda_{cr} \tau_i),\tau_i < \tau_\Delta]}.
\end{equation}
Suppose now that $\lambda_{cr}$ is in the infinite MGF regime and that $p$ has a minimal QSD. Then \eqref{eq:finite_stopped} holds and as a result of Proposition \ref{prop:weird_values}-3,  the quantity in \eqref{eq:visit_times} is infinite, which in turn implies that $e^{\lambda_{cr}} \ge  R_{cr}$, and therefore an equality holds in \eqref{eq:R_recurr}. We proved: 
\begin{cor}
The following are equivalent: 
\begin{itemize} 
\item $p$ is $R_{cr}$-recurrent and there exists a QSD with absorption parameter $\ln R_{cr}$. 
\item $\lambda_{cr}$ is in the infinite MGF regime and a minimal QSD exists. 
\end{itemize} 
Under these conditions,  $e^{\lambda_{cr}} = R_{cr}$. 
\end{cor}
Here are two additional sufficient conditions for equality in \eqref{eq:R_recurr}: 
\begin{enumerate} 
\item \tro{R1:P8}If an equality holds in \eqref{eq:theone}  with $\lambda=\lambda_{cr}$, then $e^{\lambda_{cr}}=R_{cr}$, since \cite[Theorem C]{verejones} states that $E_x [R_{cr}^{\tau_x},\tau_x<\tau_\Delta] \le 1$ for all $x$.
\item If the  conditions of Proposition \ref{prop:limit} hold, then there exists $c>0$ such that $p^{n}(i,i) > c P_i(\tau_\Delta > n)$ for all $n$ large enough. In particular,   $R_{cr}$ is larger or equal to the radius of convergence of the series $\sum_{n=1}^\infty r^n P_i(\tau_\Delta>n)$. As Corollarly \ref{cor:minimal}-2 gives that the  latter is $e^{\lambda_{cr}}$, the equality follows.    
\end{enumerate} 

Theorem \cite[Theorem 3.1]{Seneta1966} states that under irreducibility, {\bf HD-\ref{as:reg:Sirr}},  and \tro{R1:P9}aperiodicity of $p$ restricted to $S$,  the limit \eqref{eq:QLD} (as  well as several additional limits) holds for any initial distribution $\mu$ which is a point mass  if and only if: (i) $p$ is $R_{cr}$-positive; and (ii) the left positive eigenvector for $p$ corresponding to the eigenvalue $\frac{1}{R_{cr}}$ is $\ell^1$-normalizable. In this case, the limit is a QSD with absorption parameter $\ln R_{cr}$, and as noted above, we have $R_{cr}=e^{\lambda_{cr}}$. Several comments are in place: 
\begin{enumerate} 
\item The result is necessary and sufficient for convergence of these Yaglom limits, not for convergence from general initial distributions or for existence. 
\item Condition (ii) is not automatic.
\item The assumption of $R_{cr}$-positivity is stronger than $R_{cr}$-recurrence, in fact, the additional condition \eqref{eq:positive_recurrent}  in Theorem \ref{th:nu_recurr} (which we only use for convergence) is equivalent to $p$ being $R_{cr}$-positive, \cite[Lemma 1]{Seneta1966}.  
\item Finally, $R_{cr}$-recurrence is not even necessary for the existence or uniqueness of QSDs. For examples in Branching Processes and One-Sided Random walk, see \cite[Sections 5,6]{Seneta1966}, respectively. We will revisit both models using the techniques we develop here in later sections.
\end{enumerate} 
Another work that provides sufficient conditions for $p$ to be $R_{cr}$-positive, gven in terms of tails of hitting times is  
\tro{R1:P9} \tro{R1:P10}  \cite{FerrariPR}.  \tro{R1:M3} 

In summary, although the existence of a QSD in the infinite MGF regime is closely related to $R_{cr}$-recurrence, we think that our framing of parameter regimes and the corresponding necessary and sufficient conditions for existence of QSD in each, specifically those in Theorem \ref{th:nu_recurr} for the infinite MGF regime, are simple and more natural than the conditions presented in the literature in the context of $R$-recurrence. 
\subsection{Finite MGF Regime}
\label{sec:results_finite_disc}
The finite state space case is completely settled by Proposition \ref{prop:Sfinite}: there is a unique QSD which is minimal.  As a result, in this section we will work under hypotheses  {\bf HD-1',2} and will also include 
\setcounter{assum}{-1}
\begin{assum}
\label{as:inf}
$S$ is countably infinite. 
\end{assum}
As much of the discussion will involve limits at infinity of functions on $S$, we recall the relevant notions. The extended limit of a function  $f:S\to \R$ at ``infinity", denoted by $\lim_{x\to\infty} f(x)$:
\begin{itemize}
    \item Is a real number $L$ if  for every $\epsilon>0$ the set $\{x:\in S:|f(x) - L|\ge \epsilon\}$ is finite. 
    \item Is $\infty$ if for every $M>0$, the set $\{x\in S:f(x)<M\}$ is finite. 
    \item If is $-\infty$ if for every $M>0$ the set $\{x\in S: f(x)>-M\}$ is finite. 
    \item Otherwise, the limit is undefined. 
\end{itemize}
\subsubsection{Martin Boundary Representation}
\label{sec:results_finite_disc_martin}
Classically, Martin (exit) Boundary theory provides a compactification of the state space of a transient Markov Chain through a set of positive (super-)harmonic functions. These functions describe the tail of the chain:  under the new topology, the chain converges almost surely, with the limit viewed as the location where the process ``exits'' the state space. On the flip side,   the construction allows to represent all harmonic functions as integrals over elements of the boundary. The books by Woess \cite{Wolfgang} and by Kemeny, Snell, and Knapp \cite{KemenyMartin} cover (exit) Martin boundary theory. 

Due to a duality connection between harmonic functions and invariant measures, analogous compactifications and results exist for invariant measaures. The intuitive explanation of the compactification is as adding states to ``enter" from infinity. The corresponding constructions are referred to as Martin entrance boundaries. As QSDs are $e^{-\lambda}$-invariant  probability measures , Martin entrance boundary are highly relevant for classifying them.  

In our work, we adapt the ideas for the Martin exit boundary as presented in \cite{Sawyer1997} to obtain a  Martin entrance boundary compactifications of the state space which result in classification of QSDs in the finite MGF regime. One difference from previous constructions is that our is specifically tailored to QSDs:  it only has finite (sub)probability measures as its elements, eliminating all infinite measures from consideration,  reminiscent of the Poisson boundary representation of bounded harmonic functions. Moreover, our construction is directly  obtained from the Green's function, not through duality.  

Martin entrance boundaries were introduced decades ago, e.g. \cite[Chapter 10, Section 10]{KemenyMartin}.  Yet, despite having all necessary ingredients,  the reference to Martin boundary in the context of characterizing QSDs is sparse. The paper \cite{entrance_GW}  describes all stationary measures for the subcritical branching process, namely positive solutions to  $\nu p = \nu$ ($\lambda=0$). As all such solutions are necessarily infinite measures (the process is transient and as such does not possess a finite invariant measure), this is outside the realm of QSDs.  The paper  \cite{maillard} includes a complete description of all $e^{-\lambda}$-invariant measures for subcritical branching processes, both QSDs and infinite invariant measures. The paper also draws the connection with Martin entrance boundaries  and provides a general blueprint for using this for describing all invariant measures for a process, but doesn't execute it or provides the tools to identify QSDs, which are necessarily finite measures. The paper \cite{foley_mcdonald} utilizes Martin boundary calculations to describe QSDs for processes which are special cases of the processes we discuss in Section \ref{sec:B_D}. The calculations are done through duality with the Martin exit boundary. It should be noted that the main result on the paper goes in a different direction: for these processes Yaglom limits may depend on the starting point. 

Martin boundaries are constructed as limits of normalized Green's functions. We thererfore begin with the definition of the Green's functions.  \tro{R1:P18}
\begin{defn}
\label{def:KMartin}
    Let $\lambda>0$ be in the finite MGF regime. For $x\in S$, define  
    \begin{enumerate}
        \item The Green's function\tro{R1:P20}
\begin{equation} 
\label{eq:Green_def} 
\begin{split}
G^\lambda(x,y) &= \delta_x(y) + \sum_{n=1}^\infty (e^{\lambda} p)^n(x,y)\\
&= E_x [ \sum_{0\le s<\tau_\Delta }\exp (\lambda s) \delta_y(X_s)]
\\ & =
  \frac{E_x [\exp (\lambda  ^0\tau_y), ^0\tau_y <{\tau_\Delta}]}{1-E_y [\exp (\lambda \tau_y),\tau_y<\tau_\Delta]}. 
\end{split}
\end{equation}
Then for fixed $x\in S$,  $G^\lambda(x,\cdot)$ is a finite measure on $S$ with total mass 
\begin{equation} 
\label{eq:Green_mass}
G^\lambda (x,{\bf 1})= \frac{E_x [\exp (\lambda \tau_\Delta)]-1}{ e^{\lambda}-1}.
\end{equation}
\item The normalized kernel $K^\lambda(x,\cdot)$, a probability measure on $S$ in the second variable,
\begin{equation}
\label{eq:Kalpha} 
K^\lambda(x,y) = \frac{ G^\lambda(x,y)}{G^\lambda(x,{\bf 1})}= \frac{e^\lambda-1}{E_x [\exp (\lambda \tau_\Delta)]-1}\times \frac{E_x [\exp (\lambda ^0\tau_y),^0\tau_y <{\tau_\Delta}]}{1-E_y [\exp (\lambda 
 \tau_y),\tau_y<\tau_\Delta]}.  
\end{equation}
\item A sequence ${\bf x} =(x_n:n\in \N)$ of elements in $S$ is $\lambda$-convergent if for every $y\in S$, $\lim_{n\to\infty} K^\lambda (x_n,y)$ exists. If ${\bf x}$ is $\lambda$-convergent, we denote the limit (probability or sub-probability) measure by  $K^\lambda ({\bf x},\cdot)$.
\item A sequence ${\bf x}$ is $\lambda,\infty$-convergent if it is $\lambda$-convergent and $\lim_{n\to\infty} x_n=\infty$.
    \end{enumerate}
    \end{defn}
To explain the last line of \eqref{eq:Green_def}, use the strog Markov property and condition on the process at time  $\tau_y\wedge \tau_\Delta$ to obtain  
$$ G^\lambda(x,y) = \delta_{y}(x) + E_x [\exp (\lambda (\tau_y \wedge \tau_\Delta)),\tau_y < \tau_\Delta]G^\lambda(y,y).$$
Therefore $G^\lambda(y,y) = \frac{1}{1-E_y [ \exp (\lambda (\tau_y \wedge \tau_\Delta)),\tau_y < \tau_\Delta]}$. Plugging this into the last equation gives the result. To obtain \eqref{eq:Green_mass}, observe from the second line of \eqref{eq:Green_def} that $G^\lambda(x,{\bf 1}) =E_x [\sum_{0\le s < \tau_\Delta}  \exp (\lambda s)]$, and apply the Geometric series formula. In light of these, \eqref{eq:Kalpha} follows. 

A key feature of our approach and what distinguishes it from other constructions is the normalization of the Martin Kernels as probability measures \eqref{eq:Kalpha} which will eventually lead to a boundary obtained by (sub)probability measures. 

\tro{R1:P21}Any bounded sequence $(x_n:n\in\N)$ of elements of $S$ has a constant subsequence which is  trivially  $\lambda$-convergent.  On the other side of things, if $(x_n:n\in\N)$ is an unbounded sequence of elements in $S$, then the fact that  $K^\lambda (\cdot,\cdot) \in (0,1]$ and a standard diagonalization argument give that  $(x_n:n\in\N)$ has a $\lambda,\infty$-convergent subsequence.   \tro{R1:P23} In order to define the Martin compactification, let $q:S \to \N$ be a bijective function. The function $q$ is used to define a metric for our compactification but the resulting topological space is independent of the choice of $q$. \tro{R2:M6}
\begin{defn}[Martin Compactification]
Let $\lambda>0$ be in the finite MGF regime. 
\label{Def:Martin}
\quad
    \begin{enumerate}
        \item The $\lambda,\infty$-convergent sequences ${\bf x}$ and ${\bf x}'$ are $\lambda$-equivalent if $K^\lambda({\bf x},\cdot)= K^\lambda({\bf x}',\cdot)$.  Write  $[{\bf x}]$ for the equivalence class of ${\bf x}$ and  $K^\lambda([{\bf x}],\cdot)$ for $K^\lambda({\bf x},\cdot)$.
        \item The Martin Boundary $\partial^\lambda M$ is the set of equivalence classes of $\lambda,\infty$-convergent sequences.
        \item \label{eq:Martin_metric} Define the metric $\rho^\lambda$  on $M^\lambda = S \cup \partial^\lambda M$ as follows: 
        $$\rho^\lambda(a,b) = \sum_{s\in S} \frac{1}{2^{q_s}} \bigl( |\delta_{a,s}-\delta_{b,s}|+  d(K^\lambda(a,s),  K^\lambda(b,s))\bigr),$$ 
        where $d(i,j) =\frac{ |i-j|}{ 1+|i-j|}.$
        \item Let  
        $$S^\lambda = \{[{\bf x}]\in \partial^\lambda M:  K^\lambda([{\bf x}],\cdot) \mbox{ is a QSD with absorption parameter }\lambda\}.$$
    \end{enumerate}
\end{defn}
We note that the reason for the $\delta$-terms in the definition of $\rho^\lambda$ \eqref{eq:Martin_metric} is because $K^\lambda([x],\cdot)$ may be equal to $K^\lambda(a,\cdot)$ for some $a\in S$.
\begin{prop}
\label{prop:Compact_Martin}
Let $M^\lambda$ and $\rho^\lambda$ be as in  Definition \ref{Def:Martin}. Then $(M^\lambda, \rho^\lambda)$ is a compact metric space.
\end{prop}
The next theorem is a Choquet-type result, stating that the metric space introduced above characterizes all QSDs through the ways the process may ``come from infinity''.\tro{R1:P22}
\begin{thm}
\label{thm:martin}
Let $\lambda>0$ be in the finite MGF regime.  Then, there exists a QSD with absorption parameter $\lambda$ if and only if $S^\lambda$ is not empty. In this case,  $\mu$ is a QSD with absorption parameter $\lambda$ if and only if there exists a Borel probability measure ${\bar F}_\mu$ on $M^\lambda$  satisfying ${\bar F}(S^\lambda)=1$ and 
$$ \mu (y) = \int K^\lambda([{\bf x}],y) d{\bar F}_{\mu}([{\bf x}]),~y\in S.$$ 
\end{thm}
\tro{R1:P24}Theorems \ref{th:nu_recurr} and  \ref{thm:martin} provide a complete description of all QSDs for a given Markov chain. 
\subsubsection{Conditions for Existence of QSDs}
\label{sec:results_finite_disc_regular}
In this section, we focus on the behavior of the generating function at infinity to show sufficient conditions for the existence and the non-existence of a QSD. 
\begin{thm} \label{thm:QSD_tightness} 
Let $\lambda>0$ be in the finite MGF regime. Then 
\begin{enumerate} 
\item If there exists  $\lambda' \in (0,\lambda)$ satisfying   $\lim_{x\to \infty}E_x [ \exp (\lambda' \tau_\Delta)]=\infty$ then for every $\lambda,\infty$-convergent sequence ${\bf x}$, $K^\lambda([{\bf x}],\cdot)$ is a QSD with absorption parameter $\lambda$. 
\item If $\sup_x  E_x [ \exp (\lambda \tau_\Delta)] < \infty$, then there are no QSDs with absorption parameter $\lambda$. 
\end{enumerate} 
\end{thm}
\begin{cor}
\label{cor:upto_cr}
Let 
$$\lambda_ 0 = \inf\{ \lambda \in (0,\lambda_{cr}): \lim_{x\to\infty} E_x [ \exp (\lambda \tau_\Delta)]=\infty\},$$
with the convention $\inf \emptyset = \infty$. Then for every $\lambda \in (\lambda_0,\lambda_{cr}]$ there exists a QSD with absorption parameter $\lambda$.  
\end{cor}
Note that the corollary yields the existence of a minimal QSD, a result that does not follow directly from the theorem and requires an additional tightness argument. 

\tro{R1:P25}The key ingredient in the proof of the first part of Theorem \ref{thm:QSD_tightness} is an argument from  \cite{Ferrari1995}, showing that under the stated condition, for every sequence $(x_n:n\in\N)$ in the equivalence class $[{\bf x}]$, the sequence of kernels $(K^\lambda(x_n,\cdot):n\in\N)$ is tight.  The second part of the theorem essentially follows directly from  \eqref{eq:QSD_zeta_tail}, the geometric distribution of the absorption time under a QSD.  
\begin{defn}
\label{defn:Clambda}
Let $\lambda>0$ be in the finite MGF regime and for $x,y \in S$ define 
$$ C^\lambda(x,y)=\frac{E_x [\exp(\lambda \tau_\Delta),\tau_y<\tau_\Delta]}{E_x [\exp(\lambda \tau_\Delta)]-1}.$$
\end{defn} 
With additional assumptions on $p$, the existence and uniqueness of QSDs can be obtained through analysis of $C^\lambda$. 
\tro{R2:M8}\begin{prop}
\label{prop:pzy}
Assume that for every $y \in S$, 
\begin{equation}
\label{cond:pzy}
\sum_{z} p(z,y) < \infty.
\end{equation}
 Let $\lambda>0$ be in the finite MGF regime.  Let $(x_n:n\in\N)$ be a sequence with $\lim_{n\to\infty} x_n = \infty$. 
\begin{enumerate} 
\item If $\liminf_{n} C^\lambda(x_n,y)>0$ for some $y\in S$, then there exists a QSD with absorption parameter $\lambda$. 
\item If $\liminf_n C^\lambda(x_n,y)\ge 1$ for all $y\in S$, then there exists a unique QSD with absorption parameter $\lambda$ given by 
$$ \nu(y) = \frac{e^\lambda-1}{E_y [\exp(\lambda\tau_\Delta),\tau_\Delta<\tau_y]}.$$ 
\end{enumerate}
\end{prop}
\begin{cor}
\label{cor:ratios}
Assume that \eqref{cond:pzy} holds. If the set $A=\{x:p(x,\Delta)>0\}$ is finite, then for every $\lambda\in (0,\lambda_{cr})$ there exists a QSD with absorption parameter $\lambda$. 
\end{cor}
We close this section with a necessary and sufficient condition for the existence of a QSD under the condition \eqref{cond:pzy}.  \tro{R1:P26}
\begin{cor}
\label{cor:connection}
Let $\lambda>0$ be in the finite MGF regime and suppose that  \eqref{cond:pzy} holds. Let $[{\bf x}]\in \partial^\lambda M$. Then
the following are equivalent: 
\begin{enumerate} 
\item $K^\lambda ([{\bf x}],\cdot)$ satisfies \eqref{eq:eigen_description} and is not identically zero.
\item There exists $y\in S$ and a sequence $(x_n:n\in\N)\in [{\bf x}]$ such that $\lim_{n\to\infty} C^\lambda(x_n,y)>0$. 
\end{enumerate} 
\end{cor}
Note that under the equivalent conditions in the corollary and the irreducibility of $S$, $K^\lambda ([{\bf x}],\cdot)$ is strictly positive, and $\lim_{n\to\infty} C^\lambda (x_n, y)$ exists and is in $(0,\infty)$ for all $y\in S$ and  $(x_n:n\in\N)\in [{\bf x}]$. Moreover, $K^\lambda ([{\bf x}],\cdot)$ can be normalized to a probability measure which is then necessarily a QSD. 
\subsection{Auxiliary Results}
\label{sec:aux}
\begin{prop}
\label{prop:UI}
 The family of distributions of $e^{\lambda_{cr} \tau_\Delta}$ under $P_x,~x\in S$,  is not uniformly integrable. 
\end{prop}
The proposition has the following immediate corollary utilizing the fact that any finite set of integrable RVs are uniformly integrable and stochastically dominated.  
\begin{cor}
\label{cor:nu_recurr_Sfinite}
\begin{enumerate}
\item Suppose that $S$ is finite. Then $\lambda_{cr}$ is in the infinite regime.  
\item Suppose that there exists a sequence $(x_n:n\in\N)$ of elements in $S$ and that for every $n\in\N$ the distribution of $\tau_\Delta$ under $P_{x_n}$ is stochastically dominated by its distribution under $P_{x_{n+1}}$. If $\lambda_{cr}$ is in the finite regime, $\lim_{n\to\infty} E_{x_n} [\exp (\lambda_{cr} \tau_\Delta)] = \infty$. 
\end{enumerate}
\end{cor}
\begin{proof}[Proof of Proposition \ref{prop:UI}]
We argue by contradiction. Suppose that\\ $\lim_{n\to\infty} \sup_{x\in S}E_x [\exp(\lambda_{cr}\tau_\Delta),\tau_\Delta>n]=0$. Then for every $\epsilon\in (0,1)$, there  exists $n_0$ such that $\sup_x E_x [\exp(\lambda_{cr} \tau_\Delta),\tau_\Delta>n_0]< \epsilon/2$. As a result of the Markov property, it follows that 
$\sup_x E_x [\exp(\lambda_{cr} \tau_\Delta),\tau_\Delta>kn_0]\le (\epsilon/2)^k$. And so,  for every $x\in S$,  

$$\sum_{k=0}^\infty \epsilon^{-k} E_x [\exp(\lambda_{cr} \tau_\Delta),\tau_\Delta>kn_0] < \infty.$$ 
This implies 
$$ \sum_{k=0}^\infty \epsilon^{-(k+1)} E_x [\exp(\lambda_{cr} \tau_\Delta),kn_0  < \tau_\Delta \le (k+1)n_0 ] <\infty$$
But when $\tau_\Delta \leq (k+1)n_0$, $\epsilon^{-(k+1)}\geq \epsilon^{-\tau_\Delta/n_0}$, and so we have 
$$ \sum_{k=0}^{\infty }E_x [(\epsilon^{-1/n_0}e^{\lambda_{cr}})^{\tau_\Delta},kn_0 < \tau_\Delta \le (k+1)n_0 ]<\infty.$$ 
That is,  taking $e^{\widetilde {\lambda_{cr}}}=\epsilon^{-1/n_0} e^{\lambda_{cr}}>e^{\lambda_{cr}}$, we have that  $E_x [e^{\widetilde {\lambda_{cr}} \tau_\Delta}]<\infty$, contradicting the definition of $\lambda_{cr}$. 
\end{proof}
Next, we provide sufficient conditions for \eqref{eq:finite_stopped} to hold. 
\begin{prop}
\label{prop:zeta_taux}
Let $x\in S$ and suppose that at least one of the following conditions hold:
\begin{enumerate} 
\item The probability distributions $r\to P_x(X_r \in \cdot~ | ~\tau_\Delta \wedge \tau_x >r)$ are tight. 
\item There exists $x\in S$ such that $\inf_{y \in S} P_y(\tau_x < \tau_\Delta)>0$.
\end{enumerate} 
Then $\sup_x E_x [\exp(\lambda_{cr} (\tau_\Delta\wedge \tau_x))]<\infty$. 
\end{prop}
Note that if the set $S$ of states $z$ satisfying $p(z,\Delta)>0$ is finite, then the second condition automatically holds. 
\begin{proof}[Proof of Proposition \ref{prop:zeta_taux}]
Let $\lambda \in (0,\lambda_{cr})$.  Summing by parts, for any nonnegative bounded random variable $Z$, we have 
\begin{equation} 
\label{eq:summationZ} 
E_x[\exp(\lambda(\tau_\Delta \wedge \tau_x ))Z ] =e^{\lambda} E_x[Z] + (e^{\lambda}-1)\sum_{r=1}^\infty e^{\lambda r} E_x [Z {\bf 1}_{\{\tau_\Delta \wedge \tau_x > r\}}].
\end{equation} 
By monotone convergence, this holds for any nonnegative random variable $Z$. When taking $Z=1$ we have 
\begin{equation} 
\label{eq:Z1}E_x[\exp(\lambda(\tau_\Delta \wedge \tau_x )) ]  = e^\lambda + (e^\lambda-1)\sum_{r=1}^\infty e^{\lambda r} P_x (\tau_\Delta \wedge \tau_x > r).
\end{equation} 
Now take   $$Z= E_{X(\tau_\Delta \wedge \tau_x)} [\exp(\lambda ^0\tau_\Delta)].$$ The strong Markov property gives that the left-hand side of  \eqref{eq:summationZ} is equal to $E_x [\exp(\lambda\tau_\Delta)]$. As for the right-hand side, from the Markov property, 
$$E_x [ Z,\tau_\Delta \wedge \tau_x > r |{\cal F}_r]= {\bf 1}_{\{\tau_\Delta \wedge \tau_x > r \}} E_{X(r)} [E_{X_{\tau_\Delta\wedge \tau_x}}[ \exp(\lambda ^0\tau_\Delta)].$$
In our case, 
\begin{align*} 
 E_x[ Z,\tau_\Delta \wedge \tau_x >r|{\cal F}_r]&\ge {\bf 1}_{\{\tau_\Delta \wedge \tau_x>r\}} P_{X(r)}(\tau_x < \tau_\Delta) E_x [\exp(\lambda\tau_\Delta)].
\end{align*}

Therefore 
\begin{align*} E_x [ Z ,\tau_\Delta \wedge \tau_x > r] &\ge E_x [\exp(\lambda\tau_\Delta)] E_x [{\bf 1}_{\{\tau_\Delta \wedge \tau_x > r\}} P_{X(r)}(\tau_x <\tau_\Delta)]\\
 & =E_x[\exp(\lambda\tau_\Delta)] E_x [P_{X(r)} (\tau_x <\tau_\Delta) | \tau_\Delta \wedge \tau_x > r] P_x (\tau_\Delta \wedge \tau_x > r).
\end{align*} 

Assuming the first condition, the tightness condition. For every $\epsilon>0$ there exists some finite set $K_\epsilon$ such that%

$P_x(X(r) \in K_\epsilon | \tau_\Delta \wedge \tau_x>r) \ge (1-\epsilon)$. Let $c_2 =c_2(\epsilon) =  \min_{y \in K_\epsilon} P_y (\tau_x <\tau_\Delta)>0$. Therefore, we have that

\begin{equation} 
\label{eq:lower_taux}E_x [P_{X(r)} (\tau_x <\tau_\Delta) | \tau_\Delta \wedge \tau_x > r]\ge c_1,
\end{equation} 
where $c_1=(1-\epsilon)c_2$. If we assume the second condition instead, then we can use the infimum in the condition as $c_1$ in \eqref{eq:lower_taux}. Thus, under either condition, we have 
\begin{align*}  E_x [\exp(\lambda\tau_\Delta)]&\ge e^\lambda E_x[E_{X(\tau_\Delta \wedge \tau_{x})}[\exp(\lambda ^0\tau_\Delta)]] + c_1 E_x [\exp(\lambda\tau_\Delta)](e^\lambda-1)  \sum_{r=1}^\infty e^{\lambda r} P_x(\tau_\Delta \wedge\tau_x>r)\\
& \overset{\eqref{eq:Z1}}{=} e^\lambda E_x[E_{X(\tau_\Delta \wedge \tau_{x})}[\exp(\lambda ^0\tau_\Delta)]] + c_1 E_x [\exp(\lambda\tau_\Delta) ](E_x [\exp(\lambda(\tau_\Delta\wedge \tau_x))]-e^\lambda)\\
& \ge  c_1 E_x [\exp(\lambda\tau_\Delta) ](E_x [\exp(\lambda(\tau_\Delta\wedge \tau_x))]-e^\lambda).
\end{align*} 
Divide both sides by $E_x [\exp(\lambda\tau_\Delta)]$ to obtain the bound 
$$ E_x [\exp(\lambda(\tau_\Delta\wedge \tau_x))]\le e^\lambda + \frac{1}{c_1}.$$
The result now follows from monotone convergence.  
\end{proof}

\section{Proof of the results of Section \ref{sec:results_infi_disc}}
\label{sec:pf_infi_disc}
\tro{R1:P27}
\subsection{Potential Theoretic Results} 
Despite some minor differences, the proof and the formulas as essentially the same as those \cite[Theorem 1]{glynn2018PF}. 
\begin{prop}
\label{prop:mu_x}
Let $\lambda\in (0,\lambda_{cr}]$.  For $z\in S$ define the function $h_z : S \to [0,\infty)$ by letting 
\begin{equation} 
\label{eq:harmonic_function} 
h_z(x) = E_x [ \exp (\lambda \tau_z),\tau_z < \tau_\Delta].
\end{equation}
Then 
\begin{enumerate} 
\item For $x\in S$ define the measure $\mu_x$ on $S$ through
\begin{equation}
 \label{eq:mux_def}
 \begin{split} 
\mu_x (y) &= E_x [\sum_{s<\tau_x \wedge \tau_\Delta} \exp (\lambda s) \delta_y (X_s)]\\
 & =  \frac{E_x [\exp (\lambda  \tau_y),\tau_y < (\tau_x \wedge  \tau_\Delta)]}{1-E_y[\exp (\lambda \tau_y),\tau_y<(\tau_x \wedge \tau_\Delta)]},~y\in S
 \end{split}
\end{equation}
Then 
\begin{equation}
\label{eq:mux_nearlyQSD}
(\mu_x p) (z) =  e^{-\lambda}\mu_x (z) + e^{-\lambda} \delta_x(z) \left( h_x(x)-1\right),~z\in S.
\end{equation}
\item In addition, 
\begin{equation}
    \label{eq:nearly_harmonic}
    (ph_z)(x) = e^{-\lambda} h_z (x) + p(x,z) (h_z (z)-1)
\end{equation}
\end{enumerate}
\end{prop}
Note that Proposition \ref{prop:weird_values} and the irreducibility guarantee that both \tro{R1:G3}$\mu_x$ and $h_z$ defined in the proposition are strictly positive and finite on $S$. 
\begin{proof}
For the first claim,\tro{R1:G4}
\begin{align*} \mu_x p (z) &=\sum_y \sum_{s=0}^\infty E_x [{\bf 1}_{\{\tau_\Delta>s\}}{\bf 1}_{\{\tau_x>s\}} \exp(\lambda s) \delta_y (X_s)]p(y,z)\\
& = e^{-\lambda} \sum_{s=0}^\infty e^{\lambda{(s+1)}} E_x[ {\bf 1}_{\{\tau_\Delta>s\}}{\bf 1}_{\{\tau_x >s\}}\delta_z (X_{s+1})]\\
& = e^{-\lambda}\left (\mu_x(z) -\delta_x(z)+\delta_x(z) E_x [\exp(\lambda\tau_x),\tau_x < \tau_\Delta]\right)\\
& = e^{-\lambda}\mu_x (z) + e^{-\lambda}\delta_x(z)\left( E_x [\exp(\lambda\tau_x),\tau_x < \tau_\Delta]-1\right).
\end{align*} \tro{R1:M3}
For the second claim, observe that
\begin{align*}
    h_z(x)&=E_x[\exp(\lambda\tau_z),\tau_z<\tau_\Delta]\\
          &=e^\lambda p(x,z)+\sum_{y\neq z}e^\lambda p(x,y)E_y[\exp(\lambda\tau_z),\tau_z<\tau_\Delta]\\
          & = e^\lambda \sum_{y \in S} p(x,y) h_z(y) + e^\lambda p(x,z) (1- h_z(z)). 
\end{align*}
\end{proof}

\subsection{The Reverse Chain}\tro{R2:M9}
\label{sec:ReverseTool}
Throughout this section we fix a QSD $\nu$ with absorption parameter $\lambda$. We introduce the time-reversed transition function $q$ on $S$:
\begin{equation}
\label{eq:q_def} 
q (y,x) = \nu(x) p(x,y) \frac{e^\lambda}{ \nu(y)},~x,y \in S.
\end{equation} 
\tro{R1:P28}
Note that $q$ also depends on $\nu$, dependence we omit to simplify the notation. In terms of properties,  $q$ has no absorbing states, inherits the irreducibility from $p$, and represents dynamics reversings the arrow of time. Write $Q_{\cdot}, E_{\cdot}^Q$ for the probability and expectation for the Markov Chain on $S$ corresponding to the transition function $q$. We have the following simple lemma obtained from products of \eqref{eq:q_def}.

\begin{lem}
\label{lem:reversal} 
Let ${\bf x} = (x_0,x_1,\dots,x_n)$ be a sequence in $S$. Write $\overset{\leftarrow}{\bf x} = (x_n,x_{n-1},\dots,x_0)$, the reverse sequence.  Then 
 $$\prod_{j=0}^{n-1} p(x_j,x_{j+1})= e^{-\lambda n}\frac{\nu(x_n)}{\nu(x_0)} \prod_{j=n}^1 q(x_j,x_{j-1}). $$
 In particular 
 \begin{equation}
     \label{eq:path_circle}
     P_{x_0}(X_{[0,n]} = {\bf x}) = e^{-\lambda n} \frac{\nu(x_n)}{\nu(x_0)}  Q_{x_n} (X_{[0,n]}=\overset{\leftarrow}{\bf x}),
 \end{equation}
 and 
 \tro{R1:G5}\begin{equation}
     \label{eq:path_absorbed}
     P_{x_0}(X_{[0,n]} = {\bf x},\tau_\Delta = {n+1}) = e^{-\lambda n} \frac{\nu(x_n)p(x_n,\Delta)}{\nu(x_0)}  Q_{x_n} (X_{[0,n]}=\overset{\leftarrow}{\bf x}).
 \end{equation}
\end{lem}
To introduce the next result, we need some additional notation.  \tro{R1:G6}Recall $^0\tau_x$ defined in \eqref{eq:hitting0}, and let  
\begin{equation}
    \label{eq:Inu}
    I_\nu(x) = e^\lambda\sum_z \nu(z) p(z,\Delta) Q_z (^0\tau_x<\infty).
\end{equation}
A simple formula expressing $I_\nu$ not involving the reverse process will be presented in Proposition \ref{prop:reversal_probs}-2. 
\begin{cor}
\label{cor:Inu}
Let $\nu$ be a QSD for $p$ with absorption parameter $\lambda$. Then  $I_\nu(x) \le e^{\lambda} -1$, and each of the following cases implies an euqality.  
   \begin{enumerate} 
   \item $q$ is recurrent. 
   \item \label{it:second_order} There's a bijection $\sigma:S \to \Z_+\cup\{-1\}$ with $\sigma(\Delta)=-1$
   and the following properties 
\begin{enumerate} 
\item For all $x \in S$, $p(x,y)>0$ if $\sigma(y)=\sigma(x)-1$. 
\item For all $x \in S$, $p(x,y)=0$ if $\sigma(y)<\sigma(x)-1$. 
\end{enumerate} 
\end{enumerate} 
\end{cor}
A chain satisfying the latter set of conditions is also known as skip-free \cite{SkipfreeDisCts}. Such chains are the simplest to study. Complete characterization of all QSDs for such chains is given in Section \ref{sec:skipfree}. 
\begin{proof}
Since $\nu(S)=1$, we have \tro{R1:G7}
\begin{equation}   \label{eq:nupzero_summation}
    I_\nu(x) \le e^\lambda \sum_{z} \nu(z) p (z,\Delta) = e^{\lambda} \sum_z \nu(z)\left (1-\sum_{y\neq \Delta} p(z,y)\right) = e^{\lambda} (1-e^{-\lambda})= e^\lambda-1.
\end{equation}
From the definition of $I_\nu$, \eqref{eq:Inu}, the inequality in \eqref{eq:nupzero_summation} is an equality if and only if for every $z$ with $p(z,\Delta)>0$, we also have $Q_z (^0\tau_x < \infty)=1$. This latter condition is automatically satisfied by each of the conditions listed in the statement. \tro{R1:M4}
\end{proof}

As an application of Lemma \ref{lem:reversal}, we have 
\begin{prop}
\label{prop:reversal_probs}
Let  $\nu$ be a QSD for $p$ with absorption parameter $\lambda$. Then 
\begin{enumerate} 
\item \label{it:1} $E_x [\exp (\lambda\tau_x),\tau_x < \tau_\Delta]=Q_x (\tau_x<\infty).$
\item \label{it:2} $E_x [ \exp (\lambda \tau_\Delta),\tau_\Delta <\tau_x] = \frac{I_\nu(x)}{\nu(x)}$.
\item \label{it:3} $E_x [\exp(\lambda \tau_\Delta)]= \frac{I_\nu(x)}{\nu(x)}E_x^Q [N(x)]$, \tro{R1:G8}where $N(x) = \sum_{s=0}^\infty \delta_x (X_s)$.
\item \label{it:4} $E_x [\exp (\lambda \tau_x)\tau_x,\tau_x <\tau_\Delta] =E^Q_x[\tau_x,\tau_x< \infty].$
\end{enumerate} 
\end{prop}
Before moving to the proof, we wish to make a derivation and state a corollary.  If $\nu$ and $\lambda$ are as in the proposition, then   Proposition \ref{prop:weird_values}-1 gives that  $E_x [\exp (\lambda \tau_x ),\tau_x<\tau_\Delta]<1$ for all $x\in S$ and $\lambda< \lambda_{cr}$, .\tro{R1:M5} Therefore it follows from the first part of  the proposition that  $q$ recurrent if and only if  $\lambda=\lambda_{cr}$ and $E_x [\exp (\lambda_{cr} \tau_x ),\tau_x<\tau_\Delta]=1$. This, along with part 4 of the proposition, lead to 
\tro{R2:M10}
\begin{cor}
Let $\nu$ be a QSD for $p$ with absorption parameter $\lambda$. Then the reverse transition function $q$ is positive recurrent if and only if $E_x [ \exp (\lambda_{cr} \tau_x)\tau_x,\tau_x < \tau_\Delta]<\infty$ for some (equivalently all ) $x\in S$, and the stationary distribution for $q$ is 
$$\pi(x) = \frac{1}{E_x^Q[\tau_x]}=\frac{1}{E_x [ \exp (\lambda_{cr} \tau_x)\tau_x,\tau_x < \tau_\Delta]}.$$ 
\end{cor} 
\begin{proof}[Proof of Proposition \ref{prop:reversal_probs}]
For $n\in \Z_+$, $x\in S$, let $A_{x}(n)$ be the set  of paths ${\bf x}=(x_0,x_1,\cdots,x_n)$ with $x_0=x$. Also, let $A^-_{x}(n)\subset A_{x}(n)$ the subset of paths satisfying $x_1,\dots,x_{n-1} \ne x$, and finally, let $A^-_{x,x}(n)$ be the subset of $A^-_x(n)$ consisting of paths satisfying $x_n=x$.  

For the first assertion, using \eqref{eq:q_def} and \eqref{eq:path_circle}, for $0<i<n$ we have 
\begin{align*}
    E_x [\exp (\lambda\tau_x),\tau_x < \tau_\Delta]&=\sum_{n=0}^\infty \sum_{{\bf x} \in A^-_{x,x}(n)}e^{\lambda n}P_{x}(X_{[0,n]} = {\bf x})\\
    &=\sum_{n=0}^\infty \sum_{{\bf x} \in A^-_{x,x}(n) } Q_{x} (X_{[0,n]}=\overset{\leftarrow}{\bf x})\\
    &=Q_x(\tau_x<\infty)
\end{align*}
For the second assertion, consider $z\in S$. Using \eqref{eq:path_absorbed}, we obtain 
\begin{align*}
    E_x [ \exp (\lambda \tau_\Delta),\tau_\Delta <\tau_x]&= \sum_{n=0}^\infty \sum_{{\bf x}\in A_{x}^-(n)}e^{\lambda (n+1)}P_{x}(X_{[0,n]} = {\bf x}, \tau_\Delta=n+1)\\
    &=\sum_{z\in S}\frac{e^\lambda}{\nu(x)}\nu(z)p(z,\Delta)Q_z(^0\tau_x<\infty)\\
    &=\frac{I_\nu(x)}{\nu(x)}
\end{align*}
For the third assertion, using \eqref{eq:path_absorbed} for $x\in S$\tro{R1:M6}
\begin{align*}
    E_x [\exp(\lambda \tau_\Delta)]&=\sum_{n=0}^\infty \sum_{{\bf x}\in A_x(n)}e^{\lambda (n+1)}P_{x}(X_{[0,n]} = {\bf x}, \tau_\Delta=n+1)\\
    &\overset{\eqref{eq:path_absorbed}}{=}
     \sum_{n=0}^\infty \sum_{{\bf x} \in A_x(n)} e^{\lambda (n+1)} e^{-\lambda n} \frac{\nu(x_n)p(x_n,\Delta)}{\nu(x)}  Q_{x_n} (X_{[0,n]}=\overset{\leftarrow}{\bf x})\\
     & = \frac{e^{\lambda}}{\nu(x)} \sum_{n=0}^\infty\sum_{z \in S} \nu (z)p(z,\Delta)Q_z (X_n =x)\\
      & = \frac{e^{\lambda}}{\nu(x)} 
    \sum_{z\in S}\nu (z)p(z,\Delta)E^Q_z[N(x)]\\
 \end{align*}
Now $E_z^Q [N(x)] = Q_z(^0\tau_x <\infty)E^Q_x[ N(x)]$, and it follows from the definition of $I_\nu(x)$ \eqref{eq:Inu} that 
$$  E_x [\exp(\lambda \tau_\Delta)] =\frac{I_\nu(x)}{\nu(x)}E_x^Q[N(x)].$$

To prove the last assertion, using \eqref{eq:path_circle}
\begin{align*}
    E_x^Q[\tau_x,\tau_x < \infty]&=\sum_{n=1}^\infty n\sum_{{\bf X}\in A_{x,x}^-(n)}Q_x(X_{[0,n]} = {\bf x})\\
    &=\sum_{n=1}^\infty n\sum_{{\bf X}\in A_{x,x}^-(n)}e^{\lambda n}P_x(X_{[0,n]} = {\bf x})\\
    &=\sum_{n=1}^\infty ne^{\lambda n}\sum_{{\bf X}\in A_{x,x}^-(n)}P_x(X_{[0,n]} = {\bf x})\\
    &=E_x[\tau_x\exp(\lambda\tau_x),\tau_x<\tau_\Delta]
\end{align*}
\end{proof}
Corollary \ref{cor:Inu} and Proposition \ref{prop:reversal_probs}-\ref{it:2} give the following: 
\begin{cor}
\label{cor:domination}
Let $\nu$ be  a QSD  with absorption parameter $\lambda$. Then for all $x\in S$
\begin{equation}
    \label{eq:dominated}
    \nu(x) = \frac{I_\nu(x)}{E_x [\exp (\lambda \tau_\Delta),\tau_\Delta < \tau_x]}\le \frac{e^\lambda-1}{E_x [\exp (\lambda \tau_\Delta),\tau_\Delta < \tau_x]}.
\end{equation}
Moreover, if an equality holds in \eqref{eq:dominated} for all $x$, then $\nu$ is the unique QSD with absorption parameter $\lambda$. 
\end{cor}
\label{cor:conv}
The transition function $q$ may be useful in the study of convergence to a specific QSD. Indeed, if $\nu$ is a QSD with absorption parameter $\lambda$ and $q$ is the transition function for the reverse process as defined in \eqref{eq:q_def}, then for any probability measure $\mu$ on $S$ and any $y\in S$
$$ P_\mu (X_n = y)=e^{-\lambda n} \nu(y) E_y^Q [ \frac{\mu}{\nu}(X_n)].$$
This implies 
\begin{equation} 
\label{eq:ratioQ} P_\mu (X_n = y | \tau_\Delta >n) = \nu(y) \frac{E_y^Q [ \frac{\mu}{\nu}(X_n)]}{E_\nu^Q [\frac{\mu}{\nu}(X_n)]}
\end{equation}
\begin{cor}
\label{cor:limit}
Suppose that $\mu$ is a probability measure satisfying $\sup_{y\in S} \frac{\mu}{\nu}(y)<\infty$. Then, each of the conditions below implies 
$$\lim_{n\to\infty} P_\mu (X_n \in \cdot ~|~\tau_\Delta>n) = \nu(y). $$
\begin{enumerate} 
\item $q$ is positive recurrent and aperiodic. 
\item $q$ is transient and $\lim_{y\to\infty} \frac{\mu}{\nu}(y)$ exists and is strictly positive. 
\end{enumerate}
\end{cor}
Of course, both parts of the corollary follow from \eqref{eq:ratioQ}. The first part is a straightforward application of the ergodic theorem for positive recurrent and aperiodic Markov chains, e.g., \cite[Chapter 3]{Wolfgang}, and the second is a trivial application of the transience assumption. 
\subsection{Proof of Theorem \ref{th:nu_recurr}}
\underline{Uniqueness and necessity}. Suppose first that $\nu$ is a minimal QSD. We construct the reverse chain corresponding to \eqref{eq:q_def}. Proposition \ref{prop:reversal_probs}-3  gives that for every $x\in S$, 
$\infty = E_x [\exp(\lambda_{cr} \tau_\Delta)]= \frac{I_\nu(x)}{\nu(x)}E_x^Q [N(x)]$. The ratio on the righthand side is finite due to \eqref{eq:nupzero_summation}, and so $E_x^Q [N(x)]$, the expected number of visits to $x$ by the reverse chain, is infinite. As a result, $q$ is recurrent. Proposition \ref{cor:Inu}-1 then  gives $I_\nu (x)=e^{\lambda_{cr}}-1$, and then Proposition \ref{prop:reversal_probs}-\ref{it:2} gives the representation  \eqref{eq:nu_recurr} for $\nu$. This proves the uniqueness of a minimal QSD and also implies \eqref{eq:finite_stopped} due to Proposition \ref{prop:weird_values}.

\noindent\underline {Existence and sufficiency.} Suppose that  \eqref{eq:finite_stopped} holds. Then by Proposition \ref{prop:weird_values}, $E_x [ \exp (\lambda_{cr} \tau_x) , \tau_x <\tau_\Delta]=1$, and therefore the measure $\mu_x$ from Proposition \ref{prop:mu_x} with $\lambda= \lambda_{cr}$ satisfies \eqref{eq:eigen_description}. Since $\mu_x (S) = \dfrac{E_x [ \exp (\lambda_{cr}(\tau_\Delta \wedge \tau_x))]-1}{e^{\lambda_{cr}}-1}$, $\mu_x$ can  be normalized to a probability measure we denote by 
$\bar \mu_x$. Proposition \ref{prop:eigen_equivalence} implies that $\bar \mu_x$ is a minimal QSD. 

\noindent\underline {Convergence.}  Since $q$ is already recurrent, the additional assumption \eqref{eq:positive_recurrent} and Proposition \ref{prop:weird_values}-4 give that $q$ is positive recurrent. And then, we apply the first part of Corollary\tro{R1:G9} \ref{cor:limit}.
\qed

\subsection{Proof of Proposition \ref{prop:K2single}}
The statement is trivial when $|K|=1$.  We will show that if  $|K|\ge 2$, there exists $x_0\in K$ so that 
\begin{equation} 
\label{eq:lower_one} 
E_{x_0} [ \exp (\lambda_{cr} (\tau_\Delta\wedge\tau_{K-\{x_0\}}) )]
<\infty,
\end{equation}\tro{R1:P29}
and so by iterating, we can eventually reduce to the case $|K|=1$.  We, therefore, turn to prove \eqref{eq:lower_one}. For $\lambda < \lambda_{cr}$ and $x\in S$,
\begin{align*} \infty > E_x [\exp (\lambda  \tau_
\Delta) ] & = E_x[ \exp (\lambda (\tau_\Delta \wedge \tau_K)),\tau_K < \tau_\Delta] E_{X(\tau_K)} [\exp (\lambda \tau_\Delta)]\\
& \quad+E_x [\exp (\lambda (\tau_\Delta \wedge \tau_K)) ,\tau_\Delta < \tau_K].
\end{align*}

\tro{R1:M7}Let $v_m(\lambda)$ denote the minimum of the lefthand side over $x\in K$, and let $x=x_m(\lambda)\in K$ be a minimizer. Then, 
 
$$v_m(\lambda) \ge  E_{x_m}[ \exp (\lambda (\tau_\Delta \wedge \tau_K)),\tau_K < \tau_\Delta] v_m(\lambda),$$

and therefore $E_{x_m}[ \exp (\lambda (\tau_\Delta \wedge \tau_K)),\tau_K < \tau_\Delta]\le 1$.  Let $\lambda \nearrow \lambda_{cr}$ along any sequence $(\lambda_n:n\in\N)$. Since $K$ is finite, we can extract a subsequence  which we also denote by $(\lambda_n:n\in\N)$ with the same minimizer $x_m$  for all $n$. It follows from the monotone convergece theorem that 
$$ E_{x_m} [ \exp (\lambda_{cr} (\tau_\Delta \wedge \tau_K)),\tau_K < \tau_\Delta]\le 1.$$
Note that this is a version of Proposition \ref{prop:weird_values} but with the singleton replaced by a finite set. The event $\{\tau_K < \tau_\Delta\}$ is the disjoint union $\{\tau_{x_m}= \tau_K < \tau_{\Delta}\}\cup \{\tau_{K-x_m}=\tau_K < \tau_\Delta\}$. \tro{R1:M8}From the irreducibility,  the second event has positive probality under $P_{x_m}$, and therefore, 
$$\rho_m = E_{x_m} [\exp (\lambda_{cr} (\tau_\Delta \wedge \tau_K)),\tau_{x_m} = \tau_\Delta \wedge \tau_K]< E_{x_m} [ \exp (\lambda_{cr} (\tau_\Delta \wedge \tau_K)),\tau_K < \tau_\Delta]\le 1.$$

\tro{R1:M9}For $\lambda< \lambda_{cr}$ the strong Markov property gives  \begin{align*}  E_{x_m} [ \exp (\lambda (\tau_\Delta\wedge\tau_{K-\{x_m\}}) ) ]& \le \rho_m   E_{x_m} [ \exp (\lambda (\tau_\Delta\wedge\tau_{K-\{x_m\}}))]
\\ &\quad + 
E_{x_m} [\exp (\lambda (\tau_\Delta \wedge \tau_K)),\tau_{x_m} > \tau_\Delta \wedge \tau_K].
\end{align*}
Therefore 
\begin{align*} E_{x_m} [ \exp (\lambda (\tau_\Delta\wedge\tau_{K-\{x_m\}})) ] &\le \frac{E_{x_m} [\exp (\lambda (\tau_\Delta \wedge \tau_K)),\tau_{x_m} > \tau_\Delta \wedge \tau_K]}{1-\rho_m}\\
& \le \frac{E_{x_m}[\exp(\lambda_{cr} (\tau_\Delta \wedge \tau_K))]}{1-\rho_m}.
\end{align*}
The result follows by applying the monotone convergence theorem on the left-hand side, proving \eqref{eq:lower_one}, thus completing the proof of the proposition.\tro{R1:M10}  \qed

\subsection{Proof of Theorem \ref{thm:coming_infinty} and Proposition \ref{prop:Sfinite}} 

We begin by presenting a sufficient condition for $\lambda_{cr}$ to be in the infinite MGF regime. 
\begin{prop}
\label{prop:highmomments}
Suppose that $K\subsetneq S$ is nonempty and finite and that  for some $\bar\lambda > \lambda_{cr}$
\begin{equation} 
\label{eq:dominated_MGF} \sup_{x} E_x [\exp (\bar \lambda (\tau_\Delta \wedge \tau_K))] < \infty.
\end{equation}
Then $\lambda_{cr}$ is in the infinite MGF regime. 
\end{prop}
\begin{proof}
We argue by contradiction assuming $E_x [ \exp (\lambda_{cr} \tau_\Delta)]<\infty$ for all $x\in S$. For $x\in S$, let $T_1(x)$ be a random variable whose distribution is the same as $\tau_\Delta \wedge \tau_K$ under $P_x$ and let $T_2$ be independent of $T_1(x)$ and equal to the sum of $|K|$, independent random variables $(T_{2,k}:k\in K)$, \tro{R1:G10}where for each such $k$,  $T_{2,k}$ is distributed according to $\tau_\Delta$ under $P_k$. Then the distribution of $\tau_\Delta$ under $P_x$ is stochastically dominated by $T_1(x) + T_2$.\tro{R1:M11} Therefore for any $n\ge \N$, 
\begin{align*} E_x [\exp (\lambda_{cr} \tau_\Delta),\tau_\Delta > 2n ] & \le E[\exp (\lambda_{cr} (T_1 (x)+T_2)),T_1(x) + T_2 > 2n] \\
& \le  E[\exp (\lambda_{cr} (T_1 (x)+T_2))({\bf 1}_{\{T_1(x) >n\}}+{\bf 1}_{\{ T_2 > n\}})]\\
& \le E[\exp (\lambda_{cr} T_2)]E_x[\exp (\lambda_{cr} T_1 (x)),T_1(x)>n]\\
&\quad+ E[\exp (\lambda_{cr} T_1(x))]E_x[\exp (\lambda_{cr}T_2),T_2>n]\\
&=(*),
\end{align*} 
\tro{R1:P30}
with all expectations on the righthand side being finite. Let $\delta = \bar\lambda - \lambda_{cr}>0$. Then on the event $\{T_1(x) >n\}$ 
$$\exp (\lambda_{cr} T_1(x))= \exp ( \bar\lambda T_1(x))\exp (-\delta T_1(x))\le \exp ( \bar\lambda T_1(x))e^{-\delta n}.$$
Therefore 
$$E[\exp (\lambda_{cr} T_1 (x)),T_1(x)>n]\le E[\exp (\bar\lambda T_1(x))]e^{-\delta n}\le c_1 e^{-\delta n},$$
where $c_1 <\infty$ is the supremum in \eqref{eq:dominated_MGF}.  Then, 
$$(*) \le E[\exp (\lambda_{cr} T_2)]c_1 e^{-\delta n} + c_1 E[\exp (\lambda_{cr} T_2),T_2>n].$$
This upper bound is independent of $x$ and  tends to $0$ as $n\to\infty$. Thus the the distributions of $e^{\lambda_{cr} \tau_\Delta}$ under $P_x$ as $x$ ranges over $S$ is uniformly integrable, which is in contradiction to Proposition \ref{prop:UI}. 
\end{proof}

\begin{proof}[Proof of Theorem \ref{thm:coming_infinty}]
    We first show that $\lambda_{cr}$ is in the infinite MGF regime. From \eqref{eq:blowlambda0} we learn that $\lambda_{cr}\le \bar\lambda$. If an equality holds, then $\lambda_{cr}$ is in the infinite MGF regime. Otherwise, $\lambda_{cr} < \bar\lambda$, and the claim follows from Proposition \ref{prop:highmomments}. 

With this, Proposition \ref{prop:K2single} guarantees that both conditions of Theorem \ref{th:nu_recurr} hold and therefore, there exists a unique minimal QSD $\nu$, given by \eqref{eq:nu_recurr}. However, the assumptions yield more. 

\tro{R1:M12}Let $z\in K$ be as in Proposition \ref{prop:K2single}, and let  $h(x) = h_z(x) = E_x [ \exp (\lambda_{cr} \tau_z),\tau_z <\tau_\Delta]$. The proposition implies $h(z)<\infty$ and then  Proposition \ref{prop:weird_values} gives $h(z)=1$. 
Use the irreducibility and the strong Markov property to obtain 
$$1=h(z)=E_z [ \exp (\lambda_{cr}\tau_ z), \tau_z < \tau_\Delta ] \ge E_z [ \exp (\lambda_{cr} \tau_x),\tau_x<(\tau_z\wedge \tau_\Delta)]h(x),$$
therefore $h$ is finite on $S$. 

Proposition \ref{prop:mu_x} gives 
\begin{equation} 
\label{eq:h_harmonic} p h = e^{-\lambda_{cr}} h.
\end{equation}
Next we show that $h$ is also bounded. Indeed, as $z \in K$, $\tau_K \le \tau_z$,  and the strong Markov property and finiteness of $K$ then gives 
$$ h(x) \le E_x[ \exp (\lambda_{cr} \tau_K),\tau_K < \tau_\Delta] \max_{x \in K}h(x)<\infty.$$
As a result,  $\nu h $ can be normalized to be a probability measure. Let $\bar h = h / (\sum \nu h)$. Then  $\pi = \nu \bar h $ is probability measure. Let $p^h$ be the kernel on $S$ defined by 
$$p^h(x,y) = \frac{e^{\lambda_{cr}}}{h(x)} p(x,y) h(y).$$
Then \eqref{eq:h_harmonic} guarantees that $p^h$ is a transition function on $S$. Moreover, 
$$\pi  p^h =   e^{\lambda_{cr} } \nu \frac{\bar h}{h}  p h   = \nu \bar h = \pi,$$
and therefore the irreducible transition function \tro{R1:G11} $p^h$ has a stationary probability distribution, which implies that it is positive recurrent (e.g \cite[Theorem 1.7.7]{norris}\tro{R1:M13}). The same calculation shows that $\pi$ is also a stationary distribution for the transition function  for the reverse chain $q$, \eqref{eq:q_def}, and therefore $q$ is positive recurrent and  Proposition \ref{prop:reversal_probs}-\ref{it:4} guarantees that the condition \eqref{eq:positive_recurrent} holds. 

To show that no other QSD exists, note that if $\nu'$ is a QSD with absorption parameter $\lambda>0$, then 
\tro{R1:G12}$$ \sum_{x,y}(\nu' h)(x) p^h (x,y) = \sum_{x,y} \nu' (x) p(x,y) h(y),$$
and the sum is finite because $\nu' h$ is a finite measure. Summing the righthand side over $x$ first yields $e^{-\lambda} \sum_y \nu'(y) h(y)$, while summing the righthand side over $y$ first yields $e^{-\lambda_{cr}} \sum_{x} \nu'(x) h(x) $. Therefore, $\lambda= \lambda_{cr}$ and the claim follows. \tro{R1:M14}

It remains to prove the convergence result. From the construction of $p^h$ we have that for any initial distribution $\mu$, 
\begin{equation}
    \label{eq:limith_ratio}
 P_\mu (X_n=y | \tau_\Delta > n) =\frac{\sum_{x \in S}  \mu(x)  h(x) P^{h}_x (X_n = y) /  h(y)} {\sum_{x \in S}  \mu(x)  h(x) E^{h}_x [ \frac{1}{ h(X_n)} ]}= \frac{1}{h(y)}\frac{P^h_{\bar \mu} (X_n = y)}{E^h_{\bar \mu} [\frac{1}{h(X_n)}]},
\end{equation}
where $\bar \mu$ is the  probability measure\tro{R1:G13} given by  $\bar \mu(x) = (\mu h)(x)/ \sum_{y} (\mu h)(y) $. To complete the proof, we need to identify the limit on the righthand side of \eqref{eq:limith_ratio} through an application of the ergodic theorem for positive recurrent Markov chains. 

\tro{R1:M15}From  \eqref{eq:h_harmonic} we conclude that $(e^{\lambda_{cr} n } h(X_n){\bf 1}_{\{\tau_\Delta > n\}}:n\in\Z_+)$ is a martingale. In particular, for every $n\in\Z_+$,
$$h(x) = E_x [ \exp (\lambda_{cr} (\tau_{x_0} \wedge n)) h(X_{\tau_{x_0}\wedge n}),\tau_\Delta > \tau_{x_0}\wedge n ].$$ 
Taking $n\to\infty$, Fatou's lemma implies 
$$ h(x) \ge E_x [\exp (\lambda_{cr} \tau_{x_0} )h(x_0),\tau_\Delta> \tau_{x_0}]  \ge P_x (\tau_\Delta >\tau_{x_0}) h(x_0).$$
 This inequality applied to \eqref{eq:lowerh} implies that there exists a positive constant depending on $n_0$ and $x_0$ such that for all $x \in S$, 
$$ h(x) \ge c P_x (\tau_\Delta > n_0).$$
Let 
\tro{R1:M16}$$l(x) = E^h_x [ \frac{1}{h(X_{n_0})}]=\frac{e^{\lambda_{cr}n_0}}{h(x)}P_x(\tau_\Delta > n_0).$$  Then previous inequality implies that $l$ is bounded from above.



 It follows from the  Markov property and the ergodic theorem for positive recurrent aperiodic Markov chains that for any initial distribution $\bar \mu$ on $S$, 

$$\lim_{n\to\infty} E_{\bar \mu}^h [ \frac{1}{h(X_n)}] = \lim_{n\to\infty} E_{\bar \mu}^h [ l(X_{n-n_0})]=   \sum_{x} \pi(x) l(x).$$
As $\pi$ is stationary, the definition of $l$ then gives  gives that the righthand side is equal to $\sum_x \pi (x) /h(x) = \sum_x \nu(x) \bar h (x) / h(x) =\frac{1}{\sum_x \nu(x) h(x)}$. The ergodic theorem also gives $\lim_{n\to\infty} P^h_{\bar \mu} (X_n = y) = \pi (y)= \nu(y) h(y) / \sum_{x} \nu(x) h(x)$. Therefore, as $n\to\infty$, the righthand side of \eqref{eq:limith_ratio} tends to $\nu(y)$, completing the proof. 
\end{proof}

\tro{R1:P31}\begin{proof}[Proof of Proposition \ref{prop:Sfinite}]\tro{R1:G14}
 Since $S$ is finite and irreducible,  \eqref{eq:lowerh} automatically holds for every\tro{R1:G15} $n_0\in\N$ and $x_0\in S$. Moreover, \eqref{eq:blowlambda0} is Corollary \ref{cor:nu_recurr_Sfinite}-1. 

Next, since $S$ is irreducible, there exists $y_0\in S$ different from $x_0$ such that $p(x_0,y_0)>0$. In particular $1-p(x_0,\Delta)\ge p(x_0,x_0)+p(x_0,y_0)>p(x_0,x_0)$. The choice of $x$ guarantees $P_x (\tau_\Delta > n) \ge  (1-p(x_0,\Delta))^n\ge (p(x_0,x_0)+ p(x_0,y_0))^n$. Since for any $\lambda< \lambda_{cr}$, $\sum_{n=0}^\infty e^{\lambda n} P_x (\tau_\Delta> n) $ is finite, it follows that $e^{\lambda_{cr}}  (p(x_0,x_0)+ p(x_0,y_0))\le 1$, and in particular, $p(x_0,x_0) e^{\lambda_{cr}}<1$. Moreover,
$$ P_{x_0}(\tau_\Delta\wedge \tau_K> n) = p(x_0,x_0)^n,$$ 
and because $p(x_0,x_0) e^{\lambda_{cr}}< 1$, the series $\sum_{n=0}^\infty e^{\lambda_{cr} n} P_{x_0}(\tau_\Delta \wedge \tau_K > n) $ converges, which gives $E_{x_0} [\exp (\lambda_{cr} (\tau_\Delta \wedge \tau_K))]<\infty$. Thus, \eqref{eq:bd_arrival} holds. 
\end{proof}

\section{Proof of the Results of Section \ref{sec:results_finite_disc_martin}}
\label{sec:pf_finite_1}\tro{R2:M12}
We recall that in Section \ref{sec:results_finite_disc_martin}, $\lambda>0$ is in the finite MGF regime and therefore in this section we will work under this assumption as well. 

\begin{proof}[Proof of Proposition \ref{prop:Compact_Martin}]
 It is enough to show that every sequence $(x_n:n\in\N)$ in this space has a convergent subsequence. Indeed, either
\begin{enumerate}
    \item \label{case 1} Some elements in $S$ or in $\partial^{\lambda} M$ appear infinitely often;  or  
    \item \label{case 2} Every element in $S$ appears finitely often, and so does every element in $\partial^\lambda M$. Here, at least one of the two alternatives holds: 
    \begin{enumerate} 
    \item There exists a subsequence we also denote by $(x_n:n\in\N)$ consisting entirely of  elements in $S$  and  satisfying $\lim_{n\to\infty} x_n = \infty$.
    \item There exists a subsequenence we also denote by $(x_n:n\in\N)$ consisting entirely of  elements in $\partial^\lambda M$. 
    \end{enumerate} 
\end{enumerate}
In case \ref{case 1}, we have a constant subsequence with an obvious limit. In case \ref{case 2}, we proceed according to the subcases. By construction (see the comment above Definition \ref{Def:Martin}, the case \ref{case 2}-(a) has a $\lambda,\infty$-convergent subsequence with a limit in $\partial^\lambda M$. \tro{R1:M17}It remains to consider the case  \ref{case 2}-(b). By construction, for every $n\in\N$, there exists a sequence $(\bar x_{n,k}:k\in\N)$ of elements of $S$  with the property $\lim_{k\to\infty} \bar x_{n,k}= x_n$. As a result, we can pick a subsequence $(\bar x_{n,k_n}:n\in\N)$ of elements in $S$ with the properties  $\rho^\lambda(\bar x_{n,k_n},x_n) < 2^{-n}$ for all $n\in\N$ and  $\lim_{n\to\infty} \bar x_{n,k_n} = \infty$. To ease notation, we denote this subsequence by $(\bar x_n:n\in\N)$. Alternative (a) holds for the sequence $(\bar x_n:n\in\N)$, and as such, it possesses a $\lambda,\infty$-convergent subsequence $(\bar x_{n_l}:l\in\N)$ with some limit $[{\bf x}]\in \partial^\lambda M$. It follows from the triangle inequality that $\lim_{l\to\infty} \rho^\lambda(x_{n_l},[{\bf x}])=0$, and the result follows.    
\end{proof}
\begin{defn}
\label{defn:superharmonic}
Define the set of $e^{-\lambda}$-superharmonic measures on $S$ as  ${\cal M}^\lambda = \{\mu : S \to [0,\infty): \mu  p\le e^{-\lambda} \mu \}$. A measure $\mu\in{\cal M}^\lambda$ is 
\begin{enumerate} 
\item  a $e^{-\lambda}$-potential if there exists $f_\mu :S \to [0,\infty)$ such that 
 $$\mu(y) = \sum_{x\in S} f_\mu (x) K^\lambda(x,y);$$
 \item $e^{-\lambda}$-harmonic if  $\mu p= e^{-\lambda }\mu$.
\end{enumerate}
\end{defn}
Before we move to concrete examples and characterization of $e^{-\lambda}$-superharmonic functions, we present the following.  
\begin{lem}
\label{lem:super_properties}
\begin{enumerate} 
\item 
Let $\mu_1,\mu_2$ be $e^{-\lambda}$-superharmonic measures, then their pointwise minimum $\mu = \mu_1 \wedge \mu_2$ is $e^{-\lambda}$-superharmonic. 
\item Let $(\mu_n:n\in\N)$ be a sequence of $e^{-\lambda}$-superharmonic measures converging pointwise to the limit $\mu_\infty$. Then $\mu_\infty$ is $e^{-\lambda}$-superharmonic. 
\end{enumerate} 
\end{lem}
\begin{proof}
\begin{enumerate} 
\item 
For $y\in S$ and $j=1,2$ we have 
$$ (\mu p)(y) \le (\mu_j p)(y) \le e^{-\lambda}\mu_j(y).$$
As we can pick $j$ as the minimizer of $\mu_1(y)$ and $\mu_2(y)$, we have that for all $y\in S$, $\mu p \le e^{-\lambda} \mu$. 
\item For $y\in S$, 
$$ (\mu_\infty p) (y) = \lim_{n\to\infty} \sum_{x} \mu_\infty (x) p(x,y) \le \liminf_{n\to\infty} \sum_x \mu_n (x) p(x,y) \le \liminf_{n\to\infty} e^{-\lambda} \mu_n (y) =e^{-\lambda} \mu_\infty,$$
where the first inequality follows from Fatou's lemma and the second from the fact that for each $n$, $\mu_n$ is $e^{-\lambda}$-superharmonic. 
\end{enumerate}
\end{proof}
We turn to some concrete examples. For each fixed $x\in S$, we have
$$(G^\lambda (x,\cdot) p)(y)  \overset{\eqref{eq:Green_def}}{=} \sum_{m=1}^\infty e^{\lambda (m-1)} p^m (x,y) =e^{-\lambda}(G^\lambda(x,y) - \delta_x(y))\le e^{-\lambda}G^\lambda(x,y).$$ 
Thus $G^\lambda(x,\cdot)$ is $e^{-\lambda}$-superharmomic. As $K^\lambda(x,\cdot)=G^\lambda(x,{\bf 1})\times G^\lambda(x,\cdot)$, $K^\lambda(x,\cdot)$ is also $e^{-\lambda}$-superharmonic. Finally, Definitions \ref{def:KMartin} and \ref{Def:Martin} and the second part of the lemma guarantee that for $[{\bf x}] \in \partial^\lambda M$, $K^\lambda([{\bf x}],\cdot)$ is $e^{-\lambda}$-superharmonic. Thus, 
\begin{cor}
\label{cor:Klambda}
Let $x\in M^\lambda$, $K^\lambda(x,\cdot)$ is $e^{-\lambda}$-superharmonic.  
\end{cor}
Next we provide a characterization of $e^{-\lambda}$-superhamornic measures.
 \begin{lem}
 \label{lem:super_characterization}
 Let $\mu \in S\to [0,\infty)$. Then $\mu$ is $e^{-\lambda}$-superharmonic if and only if it is a sum of a $e^{-\lambda}$-potential and a $e^{-\lambda}$-harmonic measure $\mu_\infty$. In this case 
 \begin{enumerate} 
 \item 
 \begin{equation}
 \label{eq:harm_decomp} \mu (y)=\sum_{x\in S} f_\mu(x) K^\lambda(x,y) + \mu_\infty(y),
 \end{equation}
 where $f_\mu (x) = G^\lambda (x,{\bf 1})(\mu -e^\lambda \mu p)(x)$ and $\mu_\infty$ is the pointwise limit of the sequence $(\mu(e^\lambda p)^n:n\in\N)$. 
 \item $\mu$ is a $e^{-\lambda}$-potential if and only if $\lim_{n\to\infty} \mu (e^\lambda p)^n (y)=0$ for some (equivalently all) $y\in S$. 
 \end{enumerate}
 \end{lem}
 \begin{proof}
 Assume first that $\mu \in {\cal M}^\lambda$. For $n\in\N$, let  $\mu_n = \mu (e^\lambda p)^n$ and $\bar f_\mu = \mu (I-e^\lambda p)$. Then 
\begin{align*} 
\mu &= \mu - \mu (e^\lambda p)^n + \mu_n\\
& = \bar f_\mu ( I + \dots + (e^\lambda p)^{n-1}) + \mu_n.
\end{align*}
\tro{R1:G17}  Since $\mu \in {\cal M}^\lambda$, $\bar f_\mu \in {\cal M}^\lambda$, and the sequence $(\bar f_\mu (I + \dots +(e^\lambda p)^{n-1}):n\in\N)$ 
 increases pointwise to 
 $$\sum_{x\in S} \bar f_\mu (x) G^\lambda (x,y)\overset{\eqref{eq:Kalpha}}{=}\sum_{x\in S} f_\mu(x) K^\lambda (x,y).$$
Also, by definition of ${\cal M}^\lambda$, for each $y\in S$, $(\mu_n(y):n\in\N)$ decreases pointwise to a limit we denote by $\mu_\infty(y)$. The convergence implies that as a measure on $S$,   $\mu_\infty e^\lambda p = \mu_\infty$, and therefore $\mu_\infty$ is $e^{-\lambda}$-harmonic.

Suppose now that $\mu:S\to [0,\infty)$ is a sum of a $e^{-\lambda}$-potential and a $e^{-\lambda}$-harmonic function, say as given by \eqref{eq:harm_decomp}. Then 
\begin{align*} (\mu p)(y) & = \sum_z \sum_{x} f_\mu (x) K^\lambda (x,z) p(z,y) + \mu_\infty p  \\
& = \sum_x f_\mu (x) \sum_z K^\lambda(x,z) p(z,y) + e^{-\lambda} \mu_\infty \\
& \le \sum_{x}f_\mu(x) e^{-\lambda}K^\lambda(x,y) + e^{-\lambda}\mu_\infty\\
& = e^{-\lambda} \mu,
\end{align*} 
were the first equality follows from Tonelli, and the fact that $\mu_\infty$ is $e^{-\lambda}$-harmonic and the inequality follows from Corollary \ref{cor:Klambda}.

It remains to prove the characterization of potetials in the section assertion. Suppose that $\mu \in {\cal M}^\lambda$. Then from the proof of the first part $\lim_{n\to\infty} \mu(e^\lambda p)^n (y)$ exists for all $y$. From the irreducubility of $S$, the limit is zero for some $y$ if and only if it is zero for all $y$. In this case, $\mu_\infty$ in \eqref{eq:harm_decomp} is identically zero and so $\mu$ is a $e^{-\lambda}$-potential. \tro{R1:M18}Conversely, if $\mu$ is a $e^{-\lambda}$-potential, say $\mu (y) = \sum_x f_\mu (x) K^\lambda (x,y)$, or, equivalently $\mu(y) = \sum_x \bar f_\mu (x) G^\lambda(x,y)$, where $\bar f_\mu (x) = \frac{f_\mu (x)}{G^\lambda (x,{\bf 1})}$, then for $m\in\N$, 
\begin{align*}  (\mu (e^\lambda p)^m) (y)  & = \sum_z \sum_{x} {\bar f}_\mu (x) G^\lambda(x,z) (e^\lambda p)^m (z,y)\\
 & = \sum_x \bar f_\mu (x) \sum_{n=0}^\infty (e^\lambda p)^{n+m} (x,y) \\
 & = \mu (y) - \sum_{x} \bar f_\mu (x) \sum_{n=0}^{m-1}(e^\lambda p)^n,
 \end{align*} 
 a quantity that tends to zero as $m\to\infty$ due to monotone convergence. \tro{R1:G17}
 \end{proof}
 \begin{cor} \tro{R1:G18}
 \label{cor:super_potential} Let $\mu_1 \in {\cal M}^\lambda$ and let $\mu_2$ be a $e^{-\lambda}$-potential. Then $\mu_1 \wedge \mu_2$ is a $e^{-\lambda}$-potential.
 \end{cor}
 \begin{proof}
 The fact that $\mu_1\wedge \mu_2\in {\cal M}^\lambda$ follows from Lemma \ref{lem:super_properties}-(2). As $\mu_1 \wedge \mu_2 (e^\lambda p)^n \le \mu_2 (e^\lambda p)^n \underset{n\to\infty} { \to }0$, the claim follows. 
 \end{proof}
\begin{proof}[Proof of Theorem \ref{thm:martin}]

Suppose that $\mu$ is a QSD with absorption parameter $\lambda$.  Let $(D_n:n\in\N)$ be a sequence of subsets of $S$ satisfying $D_n \subset D_{n+1}$ for all $n\in\N$ and $\cup_{n=1}^\infty D_n = S$. 
\tro{R1:M19} Write $K^\lambda ({\bf 1}_{D_n},\cdot)$ for the $e^{-\lambda}$-potential $\sum_{x} {\bf 1}_{D_n}(x) K^\lambda(x,y)$ and let $\mu _n = \mu \wedge (n K^\lambda ({\bf 1}_{D_n},\cdot))$. Then by construction, $\mu_n \nearrow \mu$, as $n\to\infty$ and from Corollary \ref{cor:super_potential},  for each $n\in\N$,   $\mu_n$ is $e^{-\lambda}$-potential. Therefore, there exists $f_n :S \to [0,\infty)$ such that 
\tro{R1:G16}$\mu_n (y)= \sum_{x \in S}f_n (x) K^\lambda (x,y)$. As each $x\in S\subseteq M^\lambda$ is a closed set, this summation can be viewed as an integral of the function $M^\lambda \ni x\to K^\lambda (x,y)$ with respect to a Borel measure $F_n$ on $M^\lambda$ which vanishes on the complement of $S$. Equivalently 
\tro{R1:M20}
$$\mu_n (y) = \int_{M^\lambda} K^\lambda (x,y) d F_n (x)$$ 
 Also, since $K^\lambda (x,\cdot)$ is a probability measure on $S$, $F_n ({\bf 1})= \mu_n ({\bf 1})\nearrow  \mu ({\bf 1})=1$. Therefore, without loss of generality we may assume $F_n ({\bf 1})>0$, and therefore normalize $F_n$ to be a probability measure by letting ${\bar F}_n = F_n / F_n ({\bf 1})$, 

$$ \mu_n = F_n ({\bf 1}) \int_{M^\lambda } K^\lambda (x,y) d {\bar F}_n (x).$$ 

As Borel probability measures on a compact metric space, $({\bar F}_n:n \in N)$ contains a weakly convergent subsequence. Let ${\bar F}$ denote the limit which is also a Borel measure on $M^\lambda$. We have 

$$ \mu (y) =  \int_{M^\lambda} K^\lambda(x,y) d {\bar F}(x),$$

which provides a representation for $\mu$ in the form stated in the theorem. To complete the proof we now show that $\bar F(S^\lambda)=1$

Since $\mu$ and $\bar F$ are both probability measures and  $0\le K^\lambda(x,{\bf 1})\le 1$ for all $x\in M^\lambda$
$$ 1 =\mu({\bf 1}) = \int_{M^\lambda} K^\lambda(x,{\bf 1}) d {\bar F}(x)\le {\bar F}(M^\lambda)=1,$$
and thereore  $K^\lambda(x,{\bf 1})=1$, $\bar F$-a.s. 
As $\mu$ is a QSD with absorption parameter $\lambda$, for every $y\in S$, 
\begin{align*} 0&= \mu(y) -\mu (e^\lambda p)(y) \\
& = \int_{M^\lambda}  K^\lambda(x,\cdot)  d \bar F (x) - \sum_{z} \int_{M^\lambda}  K^\lambda(x,z)  d \bar F (x) (e^\lambda p)(z,y)\\
& =\int_{M_\lambda} K^\lambda(x,\cdot) (I - e^\lambda p)(y) d \bar F(x)\\
& \ge0,
\end{align*}
where the third equality is due to Fubini-Tonelli and the inequality due to the fact that for evety $x\in M^\lambda$, $K^\lambda(x,\cdot)\in {\cal M}^\lambda$, Corollary \ref{cor:Klambda}. Therefore the inequality is an equality, and so, $\bar F$ almost surely $K^\lambda(x,\cdot)$ is $e^{-\lambda}$-harmonic. 

\tro{R1:M21} We proved that $\bar F$-almost surely,   $K^\lambda (x,\cdot)$ is a probability measure on $S$ which is  $e^{-\lambda}$-harmonic, a QSD with absorption parameter $\lambda$. That is,  
$$ \bar F (\{x \in M^\lambda: K^\lambda(x,\cdot)\mbox{ is a QSD with absorption parameter } \lambda\})=1.$$

Let $x\in S$. Then by definition \ref{defn:superharmonic}, $K^\lambda(x,\cdot)$ is a $e^{-\lambda}$-potential.  Lemma \ref{lem:super_characterization}-(2) guarantees that $K^\lambda(x,\cdot)$ is not a QSD, and therefore $\bar F(S)=0$, or, equivalently, $\bar F(\partial^\lambda M)=1$. As defined $S^\lambda$ is the set of $[{\bf x}]\in \partial^\lambda M$ such  that  $K^\lambda([{\bf x}],\cdot)$ is a QSD with absorption parameter $\lambda$, we proved that $\bar F (S^\lambda)=1$.  
\end{proof}

\section{Proof of the Results of Section \ref{sec:results_finite_disc_regular}}
\label{sec:pf_finite_2}
\subsection{Preliminary Results}
We now present a number of results that culminate the proof of the Theorem \ref{thm:QSD_tightness}. In the remainder of this section we will  assume  that  $\lambda>0$ in the finite MGF regime. 

\begin{prop}
\label{prop:nux_prop}
Fix $x\in S$. Then  
 \begin{enumerate}  
 \item For every $y \in S$, $$ (K^\lambda (x,\cdot) p)(y) = e^{-\lambda} \left(K^\lambda (x,y) - \frac{\delta_x(y)}{G^\lambda(x,{\bf 1})}\right).$$
 \item 
 \begin{equation} 
 \label{eq:zetaunderpnu} P_{K^\lambda(x,\cdot)}(\tau_\Delta > n) =  \left ( 1- \frac{E_x[ \exp (\lambda (\tau_\Delta \wedge n))-1]}{E_x [\exp (\lambda \tau_\Delta)-1]}\right)e^{-\lambda n},~n\in \Z_+.
 \end{equation} 
 \end{enumerate} 
\end{prop} 
\begin{proof} 
The first assertion is obtained by conditioning on the first step.  As for the second,  
\begin{align*} G^\lambda(x,{\bf 1})P_{K^\lambda(x,\cdot)} (\tau_\Delta >n)&= \sum_{y\in S} \sum_{s=0}^\infty E_x [{\bf 1}_{\{s<\tau_\Delta\}}\exp (\lambda s) \delta_y (X_s)] P_y (\tau_\Delta > n)\\ 
& =  \exp (-\lambda n) E_x[ \sum_{s=0}^{\infty} \exp (\lambda(s+n)){\bf 1}_{\{\tau_\Delta > s+n\}}]\\
& = \exp (-\lambda n)\left ( \sum_{s=n}^\infty \exp (\lambda s) E_x [{\bf 1}_{\{\tau_\Delta> s\}}]\right)\\
& = \exp(-\lambda n)\left (G^\lambda(x,{\bf 1}) - E_x [ \sum_{s=0}^{(\tau_\Delta-1)\wedge (n-1) } \exp (\lambda s)]\right)\\
& = \exp (-\lambda n)\left(G^\lambda(x,{\bf 1}) - \frac{E_x[\exp (\lambda (\tau_\Delta \wedge n))-1]}{e^{\lambda}-1}\right)
\end{align*}
The result now follows because $G^\lambda(x,{\bf 1}) = E_x [\exp (\lambda \tau_\Delta)-1]/(e^\lambda-1)$.
\end{proof} 
\begin{cor}
\label{cor:zetainfinity}
Let ${\bf x}=(x_n:n\in\N)$ satisfy $\lim_{n\to\infty} x_n = \infty$.  Then the distribution of $\tau_\Delta$ under $P_{K^\lambda(x_n,\cdot)}$ converges to $\mbox{Geom}(1-e^{-\lambda})$ if and only if $\lim_{n\to\infty} E_{x_n} [\exp (\lambda \tau_\Delta)]= \infty$. 
\end{cor}

\begin{proof} 
Suppose $\lim_{n\to\infty} E_{x_n} [\exp (\lambda \tau_\Delta)] =\infty$, then from \eqref{eq:zetaunderpnu},  $P_{K^\lambda (x_n,\cdot)} (\tau_\Delta>t)\to e^{-\lambda t}$ for every $t\in \Z_+$.  Otherwise, by switching to a subsequence, we may assume that  $\lim_{n\to\infty} E_{x_n} [\exp (\lambda \tau_\Delta)]$ exists and is finite. Denote this limit by $c$, and note that $c\ge e^\lambda>1$.   Since for every $t\in\N$, $E_{x_n} [\exp (\lambda( \tau_\Delta \wedge t))]-1\ge e^\lambda-1$, it follows that 
$$ \liminf_{n\to\infty} \frac{E_{x_n}[\exp (\lambda (\tau_\Delta\wedge t))]-1} {E_{x_n}[\exp (\lambda \tau_\Delta)]-1}\ge \frac{e^{\lambda}-1}{c}>0,$$ 
and so by \eqref{eq:zetaunderpnu}, 
$$ \limsup_{n\to\infty} P_{K^\lambda(x_n,\cdot)}(\tau_\Delta >t) \le (1- \frac{e^{\lambda}-1}{c}) e^{-\lambda t}.$$ 
\end{proof} 

\begin{prop}
\label{prop:tight}
\begin{enumerate}
    \item  Let ${\bf x} = (x_n:n\in\N)$ be a $\lambda,\infty$-convergent sequence. Then $\lim_{n\to\infty} K^\lambda (x_n ,\cdot)$ is a QSD if and only if $(K^\lambda (x_n,\cdot):n\in\N)$ is tight.
    \item  Moreover, under the equivalent conditions of part 1, $\lim_{n\to\infty} G^\lambda(x_n,{\bf 1})=\infty$. 
\end{enumerate}   
\end{prop}
\begin{proof}
We begin with the first assertion. Recall that $K^\lambda([{\bf x}],\cdot)$ denotes the pointwise limit $\lim_{n\to\infty} K^\lambda(x_n,\cdot)$.  Suppose that $K^\lambda ([{\bf x}],\cdot)$ is a QSD.  In particular,  it is a probability measure. This implies the tightness. Conversely, suppose the tightness holds.\tro{R1:M22} In particular $K^\lambda ([{\bf x}],\cdot)$ is a probability measure. Let $\epsilon>0$ and pick a finite subset $K_\epsilon$  of $S$ satisfying $K^\lambda(x_n,K_\epsilon) \ge 1-\epsilon$ for $n\in\N$ whenever the seqeunce  $(x_n:n\in\N)$ is  in the equivalence class $[{\bf x}]$. Then for $y\in S$, 
\begin{align*} (K^\lambda([{\bf x}],\cdot)p)(y) &\ge \sum_{z \in K_\epsilon}K^\lambda([{\bf x}],z)p(z,y) \\
& = \lim_{n\to\infty}\sum_{z \in K_\epsilon} K^\lambda(x_n,z) p (z,y) \\
& \ge \lim_{n\to\infty}(K^\lambda (x_n,\cdot)p)(y) - \epsilon\\
&= \lim_{n\to\infty} e^{-\lambda}K^\lambda (x_n,y) -\epsilon,
\end{align*}
where the last equality is due to Proposition \ref{prop:nux_prop}-(1)\tro{R1:G19}. As $\epsilon>0$ is abribtrary, $K^\lambda ([{\bf x}],\cdot) p \ge e^{-\lambda} K^\lambda([{\bf x}],\cdot)$. The reverse inequality also holds due to Corollary \ref{cor:Klambda}. Therefore $K^\lambda([{\bf x}],\cdot)$ is a QSD. 

It remains to prove the second assertion. The weak convergence of $K^\lambda(x_n,\cdot)$ to the QSD $K^\lambda([{\bf x}],\cdot)$ gives 
$$\lim_{n\to\infty} P_{K^\lambda(x_n,\cdot)}(\tau_\Delta > t) = P_{K^\lambda([{\bf x}],\cdot)}(\tau_\Delta > t),~t\in\Z_+.$$
As $K^\lambda({\bf x},\cdot)$ is a QSD with absorption parameter $\lambda$, the righthand side is $e^{-\lambda t}$, \eqref{eq:QSD_zeta_tail}. The conclusion now follows from Corollary \ref{cor:zetainfinity}. 
\end{proof}

We are ready to prove the theorem. 

\subsection{Proof of Theorem \ref{thm:QSD_tightness} and Corollary \ref{cor:upto_cr}}
\begin{proof}[Proof of Theorem \ref{thm:QSD_tightness}]
\tro{R1:M23} We prove the two assertions in the order of appearance. 
\begin{enumerate}  
\item  This key argument appeared in  \cite{Ferrari1995}. By assumption and Proposition \ref{prop:nux_prop}, 
\begin{align*} 
E_{K^\lambda (x_n, \cdot)}[\exp (\lambda' \tau_\Delta)]& \le  \sum_{t=0}^\infty e^{\lambda' t}P_{K^\lambda(x_n,\cdot)}(\tau_\Delta > t) \\
& \le \sum_{t=0}^\infty e^{(\lambda'-\lambda)t}\\
& = \frac{1}{1-e^{\lambda'-\lambda}}.
\end{align*}
We denote the quantity on the righthand side by $c$.  Now pick $\epsilon>0$ and a finite $K=K(\epsilon)\subseteq S$, such that\tro{R1:M24} $E_x [\exp (\lambda'\tau_\Delta)]>c/\epsilon$ for all $x\in K^c$.  The lefthand side is clearly bounded below by $K^\lambda (x_n,K^c) c/ \epsilon$, and therefore $K^\lambda (x_n, K^c) \le \epsilon$. Therefore the sequence $(K^\lambda(x_n,\cdot):n\in\N)$ is tight. Hence, Proposition \ref{prop:tight} guarantees that  $K^\lambda([{\bf x}],\cdot)$ is a QSD. 
\item \tro{R2:M7}\tro{R1:M25} From \eqref{eq:QSD_zeta_tail}, if $\nu$ is a QSD with parameter $\lambda$, $E_\nu [\exp (\lambda \tau_\Delta)]=\infty$. However, under the assumption, for any probability measure $\nu$ on $S$,  $E_\nu [\exp (\lambda \tau_\Delta)]<\infty$. \tro{R1:M26}\tro{R1:M27}
\end{enumerate} 
\end{proof}
\begin{proof}[Proof of Corollary \ref{cor:upto_cr}]
From Theorem \ref{thm:QSD_tightness}, there exists a QSD with absorption parameter $\lambda$ for all $\lambda \in (\lambda_0,\lambda_{cr})$.  What remains to be shown is  the existence of a minimal QSD. To do that, suppose $\lambda_0<\lambda_n \nearrow\lambda_{cr}$ and let $\nu_n$ be a QSD with absorption parameter $\lambda_n$, whose existence is guaranteed from the theorem. Without loss of generality, we may assume that the sequence $(\nu_n (\cdot):n\in\N)$ converges pointwise to some limit. Denote this limit by $\nu$. Fatou's lemma gives 
$$ \nu p \le e^{-\lambda_{cr}} \nu.$$
We turn to the reverse inequality. \tro{R1:P32}The first step is to show that $(\nu_n:n\in\N)$ is tight.  The key to proving the tightness rests on the fact that $\tau_\Delta \sim \mbox{Geom}(1-e^{-\lambda_n})$ under $P_{\nu_n}$. As the probability of success is increasing in $n$, for any fixed $\lambda\in(\lambda_0,\lambda_n)$, the sequence of  MGFs, evaluated at $\lambda$,  $E_{\nu_n}[\exp (\lambda \tau_\Delta)]$, is decreasing. 
Pick such $\lambda$. Then for any finite set $K\subset S$,  
\begin{equation} 
\label{eq:tight_inequ} E_{\nu_1} [\exp (\lambda \tau_\Delta)] \ge E_{\nu_n} [\exp (\lambda \tau_\Delta)] \ge \nu_n (K^c)  \inf_{x \in K^c} E_x [ \exp (\lambda \tau_\Delta)]. 
\end{equation}
The first term on the righthand side is a finite sum converging $\sum_{x\in K} \nu(x) p(x,y)$ and the second term on the righthand side is nonngetative and bounded above by $\nu(K^c)<\epsilon$. Therefore, by letting $n\to\infty$ we obtain 

$$ e^{-\lambda_{cr}} \nu \le \nu p + \epsilon,$$
and as $\epsilon>0$ is arbitrary, the result follows. 




\end{proof}
\subsection{Proof of Proposition \ref{prop:pzy} and Corollary \ref{cor:ratios}}
\tro{R1:M28}
Observe \tro{R1:P33}
\begin{equation}
\begin{split} 
\label{eq:ratioC}K^\lambda (x,y) &\overset{\eqref{eq:Kalpha}}{=}    \frac{e^\lambda-1}{E_x [\exp (\lambda \tau_\Delta)]-1}\times \frac{E_x [\exp (\lambda\tau_y),\tau_y <{\tau_\Delta}]}{1-E_y [\exp (\lambda 
 \tau_y),\tau_y<\tau_\Delta]} \\ 
 & \overset{\scriptsize\mbox{ Definition }\ref{defn:Clambda}}{=
 }\frac{C^\lambda (x,y)(e^{\lambda}-1)}{E_x [\exp (\lambda \tau_\Delta),\tau_y< \tau_\Delta]}\times \frac{E_x [\exp (\lambda\tau_y),\tau_y <{\tau_\Delta}]}{1-E_y [\exp (\lambda 
 \tau_y),\tau_y<\tau_\Delta]}\\
 & =  C^\lambda (x,y) \frac{e^\lambda -1}{E_y [\exp (\lambda\tau_\Delta)](1-E_y [ \exp (\lambda \tau_y),\tau_y< \tau_\Delta])}\\
 & =C^\lambda (x,y) \frac{e^{\lambda}-1}{E_y [\exp (\lambda \tau_\Delta),\tau_\Delta< \tau_y]},
 \end{split}
 \end{equation} 
 Where the third line is due to the equality $E_x [\exp (\lambda\tau_\Delta),\tau_y <{\tau_\Delta}]=E_x [ \exp (\lambda \tau_y),\tau_y < \tau_\Delta]E_y [\exp (\lambda \tau_\Delta)]$.\tro{R1:P34}
 
\begin{proof}[Proof of Proposition \ref{prop:pzy}]
Pick $y_0\in S$ and subsequence of $(x_n:n\in\N)$ with the properties $\lim_{n\to\infty} C^\lambda(x_n,y_0)>0$, $\lim_{n\to\infty} K^\lambda(x_n,\cdot)$ exists pointwise. Denote this limit by $K^\lambda ({\bf x},\cdot)$. Then \eqref{eq:ratioC} guarantees that $K^\lambda({\bf x},y_0)>0$ and - by irreducibility - $K^\lambda({\bf x},\cdot)$  is strictly positive. By Fatou's lemma $K^\lambda({\bf x},{\bf 1})\le 1$.  Next, for every $y\in S$, assumption \eqref{cond:pzy} allows to apply the dominated convergence theorem to conclude that
$$ (K^\lambda({\bf x},\cdot) p )(y)= \lim_{n\to\infty} (K^\lambda(x_n,\cdot)p) (y) = e^{-\lambda}K^\lambda({\bf x},y),$$
\tro{R1:M29}
where the second equality follows from Proposition \ref{prop:nux_prop}-1. \tro{R1:P35}Thus, the normalized kernel $K^\lambda({\bf x},\cdot)/ K^\lambda({\bf x},{\bf 1})$ is a QSD, proving the first assertion. 

We turn to the second assertion, which is a subcase of the first assertion. Let $ \nu $ be the QSD $\frac{K^\lambda({\bf x},\cdot)}{K^\lambda({\bf x},{\bf 1})}$. As  $K^\lambda({\bf x},{\bf 1}) \le 1$, we have that 
$$ \nu(y) \overset{\eqref{eq:ratioC}}\ge  \frac{e^\lambda-1}{E_y[\exp(\lambda \tau_\Delta),\tau_\Delta< \tau_y]}.$$
Corollary \ref{cor:domination} states that the expression on the righthand side is an upper bound on any QSD with absorption parameter $\lambda$, and therefore an equality must hold. The second part of the corollary gives the uniqueness. \tro{R1:P36}
\end{proof}
\begin{proof}[Proof of Corollary \ref{cor:ratios}]
Fix  $0<\lambda<\lambda_{cr}$. Clearly for every $x\not\in A$, $E_x [\exp(\lambda \tau_\Delta),\tau_A<\tau_\Delta]=E_x[\exp(\lambda \tau_\Delta)]$. But $E_x [\exp(\lambda \tau_\Delta),\tau_A < \tau_\Delta]\le \sum_{y\in A}E_x [ \exp(\lambda \tau_\Delta),\tau_y <\tau_\Delta]$, and therefore along any subsequence tending to infinity, there exists some $y$ such that $C^\lambda(\cdot,y) \ge \frac{1}{|A|}$  infinitely often, the result follows from Proposition \ref{prop:pzy}-(1). 
\end{proof}
\begin{proof}[Proof of Corollary \ref{cor:connection}]
Let $(x_n:n\in\N)\in [{\bf x}]$. First assume $K^\lambda ([{\bf x}],\cdot)$ is not identically zero. Pick $y\in S$ such that  $K^\lambda ([{\bf x}],y)>0$. Thus, \eqref{eq:ratioC} gives $\lim_{n\to\infty} C^\lambda(x_n,y)\in (0,\infty)$.

For the converse, since by assumption  $\lim_{n\to \infty} C^\lambda (x_n,y)>0$, \eqref{eq:ratioC} guarantees that $K^\lambda ([{\bf x}],y)>0$. In addition, Proposition \ref{prop:nux_prop} gives  
\begin{equation}
    \sum_{z\in S}K^\lambda(x_n,z)p(z,y)=e^{-\lambda}\left(K^\lambda(x_n,y)-\frac{\delta_{x_n}(y)}{G^\lambda(x_n,\bf 1)}\right).
\end{equation}
Since $K^\lambda(\cdot,\cdot)$ is nonnegative and bounded, the assumption \eqref{cond:pzy} allows to invoke dominated convergence to conclude that  $K^\lambda( [{\bf x}], \cdot) = e^{-\lambda} K^\lambda ([{\bf x}],\cdot)$.
Therefore, we proved that $K^\lambda ([{\bf x}],\cdot)$ is not identically zero, and satisfies \eqref{eq:eigen_description}.
\end{proof}

\section{Examples}
\label{sec:exmples}
In this section, we provide several simple applications of our results. 
\subsection{Downward Skip-Free Chains}
\label{sec:skipfree}
Consider a chain on $S=\Z_+$ and $\Delta = \{-1\}$, with the property that for every $x\in \Z_+$, and $l\in \{1,\dots, x+1\}$, we have $p(x,x-l)>0$ if and only if $l=1$. We will further assume that the chain satisfies Assumption  {\bf HD-\ref{as:inf},\ref{as:reg:moments},\ref{as:reg:Sirr}}. One notable case is of discrete time birth and death chains on $\Z_+\cup \{-1\}$ absorbed at $-1$. We will now show that for every $\lambda\in (0,\lambda_{cr}]$ there exists a unique QSD with absorption parameter $\lambda$, represented by the righthand side of \eqref{eq:dominated}.   

Indeed, since $\tau_\Delta>x$, Corollary \ref{cor:upto_cr} implies the existence of a QSD for every absorption parameter $\lambda \in (0,\lambda_{cr}]$. Moreover, Corollary  \ref{cor:Inu}-\ref{it:second_order} (with $\sigma$ taken as the identity) and Corollary \ref{cor:domination} guarantee that for each $\lambda$ in this range, there exists a unique QSD with absorption parameter $\lambda$, $\nu_\lambda$,  given by the righthand side of \eqref{eq:dominated}.  \tro{R2:M13}
\subsection{Generalized Cyclic Transfer}
\label{sec:xmpl_transfer}
This is a concrete example of a skip-free chain and probably the simplest closed-form example. This process generalizes the cyclic transfer process from  \cite{cyclic2020}. 

Assume $S=\Z_+$,  let $q \in (0,1)$ and $\mu$ be a probability distribution on $S$ with unbounded support (if $\mu$ has finite support, all derivation in this section hold verbatim with $S=\{0,1,\dots,\max \mbox{Supp}(\mu)\}$). For $x\in S$, consider the transition function $p$ on $S\cup\{\Delta\}$ illustrated in Figure \ref{fig:sbnd} and given by 
$$ \begin{cases} 
p(x,x-1) = 1 & x\in \N\\
p(0,\Delta) = q\\
p(0,x) = (1-q) \mu(x) & ~x\in S
\end{cases} 
$$ 
\begin{figure}[!ht]
\centering
\scalebox{0.9}{
  \begin{tikzpicture}[bullet/.append style={circle,inner sep=0.8ex},x=1.8cm,auto,bend angle=40]
 \draw[->, thick] (-1,0) -- (5.8,0);
 \path (-1,0) node[bullet, fill=white] (-3) {}
 (-1,0) node[bullet,fill=none] (100) {$\triangle$}
  (0,0) node[bullet,fill=white] (0) {$0$}
  (1,0) node[bullet,fill=white] (1) {$1$}
  (2,0) node[bullet,fill=white] (2) {$2$}
  (3,0) node[bullet, fill=white] (3) {$3$}
  (4,0) node[bullet, fill=white] (4) {$\cdots$}
  (5,0) node[bullet, fill=white] (5) {$x$};

 \draw[-{Stealth[bend]}] (1) to [bend right] node[pos=0.5, above]{$1$} (0) ;
 \draw[-{Stealth[bend]}] (2) to[bend right] node[pos=0.5, above]{$1$} (1);
  \draw[-{Stealth[bend]}] (3) to[bend right] node[pos=0.5, above]{$1$} (2);
    \draw[-{Stealth[bend]}] (4) to[bend right] node[pos=0.5, above]{$1$} (3);
     \draw[-{Stealth[bend]}] (5) to[bend right] node[pos=0.5, above]{$1$} (4);
   \draw[-{Stealth[bend]}] (0) to [bend left] node[pos=0.5, below]{$q$} (100);
    \draw[->]  (0) edge [in=-70, out=-110,looseness=10] node[pos=0.5, below]{} (0);
   \draw[-{Stealth[bend]}] (0) to[bend right] node[pos=0.5, below]{} (1);
\draw[-{Stealth[bend]}] (0) to[bend right] node[pos=0.5, below]{} (2);
   \draw[-{Stealth[bend]}] (0) to[bend right] node[pos=0.5, below]{} (3);
      \draw[-{Stealth[bend]}] (0) to[bend right] node[pos=0.5, below]{} (4);
      \draw[-{Stealth[bend]}] (0) to[bend right] node[pos=0.5, below]{$(1-q)\mu(x)$} (5);
\end{tikzpicture}}
\caption{Cyclic Transfer}
\label{fig:sbnd}
\end{figure}
Let $\varphi_\mu$ be the  moment generating function for $\mu$:
\begin{equation} 
\label{eq:varphimu}
\varphi_\mu (\lambda) = \sum_{j=0}^\infty \mu(j) e^{\lambda j}.
\end{equation} 
We will assume that $\varphi_\mu(\lambda)<\infty$ for some $\lambda>0$. 
Thus, {\bf HD-\ref{as:inf},\ref{as:reg:moments},\ref{as:reg:Sirr}} hold.  Let  $\lambda\in (0, \lambda_{cr})$. Observe that 
$$E_0[\exp(\lambda\tau_\Delta)]=e^\lambda q+e^\lambda(1-q)\varphi_\mu(\lambda)E_0[\exp(\lambda\tau_\Delta)].$$
Therefore 
\begin{equation}
    \label{eq:MGF_0}
    E_0[\exp(\lambda\tau_\Delta)]=\frac{e^\lambda q}{1-e^\lambda(1-q)\varphi_\mu(\lambda)}.
\end{equation}
Hence,  $\lambda_{cr}$ is the unique solution to 
\begin{equation} 
\label{eq:lambdacr_eqn} (1-q)e^\lambda \varphi_\mu(\lambda) =1.
\end{equation}
This implies that  $\lambda_{cr}$ is in the infinite MGF regime.

Next, since the process is downward-skip free, we will apply equation \eqref{eq:dominated} to obtain that for $\lambda \in (0,\lambda_{cr}]$, the  unique QSD with absorption parameter $\lambda>0$ is  given by
\begin{equation*}
    \nu_\lambda(y)=\frac{e^\lambda-1}{E_y[\exp(\lambda\tau_\Delta),\tau_\Delta<\tau_y]},
\end{equation*}
 To obtain an explicit formula, a similar calculation shows that for all $y\in\Z_+$  
$$E_0[\exp(\lambda \tau_\Delta),\tau_\Delta<\tau_y]=e^{\lambda} q+e^{\lambda}\sum_{0\leq j<y}(1-q)\mu(j)e^{\lambda j}E_0[\exp(\lambda\tau_\Delta),\tau_\Delta<\tau_y],$$
and therefore 
$$E_0[\exp(\lambda\tau_\Delta),\tau_\Delta<\tau_y]=\frac{e^{\lambda} q}{1-e^{\lambda}\sum_{j=0}^{y-1}(1-q)\mu(j)e^{\lambda j}}.$$
In addition, \tro{R1:P37}
\begin{align*}    E_y[\exp(\lambda\tau_\Delta),\tau_\Delta<\tau_y]&=e^{\lambda y}E_0[\exp(\lambda\tau_\Delta),\tau_\Delta<\tau_y].
\end{align*}
So
$$E_y[\exp(\lambda\tau_\Delta),\tau_\Delta<\tau_y]=\frac{e^{\lambda(y+1)}q}{1-e^\lambda\sum_{j=0}^{y-1}(1-q)\mu(j)e^{\lambda j}}.$$
Thus,
\begin{equation}
    \label{eq:transfer_QSD}
    \nu_\lambda(y)=\frac 1q e^{-\lambda(y+1)}(e^\lambda-1)(1-(1-q)\sum_{0\leq j<y}\mu(j)e^{\lambda (j+1)}).
\end{equation}
In summary, $\lambda_{cr}$ is given by \eqref{eq:lambdacr_eqn}, is in the infinite MGF regime and  for each $\lambda \in (0,\lambda_{cr}]$ there exists a unique QSD given by \eqref{eq:transfer_QSD}.  Note that the existence of all these QSDs is an immediate application of Corollary \ref{cor:upto_cr}.
\subsection{Absorption Probability Bounded from Below}
\label{sec:BD_UA}
Consider any process ${\bf Y}$ on $S\cup \{\Delta\}$ with transition function $p^Y$ satisfying Assumption {\bf HD-\ref{as:inf},\ref{as:reg:moments},\ref{as:reg:Sirr}}. We will also assume that for any $\lambda\in (0,\lambda_{cr}^Y)$, $\lim_{x\to\infty} E_x [\exp (\lambda \tau_\Delta^Y)]=\infty$.  Note that we have used the superscript $Y$ to denote quantities associated with $Y$, as we now introduce the process ${\bf X}$. 

Let $J$ be a geometric random variable, independent of ${\bf Y}$ with probability of success $1-e^{-\rho}$ for some $\rho>0$. Define ${\bf X}$ as follows: 
\begin{equation*}
X_n = \begin{cases} Y_n & n < J \\ \Delta & \mbox{otherwise}
\end{cases}
\end{equation*}
This is equivalent to defining $p(x,y) =e^{-\rho} p^Y(x,y)$ for $x,y \in S$, and $p(x,\Delta) = 1- \sum_{y\in S} p(x,y)$. Clearly, 
$$ P_x (\tau_\Delta > n) = P(\tau_\Delta^Y\wedge J>n)=P_x (\tau_\Delta> n) P(J>n)=P_x (\tau_\Delta >n) e^{-n\rho}.$$
Now since for every random variable $T$  which is nonnegative and taking integer values we have 
$$E[ \exp (\lambda T) ] = 1 + (e^\lambda -1) \sum_{n=0}^\infty e^{\lambda n}P(T>n),$$
it follows that 
\begin{align*} E_x [ \exp (\lambda \tau_\Delta) ]  &  = 1+ (e^{\lambda}-1)\sum_{n=0}^\infty e^{(\lambda-\rho)  n}P_x (\tau_\Delta^Y>n)\\
& = 1 + \frac{e^{\lambda}-1}{e^{\lambda - \rho}-1} (E_x [\exp ((\lambda-\rho) \tau_\Delta^Y)]-1).
\end{align*}
Therefore this expression is bounded as a function of $x$ if $\lambda < \rho$, and is equal to $1+(e^{\lambda}-1)E_x [ \tau_\Delta^Y]$ when $\lambda= \rho$, and tends to infinity as $x\to\infty$ when $\lambda> \rho$. Also, $\lambda_{cr} = \lambda_{cr}^Y+\rho$. It follows from Corollary \ref{cor:upto_cr} that ${\bf X}$ has QSDs with absorption parameter $\lambda$ for every $\lambda \in (\rho,\rho + \lambda_{cr}^Y]$, and that it does not possess any QSDs with absorption parameter $\lambda\in (0,\rho)$. 
\subsection{Subcritical Branching Process}
\label{sec:subcritical_branching}
\tro{R2:M14}
Let ${\bf X}$ be a branching process with a nondegenerate offspring distribution $B$ on $\Z_+$ (we will abuse notation and will refer to $B$ as a random variable). As usual \cite[p. 3]{Athreya1971} and to avoid trivialities we will assume $P(B=j)<1$ for all $j\in\N$ and $P(B\le 1)<1$. The unique absorbing state is $0$.  We will also assume that the process is subcritical, namely $E[B]\in (0,1)$, and let $m=E[B]$. A straightforward calculation shows that for any $x \in\N$,  $E_x [ e^{\lambda  \tau_\Delta}] < \infty$ if $e^{\lambda} < \frac{1}{m}$ and $E_x [e^{\lambda \tau_\Delta }] = \infty$ if $e^{\lambda }> \frac{1}{m}$. Therefore $e^{\lambda_{cr}} = \frac{1}{m}$, equivalently $\lambda_{cr} = \ln \frac 1m$. Though the restriction of ${\bf X}$ to the non-absorbing set $\N$ is not irreducible, we can restrict the process to an infinite subset of $\N$ depending on the support of $B$, resulting in an irreducible process. 

Our results provide a quick way to prove the existence of a continuum of QSDs. Indeed, for any $\lambda \in (0, \lambda_{cr})$, Jensen's inequality gives  $E_x [ \exp (\lambda \tau_\Delta) ] \ge e^{\lambda E_x [\tau_\Delta]}$. As $E_x [\tau_\Delta]$ is the expectation of the maximum of $x$ independent copies of $\tau_\Delta$ under the distribution $P_1$, it immediately follows that $\lim_{x\to\infty} E_x [ \tau_\Delta]=\infty$, and therefore Corollary \ref{cor:upto_cr} holds with $\lambda_0=0$. Namely, for every $\lambda \in (0,\lambda_{cr}= \ln \frac{1}{m}]$ there exists a QSD with absorption parameter $\lambda$. 

Existence and convergence results for a minimal QSD for the subcritical branching process are among the earliest in the field of QSDs.  Let $f$ be the generating function of $B$.  Yaglom's theorem \cite[Corollary 1, p. 18]{Athreya1971} states that for $x\in \N$, $P_x(X_n \in \cdot ~|~ \tau_\Delta > n)$ converges as $n\to\infty$ to a probability distribution on $\N$ which is the unique solution to the functional equation
\begin{equation} 
\label{eq:branching_generating} 
{\cal B} (f(s)) = m {\cal B}(s) + (1-m)
\end{equation}
among all probability distributions on $\N$. Being obtained as a quasi-limiting distribution, this limit is also a QSD. A straightforward calculation of the generating function for a solution to \eqref{eq:eigen_description} with $e^{-\lambda} = m$ reveals that it must satisfy \eqref{eq:branching_generating}, and so a minimal QSD exists and is unique. Denote this QSD by $\nu_{cr}$. 

\tro{R2:M15}As is well known, \cite[Corollary 2, p. 45]{Athreya1971}, the additional assumption $E[B \ln (1+B)]< \infty$ is equivalent to $\nu_{cr}$ having finite expectation, namely $\sum_{i=1}^\infty \nu_{cr} (i) i<\infty$. As the identity function $h(i)=i$ on $\N$ satisfies $ph = m h$, a straightforward application of the definition of the reverse chain associated with $\nu_{cr}$, \eqref{eq:q_def}, reveals that under this additional condition,  $\nu_{cr} h$ can be normalized to be a stationary distribution for $q$. Hence, $q$ is positive recurrent and proves the existence of a minimal QSD.  Proposition \ref{prop:reversal_probs} and the comment below it guarantee that $\lambda_{cr}$ is in the infinite regime and that \eqref{eq:finite_stopped}
 and \eqref{eq:positive_recurrent} hold. Moreover, if we take $S$ as the irreducible non-absorbing class mentioned above, then since \tro{R1:G20} $P(B=0)>0$, it follows that for every state $x \in\N$ in the support of $B$, $p(x,x)>0$, which along the irreducibility on $S$, implies that ${\bf X}$ is aperiodic. Thus, both the representation and the convergence results in Theorem \ref{th:nu_recurr} hold, and in particular, \eqref{eq:nu_recurr} provides us with a new MCMC method for sampling from the minimal QSD under this additional condition. We refer the reader to \cite{hautphenne}, which discusses numerical methods for solving \eqref{eq:branching_generating}.  

\subsection{Rooted Tree}
\label{xmpl:Rooted tree}
Consider an infinite rooted tree, with the root $\rho_r$ being the only state from which absorption is possible. We will assume $p$ is the nearest neighbor Markov Chain on this tree with a unique absorption state $\Delta$ satisfying Assumption {\bf  HD-\ref{as:inf},\ref{as:reg:moments},\ref{as:reg:Sirr}} and condition \tro{R1:G21}\eqref{cond:pzy}. For example, we can assume the degree of each vertex is bounded, and the transition from any vertex on the tree to any neighboring vertex on the tree is strictly positive, and for vertices other than the root, the transition to the unique vertex closer to the root is uniformly bounded below by $\frac 12+\epsilon$ for some $\epsilon>0$. 

Suppose $\lambda$ is in the finite regime (which may include $\lambda_{cr}$), and take a sequence of vertices $(x_n:n\in\N)$  going to infinity along some unique branch.  Recall $C^\lambda(x,y)$ from Definition \ref{defn:Clambda}. Two alternatives are illustrated by the following specific graph:
\begin{figure}[h]
    \centering
    \scalebox{1}{
\tikzset{every picture/.style={line width=0.75pt}} 

\begin{tikzpicture}[x=0.75pt,y=0.75pt,yscale=-1,xscale=1,every node/.append style={circle}]
\path (185,158) node [red, node contents=$\rho_r$]
       (138,142) node [black, node contents=$y_1$]
       (132,46)  node [black, node contents=$x_n$]
       (86, 102) node [black, node contents=$y_2$]
       (120,111.5) node [black, node contents=$y_0$]
       (220,123.5) node [black, node contents=$y'_2$]
        (237,157.5) node[isosceles triangle,draw,
    rotate=90, inner sep=0.4ex]{}
;
\draw   (191,152) -- (235,155.5) ;
\draw   [red, very thick, -] (147,141.5) node[fill=black, inner sep=0.4ex]{}  -- (191,152) node[fill=red, inner sep=0.4ex]{} ;
\draw    (191,152) -- (216,114.5)node[fill=black, inner sep=0.4ex]{} ;
\draw    (216,114.5)node[fill=black, inner sep=0.4ex]{} -- (262,113.5)node[fill=black, inner sep=0.4ex]{} ;
\draw    (206,72.5)node[fill=black, inner sep=0.4ex]{} -- (216,114.5)node[fill=black, inner sep=0.4ex]{} ;
\draw    (199,195.5)node[fill=black, inner sep=0.4ex]{} -- (244,219.5)node[fill=black, inner sep=0.4ex]{} ;
\draw    (169,234.5) node[fill=black, inner sep=0.4ex]{}-- (199,195.5)node[fill=black, inner sep=0.4ex]{} ;
\draw   [red, very thick, -] (123,103.5)node[fill=black, inner sep=0.4ex]{} -- (147,141.5) node[fill=black, inner sep=0.4ex]{};
\draw    (147,141.5)node[fill=black, inner sep=0.4ex]{} -- (122,173.5)node[fill=black, inner sep=0.4ex]{} ;
\draw    (191,152) -- (199,195.5) node[fill=black, inner sep=0.4ex]{};
\draw    (176,53.5)node[fill=black, inner sep=0.4ex]{} -- (206,72.5) node[fill=black, inner sep=0.4ex]{};
\draw    (206,72.5)node[fill=black, inner sep=0.4ex]{} -- (231,52.5) node[fill=black, inner sep=0.4ex]{};
\draw    (262,113.5)node[fill=black, inner sep=0.4ex]{} -- (276,143.5) node[fill=black, inner sep=0.4ex]{};
\draw    (262,113.5)node[fill=black, inner sep=0.4ex]{} -- (275,80.5) node[fill=black, inner sep=0.4ex]{};
\draw   [red, very thick, -] (134,71.5) node[fill=black, inner sep=0.4ex]{}-- (123,103.5)node[fill=black, inner sep=0.4ex]{} ;
\draw    (86,92.5) node[fill=black, inner sep=0.4ex]{}-- (123,103.5) node[fill=black, inner sep=0.4ex]{};
\draw    (122,173.5)node[fill=black, inner sep=0.4ex]{} -- (118,206.5) node[fill=black, inner sep=0.4ex]{};
\draw    (88,163.5) node[fill=black, inner sep=0.4ex]{}-- (122,173.5) node[fill=black, inner sep=0.4ex]{};
\draw    (169,234.5)node[fill=black, inner sep=0.4ex]{} -- (189,261.5) node[fill=black, inner sep=0.4ex]{};
\draw    (137,249.5) node[fill=black, inner sep=0.4ex]{}-- (169,234.5) node[fill=black, inner sep=0.4ex]{};
\draw    (244,219.5) node[fill=black, inner sep=0.4ex]{}-- (276,213.5) node[fill=black, inner sep=0.4ex]{};
\draw    (244,219.5)node[fill=black, inner sep=0.4ex]{} -- (242,251.5) node[fill=black, inner sep=0.4ex]{};
\draw    (159,45.5) -- (176,53.5) ;
\draw    (182,37.5) -- (176,53.5) ;
\draw    (231,52.5) -- (245,54.5) ;
\draw    (225,38.5) -- (231,52.5) ;
\draw    (275,80.5) -- (291,71.5) ;
\draw    (266,63.5) -- (275,80.5) ;
\draw   [red, very thick, -] (132,50.5) -- (134,71.5) ;
\draw    (114,63.5) -- (134,71.5) ;
\draw    (77,73.5) -- (86,92.5) ;
\draw    (70,102.5) -- (86,92.5) ;
\draw    (71,148.5) -- (88,163.5) ;
\draw    (88,163.5) -- (71,178.5) ;
\draw    (118,206.5) -- (122,229.5) ;
\draw    (93,212.5) -- (118,206.5) ;
\draw    (137,249.5) -- (149,271.5) ;
\draw    (116,255.5) -- (137,249.5) ;
\draw    (189,261.5) -- (210,267.5) ;
\draw    (189,261.5) -- (178,280.5) ;
\draw    (276,213.5) -- (279,190.5) ;
\draw    (276,213.5) -- (298,218.5) ;
\draw    (242,251.5) -- (266,257.5) ;
\draw    (242,251.5) -- (243,273.5) ;
\draw    (276,143.5) -- (296,135.5) ;
\draw    (276,143.5) -- (288,158.5) ;
\end{tikzpicture}
}
    \caption{Regular tree of degree 3 with a unique absorption state $\Delta$}
    \label{fig:regulartree}
\end{figure}
\begin{enumerate} 
\item For $y$ in that branch, we clearly have $\lim_{n\to\infty} C^\lambda (x_n,y)=1$. For instance, $y_1$ is on the branch in Figure \ref{fig:regulartree}. 
\item For other $y$,  we need to consider two cases. 
\begin{enumerate} 
\item $y$ is not on the branch, yet it has ancestors on the branch other than $\rho_r$. One example is the vertex $y_2$ in Figure \ref{fig:regulartree}. In this case, let $y_0$ denote the most recent ancestor of $y_2$ on the branch. For a path to get to $y_2$, it must pass through $y_0$. With this, 
$$C^\lambda (x_n,y)= \frac{E_{x_n} [\exp(\lambda\tau_{y_0}),\tau_{y_0} <\tau_\Delta] E_{y_0} [\exp(\lambda\tau_y),\tau_y < \tau_\Delta] E_y[\exp(\lambda\tau_\Delta)]}{E_{x_n}[\exp(\lambda\tau_\Delta)]-1}$$
$$\underset{n\to\infty}{\to} \frac{E_{y_0} [\exp(\lambda\tau_y),\tau_y < \tau_\Delta] E_y[\exp(\lambda\tau_\Delta)]}{E_{y_0}[\exp(\lambda\tau_\Delta)]}\in (0,1)$$
\item $y$ has no ancestor on the branch other than $\rho_r$, and so to get to $y$ from $x_n$ with $n$ large enough, we need to pass through the root where absorption is possible. For instance, in Figure \ref{fig:regulartree} $y_2'$ has no ancestor on the branch, and in order to go from $x_n$ to $y_2'$, we need to pass through $\rho_r$. Thus, 
\begin{align*}
    C^\lambda (x_n,y) &= \frac{E_{x_n} [\exp(\lambda\tau_{\rho_r})]E_{\rho_{r}} [\exp(\lambda\tau_\Delta),\tau_y < \tau_\Delta]}{E_{x_n} [\exp(\lambda\tau_{\rho_r})]E_{\rho_{r}}[\exp(\lambda\tau_\Delta)]-1}\\
    &\underset{n\to\infty}{\to} \frac{E_{\rho_{r}} [\exp(\lambda\tau_\Delta),\tau_y<\tau_\Delta]}{E_{\rho_{r}} [\exp(\lambda\tau_\Delta)]}\in (0,1)
\end{align*}
\end{enumerate} 
\end{enumerate}

Next, recall $K^{\lambda}(x_n,\cdot)$ from Definition \ref{def:KMartin}, we will also show that in the finite MGF regime, \tro{R1:M30} $\lim_{n\to \infty}K^\lambda(x_n,\cdot)$ exists pointwise and is a QSD.

Let $\bar y$ be the unique element on the same branch as $(x_n:n\in \N)$ satisfying $|\bar y|=|y|$, i.e., the lengths of the shortest paths from $\rho_{r}$  are equal. Since we are interested in the limit as $n\to \infty$, without loss of generality, we can assume $|x_n|>|y|$ for all $n$. To reach $y$ from $x_n$, the process must first hit $\bar y$. Therefore, \eqref{eq:Kalpha} and the strong Markov propperty give 

$$ K^\lambda (x_n, y ) = \frac{e^\lambda -1}{E_{x_n} [\exp (\lambda \tau_{\bar y})]E_{\bar y} [\exp (\lambda \tau_\Delta)]-1}\times \frac{E_{x_n} [\exp (\lambda \tau_{\bar y})]E_{\bar y} [\exp (\lambda ^0 \tau_y),^0 \tau_y < \tau_\Delta)]}{1-E_{y}[\exp (\lambda \tau_y),\tau_y < \tau_\Delta]}.$$
By taking $n\to\infty$,  we then obtain an asymptotic cancellation leading to 
\begin{equation}
\label{eq:branch} \lim_{n\to\infty} K^\lambda (x_n, y ) = \frac{e^\lambda -1}{E_{\bar y} [\exp (\lambda \tau_\Delta)]}\times \frac{E_{\bar y} [\exp (\lambda ^0 \tau_y),^0 \tau_y < \tau_\Delta)]}{1-E_{y}[\exp (\lambda \tau_y),\tau_y < \tau_\Delta]}.
\end{equation}
Therefore it follows from Corollary \ref{cor:connection} that  $(x_n:n\in\N)$ is $\lambda,\infty$-convergent and that  $K^\lambda([{\bf x}],\cdot)$ is a QSD with absorption parameter $\lambda$. In particular, the elements of $S^\lambda$ can be indexed by the branches of the tree. The QSD for each branch is given by the righthand side of \eqref{eq:branch}, where $y$ is any vertex and $\bar y$ is the unique vertex on the branch satisfying $|\bar y| = |y|$. 

\section{Results: Continuous-Time} 
\label{sec:QSD_cts}\tro{R2:M16}
\subsection{Definitions and Assumptions}
We adapt the main results of Section \ref{sec:disc_time} to the continuous-time setting. This adaptation is mostly straightforward and routine, and we present it primarily in order to make a connection with the literature in the continuous-time setting.

We begin by introducing the set of hypotheses. Let ${\bX}=(\X_t:t\in \R_+)$ be a Markov Chain on a state space which is a disjoint union $S\cup \{\Delta\}$, where $S$ is either finite or countably infinite. We will denote the distribution and expectation of $\bX$ under the initial distribution $\mu$ by $P_\mu$ and $E_\mu$ respectively, with $P_x$ and $E_x$ serving as shorthand for $P_{\delta_x}$ and $E_{\delta_x}$, respectively. For $x\in S \cup \{\Delta\}$, let 
\begin{equation}
    \ttau_x = \inf \{ t >0: \X_t =x , \X_{t-}\ne x\}.
\end{equation}



We will work under the following hypotheses, which are the analogs of the assumptions made for the discrete-time setting: 
\begin{assumc}
\label{asc:reg:uab}
 $\ttau_\Delta < \infty$, $P_x-a.s.$ for some $x\in S$. 
\end{assumc}
\begin{assumc}
\label{asc:reg:Sirr}
The set $S$ is an irreducible class, and the exponential holding time at each $x\in S$ has parameter $q_x \in (0,\infty)$. 
\end{assumc}
Clearly, $\Delta$ is the unique absorbing state, and therefore, we will refer to $\ttau_\Delta$ as the absorption time.  Note that {\bf HC-\ref{asc:reg:moments}} guarantees that $\bX$ is non-explosive. This is mostly for the simplicity of the presentation, as the explosion can be handled by the discretization scheme we use to derive the results below. We briefly review the notion of a QSD in a continuous time setting and some basic properties.  Recall that a probability distribution $\nu$ on $S$ is a QSD if the following analog of \eqref{eq:QSD_defn} holds. 
\begin{equation}
\label{eq:QSD_cts_defn}
    P_\nu (\X_t \in \cdot~ |~ \ttau_\Delta > t) = \nu (\cdot),~t\in \R_+,
\end{equation}
If $\nu$ is a QSD, then under $P_\nu$, $\ttau_\Delta$ is exponentially distributed with parameter $\lambda>0$.  That is, the following analog of \eqref{eq:QSD_zeta_tail} holds: 
\begin{equation}
    \label{eq:QSD_cts_tail}
    P_\nu (\ttau_\Delta > t ) = e^{-\lambda t},~ t \in \R_+,
\end{equation}
and we say that $\nu$ is a QSD with absorption parameter $\lambda$. \tro{R1:P38}A proof of  \eqref{eq:QSD_cts_tail} is included in the proof of the forward implication of Proposition \ref{prop:generator} below. In light of this necessary condition, we can replace  {\bf HC-\ref{asc:reg:uab}} with the following stronger hypothesis 
\begin{assumcp}
\label{asc:reg:moments}
There exists $\beta>0$ such that $E_x [\exp(\beta\ttau_\Delta)]<\infty$  for some $x\in S$.
\end{assumcp}
We write $({\cal P}_t:t\in\R_+)$ for the semigroup of contractions on $\ell^1(S)$ given by \tro{R1:P39}
$$ (\rho {\cal P}_t)(y) = \sum_{x \in S} \rho(x) P_x (\bX_t = y),~\rho \in \ell^1(S).$$
Hypothesis {\bf HC-\ref{asc:reg:Sirr}} implies that the semigroup is weakly continuous from the right at $0$. Namely, for any $\rho\in \ell^1(S)$ and $f \in \ell^\infty(S)$, $\lim_{t\downarrow 0} (\rho {\cal P}_t )f  = \lim_{t\downarrow 0} \sum_{x\in S}\rho(x) E_x [f (\bX_t)]=\sum_{x\in S}\rho(x) f(x)$, and therefore, \cite[Theorem 1.4, p. 44]{pazy}, \cite[Chapter X, Corollary of Theorem 10.2.3]{hille}, the semigroups is strongly continuous, that is for any $\rho\in \ell^1(S)$, $\lim_{t\downarrow 0} \rho {\cal P}_t= \rho$ in $\ell^1(S)$.\tro{R1:M31} As a result \cite[Corollary 2.5, p.5]{pazy}, the semigroup has a densely defined closed generator ${\cal L}$: 
\begin{equation}
    \label{eq:generator_def} 
\nu {\cal L} = \lim_{t\downarrow 0} \frac{\rho {\cal P}_t - \rho}{t},
\end{equation} 
where the limit is in $\ell^1(S)$. In particular, we have the following
\begin{prop} Let {\bf HC-\ref{asc:reg:moments},\ref{asc:reg:Sirr}} hold. 
\label{prop:generator}
    A probability measure $\nu$ on $S$ is a QSD for $\bX$ with absorption parameter $\lambda>0$ if and only if $\nu$ is in the domain of ${\cal L}$ and
    \begin{equation} 
    \label{eq:generator} \nu {\cal L} = -\lambda \nu.
    \end{equation}
\end{prop}
\tro{R1:M32}Though the proof is straightforward, we bring it here for completeness. We note that the equation \eqref{eq:generator} can be reexpressed as an infinite system of linear equations and note that as explicitly shown for a birth and death process in \cite[p. 691-692]{van1991}, the system may have solutions which are probability measures but are not QSDs (because they are not in the domain of the generator). 
\begin{proof}
Assume first that $\nu$ is a QSD. For $t\in \R_+$, letting $c_t = \sum_{y}(\nu {\cal P}_t)(y) = P_{\nu} (\ttau_\Delta>t)$, then  \eqref{eq:QSD_cts_defn} can be rewritten as 
$$\nu {\cal P}_t  = c_t \nu .$$ 
Thus, on the one hand, $\nu {\cal P}_{t+s} = c_{t+s} \nu$, while on the other hand, using the semigroup property, 
\tro{R1:M33}\begin{align*} \nu {\cal P}_{t+s} & = \nu {\cal P}_t {\cal P}_s  \\
&=c_t  \nu {\cal P}_s  \\
&=c_t c_s   \nu.
\end{align*}
As $\nu$ is strictly positive, it follows that the function $t\to c_t$ is multiplicative. It is equal to $1$ at zero and nonincreasing and tends to $0$ as $t\to\infty$. It is therefore of the form $e^{-\lambda t}$ for some $\lambda > 0$. As a result, $\nu {\cal P}_t = e^{-\lambda t} \nu$. This implies that $\nu$ is in the domain of ${\cal L}$ and that \eqref{eq:generator} holds. \tro{R1:P40}To prove the converse, suppose that $\nu$ is in the domain of ${\cal L}$ and satisfies \eqref{eq:generator}.  Then \cite[Section 2, Theorem 1.3]{pazy} gives that for any $f \in \ell^\infty (S)$,  
$\nu {\cal P}_t f  = \nu f  + \int_0^t \nu {\cal L} {\cal P}_s f ds $. Letting,  $\phi_f (t) = \nu {\cal P}_t f $ on $\R_+$, then $\phi_f$ is continous and satisfies  
$\phi_f (t) = \phi_f (0) - \lambda \int_0^t \phi_f (s) ds$. As a result,  $\phi_f (t) = e^{-\lambda t} \phi_f (0)$. As this is true for all $f \in \ell^\infty(S)$, we have $\nu {\cal P}_t = e^{-\lambda t} \nu$, which implies \eqref{eq:QSD_cts_defn}.  
\end{proof}
Next we introduce the critical parameter $\bblambda_{cr}$ and the finite and infinite MGF regimes analogously to the discrete-time case, Definition \ref{def:lambda_Delta_def}.\tro{R1:P41}
\begin{defn}
 \label{def:lambda_Delta_def_cts}
 Let {\bf HC-\ref{asc:reg:moments},\ref{asc:reg:Sirr}} hold. 
 \begin{enumerate} 
 \item The critical absorption parameter $\bblambda_{cr}$ is given by
\begin{equation} 
\bblambda_{cr} = \sup \{ \lambda> 0: E_x [ \exp (\lambda \ttau_\Delta)] < \infty\mbox{ for some }x\in S\}.
\end{equation}
\item A parameter $\lambda \in (0,\bblambda_{cr}]$ is in the finite MGF regime if $E_x [ \exp (\lambda \bbtau_\Delta)] < \infty$ for some $x\in S$.
\item The critical absorption parameter $\bblambda_{cr}$ is in the infinite MGF regime if 
$E_x [ \exp (\lambda_{cr} \bbtau_\Delta)]=\infty$ for some $x\in S$. 
\end{enumerate} 
\end{defn} 
\tro{R1:M34} {\bf HC-\ref{asc:reg:moments}} guarantees that $\bblambda_{cr}\ge \beta >0$.  Let $x\in S$. Then the distribution of $\ttau_\Delta$ under $P_x$ dominates the exponential holding time at $x$, an exponential random variable with parameter $q_x$. Therefore $E_x [ \exp (\lambda \ttau_\Delta)]=\infty$ if $\lambda \ge  q_x$, which implies  $\bblambda_{cr} \le \inf_{x\in S} q_x$.  Thus, $\bblambda_{cr} \in [\beta,\inf_{x\in S} q_x]$. A QSD with absorption parameter $\bblambda_{cr}$ is called the minimal QSD.


We record the following analog of Proposition \ref{prop:weird_values}, which we will need in the sequel. As the proof is identical to the proof in the discrete case, it will be omitted. 
\begin{prop}\label{prop:weird_values_cts}  Let {\bf HC-\ref{asc:reg:moments},\ref{asc:reg:Sirr}} hold and let 
 $\lambda\in (0, \bblambda_{cr}]$. Then for every $x\in S$,  $E_x [\exp (\lambda \tau_x),\bbtau_x<\bbtau_\Delta]\le 1$. Moreover,
\begin{enumerate} 
\item If $E_x [ \exp (\lambda \bbtau_\Delta) ] < \infty$ then the  inequality is strict; 
\item If $E_x [\exp(\bblambda_{cr} \tau_\Delta)]=\infty$ and $E_x[\exp(\bblambda_{cr} \bbtau_\Delta),\bbtau_\Delta<\bbtau_x]<\infty$ for some $x\in S$, then $E_x [\exp(\bblambda_{cr} \bbtau_x),\bbtau_x<\tau_\Delta]=1$ for all $x\in S$. 
\end{enumerate}
\end{prop}
For the remainder of Section \ref{sec:QSD_cts} we will assume that {\bf HC-\ref{as:reg:moments},\ref{as:reg:Sirr}} hold without additional reference. 
\subsection{Discretizing Time}
All results we present below use the embedded discrete-time processes we present in this section. For $d>0$, define a discrete-time Markov chain on $S\cup \{\Delta\}$,  ${\bf X}^d=(X^d_n:n\in\Z_+)$ by letting  
\begin{equation}
 X^d_n = \X_{dn}. 
\end{equation}
For $x\in S\cup \{\Delta\}$, let 
$$ \tau^d_x = \inf \{ n \in \N: X^d_n =x \}.$$
Clearly, ${\bf X}^d$ is a Markov chain satisfying {\bf HD-\ref{as:reg:moments}} and {\bf HD-\ref{as:reg:Sirr}}. Moreover, if $\tau_\Delta^d=n$, then on the one hand necessarily, $\X_{nd} = \Delta$, and so $\ttau_\Delta \le dn$, and on the other hand, $X^d_{n-1} \in S$ or, equivalently, $\X_{(n-1)d}\in S$, so $\ttau_\Delta> d(n-1)$. We therefore have that for every $x\in S$ and $d>0$, 
\begin{equation} 
\label{eq:59}
\ttau_\Delta \le d\tau^d_\Delta < \ttau_\Delta+d.
\end{equation}
From this, it follows that 
\begin{equation} 
\label{eq:exponontial_equivalency} 
E_x [\exp (\lambda \ttau_\Delta ) ]
\le E_x  [\exp (\lambda d \tau_\Delta^d) ]  \le 
e^{\lambda d} E_x [\exp (\lambda\ttau_\Delta)].
\end{equation}
and in particular letting $\lambda_{cr}^d$ denote the critical absorption parameter for ${\bf X}^d$, then   
\begin{equation}
\lambda_{cr}^d = d \bblambda_{cr}.
\end{equation}
Next we state the continuous-time analog of Proposition \ref{prop:K2single}. As the proof is identical, it will be omitted: 
\begin{prop}
\label{prop:K2single_cts}
Suppose that $K\subsetneq S$ is nonempty and finite and that for some $x \not \in K$,
\begin{equation}
    \label{eq:getK_cts}
 E_x [ \exp (\bblambda_{cr}  (\bbtau_\Delta\wedge \bbtau_K))] < \infty.
\end{equation}
Then there exists $z\in K$ so that 
$$E_{z} [ \exp (\bblambda_{cr} ( \bbtau_\Delta \wedge \bbtau_z))]<\infty.$$
\end{prop}
The following allows connecting hitting times of states other than $\Delta$ for the continuous and the discrete processes. 
\begin{prop}
\label{prop:also_disc_finite}
Let $x_0\in S$ and suppose that $E_{x_0} [\exp (\lambda  (\bbtau_\Delta\wedge\bbtau_{x_0}) )]<\infty$. Then there exists $d_0>0$ such that for $d \in (0,d_0)$, $E_{x_0} [\exp (\lambda d (\tau^d_\Delta \wedge \tau^d_{x_0} ))]<\infty$. 
\end{prop}
\begin{proof}
We begin by recalling that if $\bX$ starts from $x_0$, then the state where it jumps to first and the time it takes to perform the jump are independent. This can be done through the construction of the Markov chain. For simplicity, label the states other than $x_0$ by $1,2,\dots$. Let $T_1, T_2,\dots$ be independent exponential random variables with respective parameters $\rho_1,\rho_2,\dots$.  We assume $q=q_{x_0}=\sum_{j} \rho_j < \infty$.  Let $T=\inf T_j$.  Clearly, $T$ is exponential with parameter $q$. Also, let $R$ be the smallest (and unique, almost surely) index $j$ satisfying $T = T_j$. Then, the chain will jump to state $R$ at time $T$. The joint distribution of $R$ and $T$ is given by 
$$ P(R=j, T > t) = P(T_j > t , \cap_{i\ne j} \{T_i \ge T_j\})= \rho_j \int_t ^\infty e^{-\rho_j s} e^{-\sum_{i \ne j} \rho_i s }ds=\frac{\rho_j}{q}e^{-q t}.$$
Therefore, $R$ and $T$ are independent. In particular, $P(R=j | T \le d) = P(R=j)=\frac{\rho_j}{q}$. 

We turn to the main claim. Let $\bbtau_{x_0}^1= \bbtau_{x_0}$ and continue inductively, defining a  sequence $(\bbtau_{x_0}^k:k=1,2,\dots)$ by letting 
$$\bbtau_{x_0}^{k+1} = \inf \{ t > \bbtau_{x_0}^k : \X_{t-} \ne x_0,~ \X_t = x_0\}.$$
  Let $J= \inf \{ k\ge 1: \bbtau_{x_0}^k < \infty, \X_{ \bbtau_{x_0}^k+ s}= x_0,\mbox{ for all }s \in [0,d]\}$. In both definitions above, we adopt the convention $\inf \emptyset = \infty$. Note that $J$ is the first visit to $x_0$ after which the $\bX$ stays put for at least $d$ units, and recall that the distribution of a Markov Chain right after a jump from a given state is independent of the time it took to jump from the state. On the event $J<\infty$, $\bbtau^J_{x_0}$ is finite, and  on $J=\infty$, we define $\bbtau_{x_0}^J = \infty$. 
  We need a few more definitions. First, for $M>0$, let  $v_M(x_0) = E_{x_0} [\exp (\lambda (\bbtau^J_{x_0} \wedge \bbtau_\Delta\wedge M))]$.  Next, let
  \tro{R1:P42}$$W= E_{x_0} [\exp (\lambda (\bbtau_{x_0}\wedge \bbtau_\Delta)),\bbtau_{\Delta}< \bbtau_{x_0}] +  E_{x_0} [\exp (\lambda( \bbtau_{x_0}\wedge \bbtau_\Delta)),\bbtau_{x_0}<\bbtau_{\Delta},J=1].$$
  This expression is bounded above by $E_{x_0} [\exp (\lambda (\bbtau_{x_0}\wedge \bbtau_\Delta))]$ and is therefore finite by assumption.  We  have the upper bound 
\begin{equation}
\label{eq:breaking_vm} v_M(x_0) \le  W + E_{x_0} [ \exp (\lambda (\bbtau_{x_0}^J\wedge \bbtau_\Delta\wedge M)),\bbtau_{x_0}<\bbtau_\Delta,J>1].
\end{equation} 
We examine the second summand on the righthand side. \tro{R1:M35}By conditioning on the process up to time $\ttau_{x_0}$, and observing that the holding time at $x_0$ is exponential with parameter $x_0$, independent of the process up to time $\ttau_{x_0}$, it is bounded  above by 
$$ \us{=\eta_d}{E_{x_0} [\exp (\lambda \bbtau_{x_0}),\bbtau_{x_0}<\bbtau_\Delta)](1-e^{-q_{x_0}d})e^{\lambda d}} v_M(x_0).$$
Let $d_0$ be such that $\eta_d <1$ for $d\in (0,d_0)$.
  Thus,   $v_M (x_0) \le S + \eta_d  v_M(x_0)$, and since $v_M(x_0)$ is finite by construction, $v_M(x_0) \le \frac{W}{1-\eta_d}$. The righthand side is independent of $M$. By letting $M\to \infty$, it follows from monotone convergence that 
$$E_{x_0} [ \exp (\lambda (\bbtau_{x_0}^J \wedge \tau_\Delta)) ] \le \frac{W}{1-\eta_d}< \infty.$$
As it is always true that  $d \tau^d_\Delta \le \bbtau_\Delta + d$, on the event $J=\infty$, we clearly have $d (\tau^d_{x_0}\wedge \tau^d_\Delta) \le \bbtau_{x_0}^J \wedge \bbtau_\Delta + d$. On the event $J<\infty$, the definition of $\bbtau_{x_0}^J$ gives  $d\tau^d_{x_0}\le \bbtau_{x_0}^J + d$, so $d (\tau^d_{x_0}\wedge \tau^d_\Delta) \le \bbtau_{x_0}^J \wedge \bbtau_\Delta + d$ too.  This completes the proof. 
\end{proof} 
Finally, we provide a connection between the QSDs for $\bX$ and those for ${\bf X}^d$. 
\begin{prop}
\label{prop:QSD_equivalency}
Let $d>0$.  The process ${\bX}$ has a QSD with absorption parameter $\lambda$ if and only if ${\bf X}^d$ has a QSD with absorption parameter $d\lambda$. Specifically, 
\begin{enumerate}
\item
Let $\nu$ be a QSD for ${\bX}$ with absorption parameter $\lambda$. Then $\nu$ is a QSD for ${\bf X}^d$ with absorption parameter $d\lambda$. 
\item 
Conversely, let $\nu$ be a QSD for ${\bf X}^d$ with absorption parameter $d\lambda $. Then $\int_0^{d} e^{\lambda s} P_\nu (\X_s \in \cdot~ ) ds$ can be normalized to be a QSD for $\bX$ with absorption parameter $\lambda$.
\end{enumerate}
\end{prop}
\begin{proof}
 The first numbered assertion is trivial. As for the second, suppose that  $\nu$ is a QSD for ${\bf X}^d$ with absorption parameter $\lambda d$. For $s\in \R_+$ and $x\in S$, let  $h(s,x) =  e^{\lambda s} P_\nu (\X_s =x)$. Then 
 \begin{align*} h(s+d,x) &= e^{\lambda s} e^{\lambda d} P_\nu (\X_{s+d} =x ) \\
  &=e^{\lambda s} e^{\lambda d} E_\nu P_{X^d_1} (\X_s=x)\\
  & = e^{\lambda s} P_\nu (\X_s=x)= h(s,x).
 \end{align*}
 Note here that we pass from the second line to the third using the fact that $\nu$ is a QSD for ${\bf X}^d$, and so for any bounded $f$ on $S\cup\{\Delta\}$ vanishing on $\Delta$ we have 
 $e^{\lambda d} E_\nu [f (X^d_1)] =E_\nu [f(X^d_0)]$. Here we used $f(u) = P_u (\X_s = x)$. We proved that $h$ is $d$-periodic in the first variable. 
Define  $\tilde\nu(x)=\int_0^d h(s,x) ds$. Then $\tilde \nu$ is a finite measure on $S$. Also, for  $t\in \R_+$ and $ y\in S$
\begin{align*}
    \sum_x\tilde\nu(x)P_x(\X_t=y)&=
    \sum_x \int_0^d e^{\lambda s}P_\nu (X_s=x)P_x (X_t =y) ds \\
   & =  e^{-\lambda t}\int_0^d h(s+t,y)ds\\ 
    &\overset{u=s+t}{=}e^{-\lambda t}\int_t^{t+d}h(u,y)du\\
    &=e^{-\lambda t}\tilde\nu(y). 
\end{align*}\tro{R1:M36}
Thus, by normalizing $\tilde \nu$, we obtain a QSD with absorption parameter $\lambda$.  
\end{proof}
\subsection{Infinite MGF Regime}
We begin with the analog of Theorem \ref{th:nu_recurr}
\begin{thm}
    \label{thm:nu_recurr_cts}
    Suppose $\bblambda_{cr}$ is in the infinite MGF regime. Then 
    \begin{enumerate} 
    \item There exists a minimal QSD if and only if there exists $x \in S$ such that 
    \begin{equation} 
    \label{eq:taudelta_min} E_x [\exp (\bblambda_{cr} (\ttau_\Delta\wedge \ttau_x))] <\infty.
    \end{equation}
    In this case, there exists a unique minimal QSD $\bbnu_{cr}$, which is given by the formula
   \begin{equation} 
   \label{eq:cycle_rep}
   \begin{split}
   \bbnu_{cr}(x) &= \frac{\bblambda_{cr}}{q_x -\bblambda_{cr}}\frac{1}{E_x[ \exp (\bblambda_{cr} (\ttau_\Delta \wedge \ttau_x ) )]-1}\\
   & = \frac{\bblambda_{cr}}{q_x -\bblambda_{cr}}
   \frac{1}{E_x [\exp (\bblambda_{cr} \bbtau_\Delta),\bbtau_\Delta < \bbtau_x]},~x \in S.
   \end{split}
   \end{equation}
   \item If, in addition to \eqref{eq:taudelta_min}, 
   \begin{equation}
       \label{eq:positive_recurrent_cts}
       E_x[\exp(\bblambda_{cr}\ttau_x)\ttau_x,\ttau_x<\ttau_\Delta]<\infty\ \mbox{for some}\ x\in S,
   \end{equation}
   then for any finitely supported $\mu$ on S, 
   \begin{equation} 
\label{eq:conv_cts} \lim_{t\to\infty} P_\mu (\X_t \in \cdot~ |~ \ttau_\Delta>t) = \bbnu_{cr}. 
\end{equation} 
 \end{enumerate} 
\end{thm}
We wish to present another version of \eqref{eq:cycle_rep}. This requires additional notation. For $x\in S$ and $y \in S \cup\{\Delta\}-\{x\}$,  let $q_{x,y}$ denote the rate of jump from $x$ to $y$. Then $q_x = \sum_{y\ne x} q_{x,y}$. Consider $\bX$ under $P_x$, and let $J$ denote the $\mbox{Exp}(q_x)$-distributed time of the first jump and $D_x$ denote the probability distribution of $\bX$ on $S\cup\{\Delta\}-\{x\}$ immediately after the first jump. Then $D_x (y) = q_{x,y} /q_x$. Also, let $^0\bbtau_\Delta$ denote the hitting time of $\Delta$,
\begin{equation}
\label{eq:0hitting} 
^0 \bbtau_\Delta =\inf\{t\ge 0 : \bX_t = \Delta\}.
\end{equation}
Then the Strong Markov property allows to write 
\begin{align*} 
E_x [\exp (\bblambda_{cr} \bbtau_\Delta),\bbtau_\Delta < \bbtau_x] & = E_x [ \exp (\bblambda_{cr} J)]E_{D_x}[\exp (\bblambda_{cr}\, ^0\bbtau_\Delta),\,^0\bbtau_\Delta < \bbtau_x] \\ 
& = \frac{q_x}{\bblambda_{cr}-q_x}E_{D_x}[\exp (\bblambda_{cr} \,^0\bbtau_\Delta),\,^0\bbtau_\Delta < \bbtau_x].
\end{align*} 
We therefore proved the following:
\begin{cor}
\label{cor:alt_bb}
 An alternative representation for $\bbnu_{cr}$ from \eqref{eq:bbcycle} is 
 \begin{equation}
 \label{eq:afterjump_rep}\bbnu_{cr}(y) = \frac{\bblambda_{cr}}{q_x} \frac{1}{E_{D_x}[\exp (\bblambda_{cr} \,\,^0\bbtau_\Delta),\,^0\bbtau_\Delta < \bbtau_x]}.
 \end{equation} 
\end{cor}

\begin{proof}[Proof of Theorem \ref{thm:nu_recurr_cts}]
We prove the assertions in this order:  the sufficiency of \eqref{eq:taudelta_min} followed by its necessity,  then the uniqueness of a minimal QSD, the representation, and finally, the convergence. 

\underline {Sufficiency}. Suppose first that \eqref{eq:taudelta_min} holds. Apply  Proposition \ref{prop:also_disc_finite} with  $x_0=x$, and let  $d=d_0/2$, where $d_0$ is the positive constant obtained in the proposition.  Then ${\bf X}^d$ satisfies the conditions of Theorem \ref{th:nu_recurr}, and as a result possesses a unique minimal QSD, $\nu_{cr}$, given by \eqref{eq:nu_recurr}. Denote the minimal QSD for $\bX$ obtained from $\nu_{cr}$ through the application of Proposition \ref{prop:QSD_equivalency}-2 by $\bbnu_{cr}$. 

\underline{Necessity}.  Suppose that $\bbnu$ is a minimal QSD for $\bX$. Let $d>0$. Then from Proposition \ref{prop:QSD_equivalency}-1, $\bbnu$ is a minimal QSD for ${\bf X}^d$ and therefore \eqref{eq:finite_stopped} holds for ${\bf X}^d$. Thus,  $E_{x} [ \exp (\bblambda_{cr} d (\tau^d_{x} \wedge \tau^d_{\Delta}))] <\infty$. But since $\bbtau_{x}\wedge \bbtau_\Delta \le d(\tau^d_{x} \wedge \tau^d_\Delta)$, it follows that \eqref{eq:taudelta_min} holds. 

\underline{Uniqueness}.  Suppose that $\bbnu$ and $\bbnu'$ are minimal QSDs for $\bX$. Then by Proposition \ref{prop:QSD_equivalency}-1, both are also minimal for ${\bf X}^d$, and by the uniqueness of a minimal QSD for ${\bf X}^d$, it follows that $\bbnu=\bbnu'$. We will, therefore, denote the unique minimal QSD for $\bX$ by $\bbnu_{cr}$. 

\underline{Representation}. As argued above, when it exists, $\bbnu_{cr}$ is the unique minimal QSD for each of the processes ${\bf X}^{1/m}$ for all integer $m$ large enough. We will denote all quantities associated with ${\bf X}^{1/m}$ with the superscript $\frac{1}{m}$. Write $\mu_x^{1/m}$ for the measure defined in Proposition \ref{prop:mu_x} relative to the process ${\bf X}^{1/m}$  with \tro{R1:M37}$e^\lambda$ in the proposition taken as $e^{\bblambda_{cr} /m}=e^{\lambda_{cr}^{1/m}}$. Since  $\mu_x^{1/m} p^{1/m} = e^{-{\lambda_{cr}^{1/m}}} \mu_x^{1/m}$, it follows from Proposition \ref{prop:eigen_equivalence} and the uniqueness of the minimal QSD $\bbnu_{cr}$ that 
\begin{equation}
\label{eq:bbcycle}\bbnu_{cr}(\cdot) = \frac{\mu_x^{1/m}(\cdot)}{\mu_x^{1/m}({\bf 1})}.
\end{equation}
We will utilize this and a Riemann sum approximation to obtain the representation \eqref{eq:cycle_rep}. \tro{R1:M38}For a bounded function $f$ on $S$, define 
\begin{align*} 
I_x (f) = E_x [ \int_0^{\ttau_\Delta\wedge\ttau_x } \exp (\bblambda_{cr}s) f(\X_s) ds ]. 
\end{align*} 
Then, a Riemann sum approximation and the dominated convergence theorem give 
\begin{equation} 
\label{eq:Riemann} I_x (f) = \lim_{m\to\infty}
E_x [ \frac{1}{m}\sum_{0\le n /m \le \ttau_\Delta \wedge \ttau_x  } e^{\bblambda_{cr} n/m} f(\X_{n/m})]
\end{equation}
\tro{R1:M39}
Now, as  ${\bf X}^{1/m}$ is a snapshot of $\bX$ at discrete intervals, we have $m\ttau_\Delta\le \tau^{1/m}_\Delta$ and $m\ttau_x \le \tau^{1/m}_x$. Therefore if $n\le m (\ttau_\Delta \wedge \ttau_x)$ then $n\le (\tau^{1/m}_\Delta \wedge \tau^{1/m}_x)$. Letting $A_m=\{\tau^{1/m}_x<m\ttau_x+1\}$, we see that 
$$P_x (A_m)\ge e^{-q_x/m}=1-O(\frac1m).$$
As $\Delta$ is absorbing, $\tau^{1/m}_\Delta < m\tau_\Delta+1$ and on $A_m$, $\tau^{1/m}_x < m \ttau_x+1$. Therefore on $A_m$, if  $n \le \tau^{1/m}_x \wedge \tau^{1/m}_\Delta$, then necessarily $n< m(\ttau_\Delta\wedge \ttau_x)+1$, or $n/m\le \ttau_\Delta \wedge \ttau_x$. We proved that on $A_m$,  the expression inside the expectation in \eqref{eq:Riemann} can be rewritten as  
$$ \frac{1}{m}\sum_{0\le n  < \tau_x^{1/m}\wedge \tau_\Delta^{1/m}} e^{\lambda^{1/m}_{cr} n} f(X^{1/m}_{n})+\frac1m \exp (\bblambda_{cr} {\bbzeta_m} )f(\X_{\bbzeta_m}),$$
where $\bbzeta_m = \frac{ \lfloor m (\ttau_\Delta \wedge \ttau_x)\rfloor}{m}\le \ttau_\Delta\wedge\ttau_x$. 

Taking expectation on $A_m$,  this is equal to $\frac{1}{m}\mu_x^{1/m} (f)+O(\frac 1m)$, where the second term is due to the fact that $f$ is bounded and $E_x[\exp (\bblambda_{cr} (\ttau_\Delta\wedge \ttau_x))]<\infty$. From \eqref{eq:bbcycle}, this is also equal to $\frac{1}{m} \mu_x^{1/m}({\bf 1}) \bbnu_{cr}(f)+O(\frac1m)$. \tro{R1:M40}The boundedness of $f$ and the fact that   $E_x[\exp (\bblambda_{cr} (\ttau_\Delta\wedge \ttau_x))]<\infty$ also show  that  the expectation of the integral on the complement of $A_m$ tends to zero as $m\to\infty$. We proved
$$I_x (f) = \lim_{m\to\infty} \frac{1}{m} \mu_x^{1/m} ({\bf 1})\bbnu_{cr}(f).$$
Using this with $f\equiv {\bf 1}$, gives $ I_x ({\bf 1}) =\lim_{m\to\infty} \frac{1}{m} \mu_x^{1/m} ({\bf 1})$, and so \tro{R1:P44}
$$ \bbnu_{cr} (f) = \frac{ I_x (f)}{I_x ({\bf 1})}.$$
\tro{R1:M41} Let $J$ be exponential with parameter $q_x$. Then 
$$ I_x (\delta_x) = E[\int_0^J e^{\bblambda_{cr} s} ds ]= \frac{E [ e^{\bblambda J}-1]}{\bblambda_{cr}}=\left(\frac{q_x}{q_x - \bblambda_{cr}}-1\right)/\bblambda_{cr}=\frac{1}{q_x - \bblambda_{cr}}.$$
As $I_x ({\bf 1}) = \frac{E_x [ e^{\bblambda_{cr} (\ttau_\Delta \wedge \ttau_x)}-1]}{\bblambda_{cr}}$, we have 
$$ \bbnu_{cr}(x) = \frac{\bblambda_{cr}}{q_x-\bblambda_{cr}}\frac{1}{E_x [ e^{\bblambda_{cr} (\ttau_\Delta \wedge \ttau_x)}-1]}.$$
This gives the first expression in \eqref{eq:cycle_rep}. The second expression is obtained from  Proposition \ref{prop:weird_values_cts}-2. 
\end{proof}
Next, we consider an analog of Theorem \ref{thm:coming_infinty}. For continuous-time processes, Ferrari, Kesten, Martinez, and Picco showed the convergence based on a renewal technique \cite{Ferrari1995}. Martinez, San Martin, and Villemonais presented that the conditional distribution of the process converges exponentially fast in total variation norm to a unique QSD \cite{Villemonais2014}. Convergence in total variation for processes satisfying strong mixing conditions was obtained using Fleming-Viot particle systems \cite{cloez2014Fleming}.  More recently, Champagnat and Villemonais stated a general criterion for uniform exponential convergence in total variation for absorbed Markov processes conditioned to survive \cite[Assumption A]{champagnat2014exponential}. They also provide analogous conditions involving Lyapunov functions, tailored to the situation where the convergence is non-uniform \cite{EJPChampagnat}. Our work considers the time-reversal at the quasi-stationarity of the absorbed Markov process, similar to the approach by Tough \cite{tough2022linftyconvergence}. In particular, both Theorem \ref{thm:coming_infinty} and Theorem \ref{thm:coming_infinty_cts} were inspired by and should be viewed as weaker versions of the main result in \cite{Villemonais2014}, but with a focus on the representation of the QSD rather than convergence to it. We provide more details on the difference in the statement of the theorem below. 
\begin{thm}
\label{thm:coming_infinty_cts}
Suppose that there exists some $\bar \bblambda>0$ and a nonempty finite $K\subsetneq S$ 
\begin{align} 
\label{eq:blowlambda0_cts}
&E_x [ \exp (\bar\bblambda \tau_\Delta)]=\infty \mbox{ for some }x \in S;\mbox{ and }\\
\label{eq:bd_arrival_cts} 
&\sup_{x\not\in K} E_x [\exp (\bar\bblambda (\ttau_\Delta\wedge \ttau_K))]<\infty. 
\end{align}
Then all conditions and conclusions in Theorem \ref{thm:nu_recurr_cts} hold. In particular: 
\begin{enumerate} 
\item $\bblambda_{cr}\in (0,
\bar\bblambda]$ and is in the infinite MGF regime. 
\item There exists a unique QSD $\bbnu_{cr}$ given by \eqref{eq:cycle_rep}  which is also minimal. 
\end{enumerate} 
If, in addition, there exists some $x_0 \in S$ such that 
\begin{equation}
\label{eq:lowerh_cts}
\inf_{x\in S} \frac{P_x (\ttau_\Delta>\ttau_{x_0})}{P_x (\ttau_\Delta >1)}>0
\end{equation}
then \eqref{eq:conv_cts} holds for any initial distribution $\mu$. 
\end{thm}
 The main result of \cite{Villemonais2014} gives the existence, uniqueness, and exponential convergence with an explicit bound on the total variation norm. The authors of that paperwork work under the following set of assumptions. First, they assume that $S$ is countably infinite, $\Delta$ is a unique absorbing state, $P_x (\ttau_\Delta<\infty)=1$ for all $x\in S$, as well as the following:   
 \begin{enumerate} 
 \item[{\bf H1.}] There exists a finite non-empty $K\subsetneq S$ and a constant $c_1>0$ such that for all $t>0$, 
 $$ \inf_{x \in K} P_x (\ttau_\Delta>t) \le c_1 \sup_{x\in K} P_x (\ttau_\Delta>t).$$
\tro{R1:M42} \item[{\bf H2.}] There exists  $x_0\in K$ and constants $\lambda_0,c_2,c_3>0$ such that  $\sup_{x\in S} E_x [\exp (\lambda_0 (\ttau_K \wedge \ttau_\Delta))]\le c_2$ and $P_{x_0} (X_t \in K) \ge c_3 \exp (-\lambda_0 t)$ for all $t>0$. 
 \item[{\bf H3.}] There exists $x_0\in K$ and a constant $c_4>0$ such that $\inf_{x\in S} P_x (\X_1 = x_0~ | ~\ttau_\Delta > 1)\ge c_4$.
 \end{enumerate} 
 In part, {\bf H1} and {\bf H3} allow to extend the discussion to processes not satisfying our irreducibility condition {\bf HC-\ref{asc:reg:Sirr}}.  We will now show that when $S$ is irreducible, assumptions {\bf H2} and {\bf H3} imply all conditions of Theorem \ref{thm:coming_infinty_cts}. \tro{R1:M43}Indeed, the first condition in {\bf H2} implies \eqref{eq:bd_arrival_cts}
with $\bar\bblambda$ taken as $\lambda_0$, and the second condition in {\bf H2} implies $P_{x_0} (\bbtau_\Delta>t) \ge c_3 e^{-\lambda_0 t}$, which in turn implies \eqref{eq:blowlambda0_cts}, again with  $\bar \bblambda$ as $\lambda_0$. Since 
 $P_x (\ttau_\Delta > \ttau_{x_0} ) \ge P_x (\X_1 = x_0)$, condition {\bf H3} implies \eqref{eq:lowerh_cts}.  The existence and uniqueness in \cite{Villemonais2014} were established by obtaining uniform exponential bounds expressed in terms of the constants in {\bf H1},{\bf H2} and {\bf H3}, yet no formula for the QSD was obtained. In addition, it  relies on mathematical apparatus specifically developed to prove convergence. The focus of this work is on existence and representation, and our convergence result is obtained through an application of the ergodic theorem for positive-recurrent and aperiodic discrete-time Markov chains, which is applicable to the reverse chain from Section \ref{sec:ReverseTool}. 
\begin{proof}[Proof of Theorem \ref{thm:coming_infinty_cts}]\tro{R1:M44}
We begin with the first  number item. Of course,  \eqref{eq:blowlambda0_cts} clearly implies $\bblambda_{cr}\le \bar\bblambda$. We need to show that under the stated condition $\bblambda_{cr}$ is in the infinite MGF regime.   If $\bblambda_{cr} = \bar\bblambda$,  there is nothing to show. Otherwise,  we  obtain the desired result by invoking Proposition \ref{prop:highmomments}, noting that statement of the  proposition and its proof carry over to the continuous-time setting verbatim (this includes the statement and proof of Proposition \ref{prop:UI}). This completes the proof of the first numbered item.  

For the remainder of the proof we will utilize our discretization scheme.

We turn to proving the second numbered item. Let $d>0$. The transition function $p^d$ for the discretized process ${\bf X}^d$ is automatically irreducible and aperiodic.  As we have already shown that $\bblambda_{cr}$ is in the infinite MGF regime, \eqref{eq:exponontial_equivalency} implies that $\lambda_{cr}^d =d \bar\bblambda_{cr}$ and that 
$$E_x [ \exp (d\bblambda_{cr} \tau^d_\Delta)] = \infty.$$  

Applying Proposition \ref{prop:K2single_cts} followed by  Proposition \ref{prop:also_disc_finite}, we  conclude that there exists some $d_0>0$ such that for $d<d_0$, 
$$E_{z} [ \exp (d\bblambda_{cr} (\tau^d_\Delta \wedge \tau^d_z))]<\infty.$$

The last two displayed equations show that the  conditions in  Theorem \ref{thm:coming_infinty} hold for the process  ${\bf X}^d$, with the choices $\bar\lambda = d \bblambda_{cr}$ and $K=\{z\}$. In particular ${\bf X}^d$ has a unique QSD which is also minimal. The second  part of Proposition \ref{prop:QSD_equivalency} shows that $\bX$ has a minimal QSD. In light of Theorem \ref{thm:nu_recurr_cts}, this is a unique minimal QSD and it is given by \eqref{eq:cycle_rep}. We denote it by $\bbnu_{cr}$.  The first part of the proposition shows that $\bX$ does not possess other QSDs. This completes the proof of the second numbered item. 

We turn to final statement:  convergence. Without loss of generality, we will assume that $d\in (0,1)$. First we show  that  \eqref{eq:lowerh_cts} implies its discrete equivalent \eqref{eq:lowerh}. Consider any pathy starting from $x$ and satisfying $\bbtau_\Delta > \bbtau_{x_0}$ up to time $\bbtau_{x_0}$. If $q_{x_0}$ is the jump rate from $x_0$, then with probability bounded below by $e^{-q_{x_0} d}$, the path will remain in $x_0$ until the next integer multiple of $d$, in which case, the discretized process will be also at $x_0$. Therefore 
 $$P_x (\tau^d_\Delta>\tau^d_{x_0})\ge P_x (\bbtau_\Delta > \bbtau_{x_0})e^{-q_{x_0}d}.$$
 On the other hand, as $\bbtau_\Delta +d\overset{\eqref{eq:59}}{>}d \bbtau^d_\Delta$, it follows that 
 $$P_x (\bbtau_\Delta >1) =P_x (\bbtau_\Delta > 1+d-d)\ge P_x (d \tau^d_\Delta > 1+d )=P_x (\tau^d_\Delta > \frac {1}{d} + 1).$$
Thus letting $n_0$ the ceiling of $\frac{1}{d}+1$, and letting $c$ the ratio in \eqref{eq:lowerh_cts}, the last two displayed inequalities imply: 

\begin{align*} P_x (\tau^d_\Delta>\tau^d_{x_0})&\ge e^{-q_{x_0}d}  P_x (\bbtau_\Delta > \bbtau_{x_0})\\
& \ge c e^{-q_{x_0}d} P_x (\bbtau_\Delta >1) \\
& \ge c e^{-q_{x_0}d} P_x (\tau^d_\Delta > n_0).
\end{align*}
This gives \eqref{eq:lowerh} with allows to apply the convergence result in Theorem \ref{thm:coming_infinty} to the process ${\bf X}^d$. The first part of Proposition \ref{prop:QSD_equivalency} guarantees that $\bbnu_{cr}$ is the unique QSD for ${\bf X}^d$ for all $d$. As we are allowed to pick $d$ arbitrary small, we will assume $d=1/m$ for some integer $m>1$.  For every $t$, let \tro{R1:P45}$[t]_m = \lfloor tm \rfloor $, the corresponding ``time'' for the process ${\bf X}^{1/m}$. The discussion above gives that for any initial distribution $\mu$, 
$$ \lim_{t\to\infty} P_\mu (X^{1/m}_{[t]_m}=y | \tau^d_\Delta>t) = \bbnu_{cr}(y).$$
As argued above, if $\X_s=y$, then with probability bounded below by $e^{-q_y /m}$ $\bX$  will remain at $y$ during the entire interval $[s,s+1/m]$. Therefore, for every $t>0$, 
$$ P_\mu (X^{1/m}_{[t]_m}=y)e^{-q_y/m}  \le  P_\mu (\X_t =y).$$
Combining the last two displayed equations, we have \tro{R1:M45} 
\begin{equation}
\label{eq:use_discrete} 
P_\mu (\X_t =y) \ge e^{-q_y/m}P_\mu (X^{1/m}_{[t]_m}=y)\underset{t\to\infty}{\sim }e^{-q_y/m}\bbnu_{cr}(y) P_\mu (\tau^{1/m}_\Delta > [t]_m).
\end{equation}
\tro{R1:M46}Now $$\{\bbtau_\Delta > t\} \overset{\eqref{eq:59}}{\subseteq} \{ \frac{1}{m}\tau_\Delta^{1/m} > t\}=\{\tau_\Delta^{1/m}> tm\} \subseteq \{\tau_\Delta^{1/m}> [t]_m\},$$
so that $P_{\mu} (\bbtau_\Delta >t) \le P_\mu(\tau_\Delta^{1/m} > [t]_m)$. Using this in the righthand side of \eqref{eq:use_discrete} yields 
$$ \liminf_{t\to\infty} P_{\mu} (\X_t = y | \ttau_\Delta>t) \ge \bbnu_{cr}(y) e^{-q_y /m}.$$
Let $\nu_t = P_\mu (\X_t \in \cdot ~ | \ttau_\Delta> t)$. Since $m$ is arbitrary,  $\liminf_{t\to\infty} \nu_t (y) \ge \bbnu(y)$.  \tro{R1:G22}Fatou's lemma gives that for any  $A\subseteq S$, $\liminf_{t\to\infty}\nu_t (A) = \liminf_{t\to\infty} \sum_{x\in A}  \nu_t (x) \ge \bbnu(A)$. The result follows from Portmanteau's theorem. 
\end{proof}
\subsection{Finite MGF Regime}
It is clear that the finite-state $S$ is settled similarly to the discrete case, and therefore, throughout this section we will impose {\bf HD-\ref{as:inf}}, namely that $S$ is infinite. We begin with the analog of Theorem \ref{thm:QSD_tightness}. 
\begin{thm}
\label{thm:QSD_tightness_cts}
Let $\lambda>0$ be in the finite MGF regime. Then 
\begin{enumerate}
\item 
If for some $\lambda' \in (0,\lambda)$, $\lim_{x\to\infty} E_x [ \exp (\lambda'  \ttau_\Delta)] =\infty$ then there exists a QSD for $\bX$ with absorption parameter $\lambda$. 
\item If $\sup_x E_x [\exp (\lambda \ttau_\Delta)]<\infty$, then there does not exist a QSD for $\bX$ with absorption parameter $\lambda$. 
\end{enumerate}
\end{thm}
\begin{proof}
Fix some  $d>0$. We prove the two assertions in order of appearance:
\begin{enumerate}
    \item Let $0<\lambda'<\lambda$ \tro{R1:G23}satisfy $\lim_{x\to \infty}E_x[\exp(\lambda'\ttau_\Delta)]=\infty$. The first inequality in \eqref{eq:exponontial_equivalency} gives $\lim_{x\to \infty}E_x[\exp(\lambda' d\tau^d_{\Delta})]=\infty$ for the discrete process ${\bf X}^d$. By Theorem \ref{thm:QSD_tightness}-1, we conclude that ${\bf X}^d$ has a QSD with absorption parameter $\lambda d$, hence Proposition \ref{prop:QSD_equivalency} implies $\bX$ has a QSD with absorption parameter $\lambda$.
    \item Given $\sup_x E_x [\exp (\lambda \ttau_\Delta)]<\infty$, the second inequality in \eqref{eq:exponontial_equivalency} gives $\sup_x E_x[\exp(\lambda d\tau^d_\Delta)]<\infty$ for the discrete process ${\bf X}^d$. From Theorem \ref{thm:QSD_tightness}-2, we conclude that ${\bf X}^d$ does not have a QSD with absorption parameter $\lambda d$. Thus, from Proposition \ref{prop:QSD_equivalency}, we obtain that $\bX$ has no QSD with absorption parameter $\lambda$. 
\end{enumerate}
\end{proof}
\tro{R1:M47}As an immediate corollary, we also have the following analog of Corollary \ref{cor:upto_cr}:
\begin{cor}
\label{cor:upto_cr_cts}
Let 
$$\bblambda_ 0 = \inf\{ \lambda \in (0,\bblambda_{cr}): \lim_{x\to\infty} E_x [ \exp (\lambda \ttau_\Delta)]=\infty\},$$
with the convention $\inf \emptyset = \infty$. Then for every $\lambda \in (\bblambda_0,\bblambda_{cr}]$ there exists a QSD with absorption parameter $\lambda$.

\end{cor}
Several comments are in place: 
\begin{enumerate} 
\item 
A special case is the main result of  \cite{Ferrari1995}. There, the authors proved the existence of a QSD under the assumption that for all $t>0$, 
$$\lim_{x\to\infty} P_x (\ttau_\Delta \le t) \to 0.$$ 
This assumption implies  $\lim_{x\to\infty} E_x [\exp (\lambda' \ttau_\Delta)] = \infty$ for every $\lambda' > 0$, and therefore under this assumption $\bblambda$ in the Corollary is equal to $0$ and corollary yields the existence of a QSD for every absorption parameter in the interval $(0,\bblambda_{cr}]$. 
\item The MGFs $E_x [ \exp (\lambda \tau_\Delta)]$ may be unbounded, yet not have a limit at inifinity. The papers \cite{grois1} and \cite{grois2} provide several examples featuring such behavior. The processes described there possess a unique QSD and therefore the set in the corollary is empty. On the other hand, it can be easily shown that the MGFs are unbounded for any $\lambda\in (0,\bblambda_{cr})$, and that \eqref{eq:bd_arrival_cts} goes not hold for any $\bar\bblambda>0$, so that Theorem \ref{thm:coming_infinty_cts} is not applicable as well.  
\end{enumerate} 

\subsection{Martin Boundary} 
In this section, we provide a continuous-time version of Theorem \ref{thm:martin}. As in the previous sections, we adapt the results from the discrete setting.

Assume that $\lambda>0$ is in the finite MGF regime.  Fix any $d>0$. Then the process  ${\bf X}^d$ induces a Martin compactification as described in Definition \ref{Def:Martin}\tro{R1:P46}. We write $K^{d\lambda}$, $\partial^{d\lambda} M$ and $(M^{d\lambda},\rho^{d\lambda})$ for the corresponding Martin kernels, boundary and metric space introduced in that definition.  We also write $S^{d \lambda}$ for the elements in $\partial^{d\lambda} M$, which are QSDs for the transition function for ${\bf X}^d$ with absorption parameter $d \lambda$. 

We need some preparations. First, we introduce the analogs of the kernels $K^{d\lambda}(\cdot,\cdot)$.  For $x\in S$, define the kernel 
\begin{equation} 
\label{eq:kernel_cts} {\mathbb K}^\lambda (x,y) = \frac{ \int_0^\infty e^{\lambda s} P_{x} (\bX_s =y) ds }{\int_0^\infty e^{\lambda s} P_{x} (\bbtau_\Delta >s)ds}.
\end{equation} 
Note that by our assumption that $\lambda$ is in the finite MGF regime, both integrals are finite and nonzero. We make a connection with the kernels  $K^{d\lambda} (\cdot,\cdot)$. 
\begin{lem}
\label{lem:Kcts2disc}
$$ {\mathbb K}^\lambda (x,y) = \frac{\int_0^d e^{\lambda s} P_{K^{d\lambda}(x,\cdot)}(\bX_s = y)ds}{\int_0^d e^{\lambda s} P_{K^{d\lambda}(x,\cdot)}(\bbtau_\Delta>s) ds}.$$
\end{lem}
\begin{proof}
The numerator in \eqref{eq:kernel_cts} can be rewritten as \begin{align*} \sum_{m=0}^\infty e^{d\lambda m} \int_{0}^{d} e^{\lambda s} P_{x} (\bX_{dm+ s}=y) ds & = \sum_{m=0}^\infty e^{d \lambda m}  \int_0^d e^{\lambda s} E_{x}[ P_{X^d_m} (\bX_s =y)] ds \\
& = \sum_{m=0}^\infty \sum_z e^{d\lambda m} P_{x} (X^d_m = z)\int_0^d e^{\lambda s} P_z (\bX_s =y) ds\\
& = \sum_{z} G^{d \lambda}(x ,z) \int_0^d e^{\lambda s}P_z (\bX_s=y) ds\\
& = G^{d\lambda}(x,{\bf 1}) \int_0^d e^{\lambda s} P_{K^{d\lambda}(x,\cdot)}(\bX_s = y) ds
\end{align*} 
and similarly, the denominator is equal to 
$G^{d\lambda} (x,{\bf 1}) \int_0^d e^{\lambda s} P_{K^{d\lambda}(x,\cdot)}(\bbtau_\Delta>s) ds$. 
\end{proof}
\begin{prop}
\label{prop:ker_QSD_cts}
Let $[{\bf x}]\in S^{d\lambda}$. Then
\begin{enumerate} 
\item 
For every sequence $(x_n:n\in\N)$ of elements in $S$ which is in $[{\bf x}]$,
$$\lim_{n\to\infty} {\mathbb K}^{\lambda}(x_n,y)=
\frac{\int_0^d e^{\lambda s} P_{K^{d\lambda}([{\bf x}],\cdot)}(\bX_s = y)ds}{\int_0^d e^{\lambda s} P_{K^{d\lambda}([{\bf x}],\cdot)}(\bbtau_\Delta>s) ds},~y\in S.$$
Denote this limit by ${\mathbb K}^\lambda ([{\bf x}],\cdot)$. 
\item $y \to {\mathbb K}^\lambda ([{\bf x}],\cdot)$ is a QSD for $\bX$ with absorption parameter $\lambda$. 
\end{enumerate}
\end{prop}
\begin{proof}
The first statement follows from applying the dominated convergence theorem to the identity in Lemma \ref{lem:Kcts2disc}. Since $K^{d\lambda}([{\bf x}],\cdot)$ is a QSD for ${\bf X}^d$ with absorption parameter $d\lambda$, the second statement follows  from Proposition \ref{prop:QSD_equivalency}-2 and the fact that ${\mathbb K}^{\lambda}([{\bf x}],\cdot)$ is a probability measure on $S$.
\end{proof}
We are ready to state the analog of Theorem \ref{thm:martin}. 
\begin{thm}
\label{thm:martin_cts}
Let $\lambda>0$ be in the finite MGF regime for $\bX$.  Let $\nu$ be a QSD for $\bX$ with absorption parameter $\lambda$. Then, there exists a probability measure $\hat F_\nu$ on $\partial^{d\lambda} M$ satisfying $\hat F_{\nu} (S^{d\lambda})=1$ such that 
$$ \nu(y) = \int {\mathbb K}^\lambda ([{\bf x}],y)d\hat F_{\nu} ([{\bf x}]).$$
\end{thm} 
\begin{proof}
By Proposition \ref{prop:QSD_equivalency}-1, $\nu$ is a QSD for ${\bf X}^d$ with absorption parameter $d\lambda$. Theorem \ref{thm:martin} then gives a probability measure $\bar F_\nu$ with $\bar F_\nu (S^{d\lambda})=1$, satisfying 
$$ \nu (y) = \int K^{d\lambda} ([{\bf x}],y) d \bar F_\nu([\bf x]).$$
Since $\nu$ is a QSD for $\bX$ with absorption \tro{R1:G24}parameter $\lambda$, for every $s>0$, $e^{\lambda s} P_\nu (\bX_s =y) =\nu(y)$ and so $\nu (y) = \frac 1d\int_0^d e^{\lambda s} P_\nu (\bX_s=y) ds$. Plugging the first representation of $\nu$ into the righthand side of the second representation and applying the  Fubini-Tonelli theorem gives  
\begin{align*}  \nu(y) &= \int \left (\frac{1}{d} \int_0^d e^{\lambda s} P_{K^{d\lambda}([\bf x],\cdot)}(\bX_s=y) ds \right) d {\bar F}_\nu([{\bf x}])\\
&=\int {\mathbb K}^\lambda ([{\bf x}],y) \frac{\int_0^d e^{\lambda s} P_{K^{d\lambda}([{\bf x}],\cdot)}(\bbtau_\Delta > s) ds }{d}\bar F_{\nu} ([{\bf x}])\\
& = \int {\mathbb K}^\lambda ([{\bf x}],y) d \hat F_\nu ([{\bf x}]),
\end{align*} 
\tro{R1:M48}where the second equality is due to Lemma \ref{lem:Kcts2disc}, and 
where $\hat F_\nu$ is a measure absoultely continous with respect to $\bar F_\nu$, given by $\frac{d \hat F_\nu} {d \bar F_\nu}([{\bf x}])=\frac{\int_0^d e^{\lambda s} P_{K^{d\lambda}([{\bf x}],\cdot)}(\bbtau_\Delta > s) ds }{d}$. As $\nu$ and each  ${\mathbb K}^\lambda([{\bf x}],\cdot)$ are probability measures, it follows that $\hat F_\nu$ is a probability measure too. 
\end{proof}
\subsection{Example: QSDs for Birth and Death Process}
In \cite{van1991}, Van-Doorn obtained all QSDs for birth and death processes on $\Z_+$, which are eventually absorbed at $-1$. This was done through a very detailed analysis of a spectral representation for the transition kernel of the process.  Our work allows us to recover some of the main results using the general theory developed in earlier sections. We stress that the work in \cite{van1991} contains many additional results, which we will not cover here, specifically regarding convergence.

\tro{R1:P47}  In this section 
we assume that $\bX$ is a birth and death process on $\Z_+\cup\{-1\}$, with birth and death rates $(\lambda_k:k\in \Z_+)$ and  $(\mu_k:k\in \Z_+)$ respectively, which are all in $(0,\infty)$, and with  $\Delta=-1$ as a unique absorbing state. Note that in general a Birth and Death process may not satisfy  {\bf HC-1',2}. As a result, and in order to be consistent with the literature, in this section we will drop them as apriori assumptions. 

Letting 
$$ \pi_0=1,\pi_n = \prod_{j=1}^n \frac{\lambda_{j-1}}{\mu_j},~n\in\N,$$
then
\begin{equation}
\label{eq:taudelta_BD}
    \sum_{n=0}^\infty \frac{1}{\lambda_n \pi_n}=\infty,
\end{equation}
which is equivalent to $\ttau_\Delta<\infty$ a.s. from any initial distribution on $\Z_+$ \cite[Chapter 8]{anderson2012BDcts}. \tro{R1:M49}Note that this implies that the process does not explode.  We will assume that \eqref{eq:taudelta_BD} holds. Thus, {\bf HC-1,2} automatically hold. Next, we introduce an array of random variables that are crucial for the analysis. Let $x \in \Z_+$ and $y\in \Z_+\cup\{-1\}$ satisfying \tro{R1:M50}$y<x$ let $T_{x,y}$ be a random variable whose distribution is the same as $\tau_y$ under $P_x$. For each $y$, $T_{x,y} \preceq T_{x+1,y}\preceq \cdots$ and $T_{x,y+1}\preceq T_{x,y}$, where for two random variables $X$ and $Y$, $X\preceq Y$ means that $X$ is stochastically dominated by $Y$. Therefore, without loss of generality, we may assume all these RVs are realized in one probability space with the stochastic domination $\preceq$ replaced by a pointwise inequality $\le $. With this, let $SU_y =\lim_{x\to \infty} T_{x,y}$ and let \tro{R1:P48}$U=\sup U_y = U_{-1}$. This RV represents the passage time from $+\infty$ to $\Delta=-1$, and 
$$ E[U] = \sum_{n=1}^\infty \frac{1}{\lambda_n \pi_n}\sum_{i=n+1}^{\infty } \pi_i,$$
see \cite[Chapter 8]{anderson2012BDcts}.
In order to apply our results, we need the following lemma. 
\begin{lem}
\label{prop:Sexp}
 \begin{enumerate} 
 \item Suppose $E[U]<\infty$.  Then: 
 \begin{enumerate} 
 \item $\bblambda_{cr} \in (0,\infty)$.
 \item \tro{R1:M51}For $x\in \Z_+$ and $\lambda>0$: $E_x[\exp (\lambda \ttau_\Delta)]<\infty$ if and only if $E[\exp (\lambda U)] < \infty$ if and only if $\lambda < \bblambda_{cr}$. 
 \item For $x\in \Z_+$, $E_x [ \exp (\bblambda_{cr}(\ttau_\Delta \wedge \ttau_x))]< \infty$.
 \end{enumerate}
 \item Suppose $E[U]=\infty$. Then $\bblambda_{cr} =0$ or $\bblambda_{cr} >0$. In the latter case, $\lim_{x\to\infty} E_x[\exp (\lambda \ttau_\Delta)] =\infty$ for all $\lambda \le \bblambda_{cr}$. 
 \end{enumerate} 
\end{lem}  
\begin{proof}
 \begin{enumerate} 
\item 
 Suppose $E[U]<\infty$. Let $y\in \Z_+ \cup \{-1\}$, and define $h_y(t) = \sup_{x > y} P_x (\ttau_y > t)=P(U_y > t)$. \tro{R1:M52}From Markov's inequality,  $h_y(t) \le E[U_y] /t$. Since by construction $h_y(t+s) \le h_y (t) h_y (s)$, it follows from Fekete's lemma that $\lim_{t\to\infty} \frac{\ln h_y (t)}{t} = \inf_{t>0}\frac{\ln h_y(t)}{t}= -c_y$.  Thus, for every $t>0$
\begin{equation}
\label{eq:cylowbound}
 -c_y \le \frac{\ln h_y (t)}{t}\le  \frac{ \ln (E[U_y]/t)}t.
\end{equation} 
\tro{R1:P50}Note that $h_y(t) \ge  P(U_y > t) \ge P(T_{y+1,y}>t) = e^{- t (\mu_{y+1}+\lambda_{y+1})}$, and therefore $-c_y \ge -(\mu_{y+1}+\lambda_{y+1})$, and in particular, $c_y < \infty$. \tro{R1:P49}By choosing $t
>E[S_y]$, the righthand side of \eqref{eq:cylowbound} can be made strictly negative and therefore $c_y >0$. 

As $U\ge U_y =\sum_{x=y}^\infty T_{x+1,x}$, it follows from dominated convergence that $\lim_{y\to\infty} E[U_y]=0$ and therefore by freezing $t$ and taking $y\to \infty$ we conclude that $\lim_{y\to\infty} c_y = \infty$. From the defintion of $c_y$, we have that $E[\exp (\lambda U_y)]<\infty$ if $\lambda <c_y$ and $=\infty$ if $\lambda> c_y$. In particular, $E[\exp (\lambda U)]=E[\exp (\lambda U_{-1})]$ is finite for $\lambda < c_{-1}$ and infinite for $\lambda > c_{-1}$. In addition, 
$$E[\exp (\lambda U) ] =E[ \exp (\lambda U_y)] E_y [\exp (\lambda \ttau_\Delta)].$$
Consider now $\lambda>0$ such that $E_y [ \exp (\lambda \bbtau_\Delta)]<\infty$ for some, equivalently all, $y\in \Z_+$. Then by picking $y$ large enough, we may assume $\lambda< c_y$ and therefore both terms on the righthand side are finite, hence so is the lefthand side. Now if $\lambda>0$ is such that $E_y [ \exp (\lambda \bbtau_\Delta)]=\infty$ for some (equivalently all) $y\in\Z_+$, then necessarily the lefthand side is inifinte. This proves the first equivalence in (b). Since $E[\exp (\lambda U)]$  is finite if   $\lambda < c_{-1}$ and infinite if $\lambda> c_{-1}$, it follows that $\bblambda_{cr}=c_{-1}\in (0,\infty)$. This proves (a). \tro{R1:M53}To prove the second equivalence in (b) we apply Proposition \ref{prop:UI} which is also valid in the continuous-time setting with a change of notation. We argue by contradiction assuming $E[\exp (\bblambda_{cr} \ttau_\Delta)]<\infty$. By the first equivalence in part (b),  implies $E[\exp (\bblambda_{cr} S)]<\infty$. For each $y\in\Z_+$, we have that the distribution of $e^\bblambda_{cr}\bbtau_{\Delta}$ under $P_y$ coincides with the distribution of $e^{\bblambda_{cr}T_{y,-1}}$, a random variable which, by construction, satisfies  $e^{\bblambda_{cr}T_{y,-1}}\le e^{\bblambda_{cr} S}$. By the contradiction hypothesis the random variable on the righthand side is integrable, which in turn implies that the family of distributions of $e^{\bblambda_{cr} \bbtau_\Delta}$ under $P_y$ for $y\in\Z_+$, is uniformly integrable, a contradiction to Proposition \ref{prop:UI}. 

It remains to prove (c). We apply Proposition \ref{prop:zeta_taux}, which is also valid in the present setting with the obvious adaptations.  The second condition in the proposition holds because $\inf_{y\in\Z_+} P_y (\ttau_0 < \ttau_\Delta)=\frac{\lambda_0}{\lambda_0+\mu_0}>0$, and so (c) holds.  
\item Suppose $E[U]=\infty$. If $\bblambda_{cr}>0$ then mononote convergence gives that  for any $\lambda \in (0,\bblambda_{cr})$, $\lim_{x\to\infty} E_x [\exp (\lambda \ttau_\Delta) ] =E [\exp (\lambda U) ] =\infty$. By monotonicity, this limit holds for $\lambda=\bblambda_{cr}$. 
\end{enumerate} 
\end{proof}
With the lemma, we can prove the following characterization and description of QSDs for Birth and Death processes. This result is equivalent to  \cite[Theorem 3.2]{van1991}, which was the first to characterize and describe all QSDs for Birth and Death processes through spectral analysis of the transition kernels and corresponding orthogonal polynomials. 
\begin{thm}[Theorem 3.2, \cite{van1991}]
\begin{enumerate}
\item Suppose $E[U]<\infty$. Then $\bblambda_{cr}>0$, and there exists a unique QSD, which is also minimal. 
\item Suppose that $E[U]=\infty$. Then either $\bblambda_{cr}=0$ and there are no QSDs or $\bblambda_{cr}>0$ and for every $\lambda \in (0,\bblambda_{cr}]$ there exists a QSD with absorption parameter $\lambda$. 
\item When exists, a QSD with absorption parameter $\lambda>0$ is unique and given by the formula  
\begin{equation} 
\label{eq:rep_cts} \bbnu_\lambda (y) = \frac{\lambda}{q_y} \frac{1}{E_{D_y} [ \exp (\lambda \,^0\bbtau_\Delta),\,^0\bbtau_\Delta < \bbtau_y]},~y \in \Z_+ 
\end{equation}
where $D_y$ is the probability distribution assigning $\mu_y/q_y$ to $y-1$ and $\lambda_y/q_y$ to $y+1$ and $^0\bbtau_\Delta$ is the hitting time of $\Delta$ defined in \eqref{eq:0hitting}.
\end{enumerate}
\end{thm}
We comment that letting $y=0$ in \eqref{eq:rep_cts} gives 
$$ \nu_\lambda (0) = \frac{\lambda}{\mu_0},$$
because $E_{D_0} [ \exp (\lambda \,^0\bbtau_\Delta),\,^0\bbtau_\Delta < \bbtau_0]=\mu_0/q_0$, see also \cite[equation (3.4)]{van1991}. Thus, a necessary condition for the existence of a QSD with absorption parameter $\lambda$ is $\lambda < \mu_0$. As $\nu_\lambda(-1)=0$, these two initial values can be used to solve the system of difference equations resulting from \eqref{eq:generator}. 

We also comment that the argument leading to \eqref{eq:rep_cts} is valid for any chain which is downward skip-free and that a simple calculation shows that when $E[S]<\infty$ and $\lambda \in (0,\bblambda_{cr})$, the pointwise limit of $({\mathbb K}^{\lambda}(n,\cdot):n\in\N)$ along any convergent subsequence can be normalized to be a probability measure on $S$  which satisfies the system of difference equations resulting from \eqref{eq:generator} but is not a QSD.   
\begin{proof}
If $E[U]<\infty$, Lemma \ref{prop:Sexp}-1 guarantees that the conditions of Theorem \ref{thm:nu_recurr_cts} hold. This yields the existence and uniqueness of a minimal QSD. For $\lambda<\bblambda_{cr}$, $\sup_x E_x [\exp (\lambda \ttau_\Delta)]\le E[\exp (\lambda U)]<\infty$ and therefore Theorem \ref{thm:QSD_tightness_cts}-2 shows that no other QSDs exist. 

If $E[U]=\infty$ and $\bblambda_{cr}>0$, Lemma \ref{prop:Sexp}-2 and Corollary \ref{cor:upto_cr_cts} give the existence of QSDs for each of the absorption parameters in the range $(0,\bblambda_{cr}]$. The remaining case is $E[U]=\infty$ and $\bblambda_{cr}=0$. \tro{R1:M54}In this case the necessary condition for the existence of a QSD \eqref{eq:QSD_cts_tail} does not hold for any initial distribution, and therefore, no QSDs exist. 

It remains to establish the representation formula. The formula holds in the infinite MGF regime due to \eqref{eq:cycle_rep}. Suppose that  $\nu$ is a QSD with absorption parameter $\lambda$ in the finite MGF regime, then Theorem \ref{thm:martin_cts} implies that it is in the convex hull of ${\mathbb K}^{\lambda}([{\bf x}],\cdot)$ where $[{\bf x}]$ ranges over $S^{d\lambda}$. In particular, the latter is not empty. We will show that for any sequence $(x_n:n\in\N)$ satisfying $\lim_{n\to\infty} x_n = \infty$, ${\mathbb K}^{\lambda}(x_n, \cdot)$ converges pointwise to a limit independent of the sequence, given by the formula in the statement of the theorem. As by assumption $S^{d\lambda}$ is not empty, this guarantees that $S^{d\lambda}$ has a unique element equal to that limit. 

\tro{R1:M55}Indeed, by the strong Markov property and the definition of ${\mathbb K}^\lambda$ in \eqref{eq:kernel_cts}, 
$$ {\mathbb K}^\lambda (x_n , y) = \frac{E_{x_n} [\exp (\lambda \bbtau_y)] I_\lambda(y) }{E_{x_n} [ \int_0^{\bbtau_\Delta} e^{\lambda s} ds] },$$
where $I_\lambda(y)= \int_0^\infty e^{\lambda s} P_y (\bX_s = y) ds$. The denominator is equal to $\frac{1}{\lambda} \left ( E_{x_n} [\exp (\lambda \bbtau_\Delta)]-1\right)$. Lemma \ref{prop:Sexp}-2 gives that $\lim_{n\to\infty} E_{x_n} [\exp (\lambda \bbtau_\Delta )]=\infty$. By the strong Markov property, 
$E_{x_n} [\exp (\lambda \bbtau_\Delta )]=E_{x_n} [\exp (\lambda \bbtau_y)]E_y [\exp (\lambda \bbtau_\Delta)]$, and therefore the denominator is asymptotically equivalent to $E_{x_n} [\exp (\lambda \bbtau_y)] \lambda^{-1} E_y [\exp (\lambda \bbtau_\Delta)]$, resulting in 
$$ \lim_{n\to\infty}  {\mathbb K}^\lambda (x_n , y) = \frac{ \lambda I_\lambda(y) }{E_y [ \exp (\lambda \bbtau_\Delta)]}.$$
We evaluate $I_\lambda(y)$. Let $J=\inf\{t\in \R_+:{\bX}_t \ne \bX_{t-}\}$, the time of the first jump. Under $P_y$, $J\sim \mbox{Exp}(q_y)$. Breaking the integral in the definition of $I_\lambda (y)$  we have 
$$I_\lambda (y) = E_y [ \int_0^J e^{\lambda s}ds] +  E_y [ \exp (\lambda \bbtau_y),\bbtau_y < \bbtau_\Delta]I_\lambda (y),$$ 
and so $$ I_\lambda (y) = \frac{ E_y [\exp (\lambda J)]-1}{\lambda(1 - E_y [ \exp (\lambda \bbtau_y),\bbtau_y < \bbtau_\Delta])},$$
which in turn gives 
\begin{align*} \lim_{n\to\infty}  {\mathbb K}^\lambda (x_n , y) &= \frac{E_y [\exp (\lambda J)]-1}{E_y [ \exp (\lambda\bbtau_\Delta)](1 - E_y [ \exp (\lambda \bbtau_y),\bbtau_y < \bbtau_\Delta])} \\
& = \frac{\lambda}{q_y - \lambda}\frac{1}{E_y [ \exp (\lambda \ttau_\Delta ), \bbtau_\Delta < \bbtau_y]}.
\end{align*}
The formula given in the theorem is derived from this analogously to the derivation of Corollary \ref{cor:alt_bb}
\end{proof}

\section{Analysis of QSDs for a One-Parameter Family}
\label{sec:B_D}
\subsection{The model and $\lambda_{cr}$}
\label{sec:Process}
In this section, we study in detail the minimal QSDs for one-parameter family processes, all of which are special cases of the rooted tree of Section \ref{xmpl:Rooted tree}. This is among the easiest cases to study, yet exhibits some important features: 
\begin{itemize}
    \item $\lambda_{cr}$ may be in the infinite MGF regime or in the finite MGF regime. 
    \item If $\lambda_{cr}$ is in the infinite MGF regime, there exists a unique minimal QSD. 
    \item For $\lambda>0$ in the finite MGF regime, the set of QSD with absorption parameter $\lambda$ is the convex hull of two QSDs, corresponding to arriving from $+\infty$ and $-\infty$, respectively. 
\end{itemize}
 We note that the  latter case  is a minor generalization of the  ``hub and spokes" example of \cite{foley_mcdonald} that exhibits the same behavior. We will discuss the specifics later. The generalization is modifying the transition function in one state, allowing the process to stay there. 
 
The section is organized as follows. We begin by describing the model, calculating  $\lambda_{cr}$ and the respective MGF regime in Proposition \ref{prop:survival para}. We then record the existence and uniqueness results in the inifinite MGF regime in Proposition \ref{prop:minimal_null}. We then turn to the case of absorption parameters in the finite MGF regime, and calculate the respective two-dimensional Martin boundaries. Nearly all of the calculations provided are standard derivations of solutions to second-order difference equations. 

The model is essentially two birth and death chains glued at zero. Fix  $q \in (\frac 12,1)$, let  $\delta \in (0,q]$, and set $r=q-\delta$.  Consider the Markov chain ${\bf X} = (X_n:n\in \Z_+)$ on $ \Z\cup \{\Delta\} $ with transitions as in Figure \ref{fig:tree_example}. Let $\lambda_{cr}$ be the critical absorption parameter for ${\bf X}$, for $x\in \Z\cup \{\Delta\}$ define $\tau_x=\inf\{n\in \Z_+:X_n=x\}$ and write $P_x$ for the probability of ${\bf X}$. The ``hub and two spokes'' example of \cite[Section 5]{foley_mcdonald} is a special case of our example corresponding to the case $\delta =q$. We comment that this is the only value of the parameter $\delta$ for which the restriction of the transition function to $S$ is not irreducible. 

The main result of \cite{foley_mcdonald} on the ``hub and two spokes" model is the following. The sequence $(P_x (X_{n} \in \cdot | \tau_\Delta > n):n\in\Z_+)$ has two convergent subsequences: one along even $n$ and one along odd $n$, but not only that:  the limits depend on $x$. Explicit expressions for the limits are provided.  We  observed similar behavior for the our slightly more general model in the finite MGF regime, but as our work is primarily focused on existence and representation of QSDs, we omit the statement and the proof. 
\begin{figure}[!ht]
\centering
\scalebox{0.7}
{
  \begin{tikzpicture}[bullet/.append style={circle,inner sep=0.8ex},x=1.8cm,auto,bend angle=40]
 \draw[<->, thick] (-4.8,0) -- (4.8,0);
 \draw[-,thick] (0,0) -- (0,-1.1); 
 \path (-1,0) node[bullet, fill=white] (-3) {}
  (0,0) node[bullet,fill=white] (0) {$0$}
  (1,0) node[bullet,fill=white] (1) {$1$}
  (2,0) node[bullet,fill=white] (2) {$2$}
  (3,0) node[bullet, fill=white] (3) {$3$}
  (4,0) node[bullet, fill=white] (4) {$4$}
   (0,-1.4) node[fill=none] (100) {$\triangle$}
  (-1,0) node[bullet,fill=white] (-1) {$-1$}
  (-2,0) node[bullet,fill=white] (-2) {$-2$}
  (-3,0) node[bullet, fill=white] (-3) {$-3$}
  (-4,0) node[bullet, fill=white] (-4) {$-4$};

 \draw[red,thick,->]  (0) edge [in=70, out=110,looseness=12] node[pos=0.5, above]{$r$} (0);
 \draw[red,thick,->]  (100) edge [in=-70, out=-110,looseness=10] node[pos=0.5, below]{$1$} (100);
 \draw[-{Stealth[bend]}] (0) to [bend left] node[pos=0.5, above]{$\frac{1-q}{2}$} (1) ;
\draw[-{Stealth[bend]}] (0) to [bend left] node[pos=0.75, right]{$\delta$} (100);
 \draw[blue,-{Stealth[bend]}] (1) to[bend left] node[pos=0.5, below]{$q$} (0);
 \draw[blue, -{Stealth[bend]}] (1) to[bend left] node[pos=0.5, above]{$1-q$} (2);
  \draw[-{Stealth[bend]}] (2) to[bend left] node[pos=0.5, above]{$1-q$} (3);
    \draw[-{Stealth[bend]}] (3) to[bend left] node[pos=0.5, above]{$1-q$} (4);
   \draw[-{Stealth[bend]}] (2) to[bend left] node[pos=0.5, below]{$q$} (1);
\draw[-{Stealth[bend]}] (3) to[bend left] node[pos=0.5, below]{$q$} (2);
 \draw[ -{Stealth[bend]}] (0) to[bend right] node[pos=0.5, above]{$\frac{1-q}{2}$} (-1); 
 \draw[blue,-{Stealth[bend]}] (-1) to[bend right] node[pos=0.5, above]{$1-q$} (-2);
 \draw[-{Stealth[bend]}] (-2) to[bend right] node[pos=0.5, above]{$1-q$} (-3);
 \draw[-{Stealth[bend]}] (-3) to[bend right] node[pos=0.5, above]{$1-q$} (-4);
\draw[-{Stealth[bend]}] (-3) to[bend right] node[pos=0.5, below]{$q$} (-2);
 \draw[-{Stealth[bend]}] (-2) to[bend right] node[pos=0.5, below]{$q$} (-1);
 \draw[blue,-{Stealth[bend]}] (-1) to[bend right] node[pos=0.5, below]{$q$} (0);
\end{tikzpicture}
}
\caption{Transition Probabilities Diagram}
\label{fig:tree_example}
\end{figure}
We first examine the dependence of the critical absorption parameter $\lambda_{cr}$ on $\delta$. To do that, let $\lambda_0$ denote the critical absorption parameter for the system restricted to $\Z_+$ (or equivalently $-\Z_+$) and absorbed when hitting $0$. 
To continue, we introduce the following parameters. They will make some of the formulas simpler. First: 
$$ \rho =\sqrt \frac {1-q}{q}\in (0,1).$$
Next, we express $r$, the probability to transition from $0$ to itself as 
$$ r = \alpha \sqrt{q(1-q)}$$. As $r\in [0,q)$ it follows that $\alpha \in [0,\rho^{-1})$. Note then that $\delta$, the probability of absorption from $0$, is given by 
$$ \delta = q-r = \sqrt{q}(\sqrt{q}-\alpha \sqrt{1-q}).$$
We have the following: 
\begin{prop}
\label{prop:survival para}
\begin{enumerate} 
\item $e^{\lambda_0} = \frac{1}{2\sqrt{q(1-q)}}$. 
$$e^{\lambda_{cr}} = e^{\lambda_0}\times 
\begin{cases} 
1 & \alpha \in [0,1] \\  \frac{2}{\alpha + \alpha^{-1}}  & \alpha \in (1,\rho^{-1})
\end{cases}$$
Moreover, 
\begin{equation}
\label{eq:delta_vals}
\begin{array}{|c|c|c|c|}
\hline
  \alpha &   [0,1)  & 1 & (1,\rho^{-1}) \\
    \hline
    \lambda_{cr} & \multicolumn{2}{|c|}{= \lambda_{0} }  & < \lambda_0\\
    \hline 
     E_0[\exp (\lambda_{cr}\tau_\Delta)] & < \infty & \multicolumn{2}{|c|}{=\infty}  \\
    \hline
\end{array}
\end{equation}
\end{enumerate}
\end{prop}
\begin{center}
\begin{figure}[ht]
\centering
\label{fig:surv_prob}
\begin{tikzpicture}
\pgfmathsetmacro{\q}{0.95}
\pgfmathsetmacro{\ro}{((1-\q)/\q)^0.5}
\begin{axis}[
    axis lines = left,
    ylabel = {$e^{\lambda_{cr}},~q=\q$},
    xlabel = {$\alpha$},
    ymin=1,ymax=2.5,
    xmin=0,xmax=5]
\addplot[
domain=0:1,
samples=500,
color=blue]
{0.5/((\q*(1-\q))^0.5)};
\addplot[
domain=1:(\q/(1-\q))^0.5,
samples=500, 
color=blue]
{1/(((\q*(1-\q))^0.5)*(1/x + x))};
\addplot[dashed,
samples=500, 
color=red]
coordinates {(1, 0)(1, 2.5)};
\addplot[dashed, samples=500, color=red] coordinates {(1/\ro, 0)(1/\ro, 2.5)}; 
\end{axis}
\end{tikzpicture}
\caption{Absorption parameter dependence on $\alpha$ for $q=0.95$. The left vertical line corresponds to $r=\sqrt{q(1-q)}$ and the right one to $r\nearrow \rho^{-1} = 4.3589$. Here $e^{\lambda_0}\approx 2.29416$}
\end{figure}
\end{center}
\subsection{Infinite MGF regime} 
\begin{prop}
\label{prop:minimal_null}
   Suppose  $\alpha \in [ 1,\frac{1}{\rho})$. Then 
    \begin{enumerate} 
    \item $\lambda_{cr}$ is in the infinite MGF regime and condition \eqref{eq:finite_stopped} holds.  In particular, ${\bf X}$ has a unique minimal QSD given by 
\begin{equation}
\label{eq:eightyone}
       \begin{split}
    \nu_{cr}(y) &= \frac{e^{\lambda_{cr}}-1}{E_y[\exp(\lambda_{cr} \tau_\Delta),\tau_\Delta<\tau_y]}\\
    & = (1-\frac{\rho}{\alpha})\times \begin{cases} 1 & y=0 \\ \frac 12 (\frac\rho \alpha)^{|y|} & y \in \Z-\{0\}.\end{cases}
    \end{split}
    \end{equation}  
    \item Condition  \eqref{eq:positive_recurrent} holds if and only if $\alpha>1$.
    \end{enumerate}
\end{prop}
\subsection{Finite MGF regime}
We begin with the first observation on the finite MGF regime: 
\begin{cor}
\label{cor:twosided_subcrit}
Suppose $\lambda\in (0,\lambda_{cr}]$ is in the finite MGF regime. Then the set of QSDs with absorption parameter $\lambda$ is a two-dimensional convex cone spanned by $\nu^\lambda_+$ and $\nu^\lambda_-$ where for $y\in \Z$
\begin{equation}
\begin{split} 
\label{eq:QSDtwosided_gen}
\nu^{\lambda}_{\pm} (y) &= \lim_{n\to \pm \infty} K^\lambda (n,y)\\=
&\frac{e^\lambda -1}{E_y [ \exp (\lambda \tau_\Delta),\tau_\Delta< \tau_y]}\times \begin{cases} 1 & y \in \pm \Z_+ \\ E_{-y} [\exp (\lambda \tau_y),\tau_y < \tau_\Delta] & \mbox{otherwise} \end{cases} 
 \end{split}
\end{equation}
\end{cor} 
\begin{proof}
The model is special case of the random walk on a tree presented in Section \ref{xmpl:Rooted tree} and therefore the set of QSDs for $\lambda$ in the finite MGF regime is indexed by the two branches of the tree, both given by \eqref{eq:branch}. We identify the branch corresponding to sequences in $\Z$, $(x_n:n\in\N)$ with limit   $+\infty$ in the extended sense witn the $+$ sign, and the seqeuences with limit $-\infty$ with the $-$ sign. Specifically, the two QSDs in the righthand side of the formula are denoted by $\nu_{\pm}^\lambda$, with $\bar y$ given by $\pm |y|$, both according to the sign. To obtain the expression in the corollary for $\nu_{\pm}^\lambda$ from \eqref{eq:branch} we observe that due to the symmetry about $0$, the product of the denominators is euqal to 
$E_{y} [\exp (\lambda \tau_\Delta)]  - E_{y} [\exp (\lambda \tau_\Delta),\tau_{y}<\tau_\Delta] = E_y [\exp (\lambda \tau_\Delta),\tau_\Delta < \tau_y]$. Finally, note that if $y \in \pm \Z_+$, then $\bar y =y$, and otherwiser $\bar y = -y$. In the former case 
 $E_{\bar y}[\exp (\lambda ^0 \tau_y),^0\tau_y < \tau_\Delta]$ is equal to $1$ because then $^0 \tau_y=0$, while in the latter case, the $^0\tau_y = \tau_y$ and so the expression is equal to $E_{-y} [\exp (\lambda \tau_y),\tau_y < \tau_\Delta]$. 
\end{proof}
We wish to contrast \eqref{eq:QSDtwosided_gen} with the formula for the unique minimal QSD in the infinite MGF regime, the first line of \eqref{eq:eightyone}. To do that, observe  that by strong Markov property and the symmetry around $0$,
\begin{align*}
E_{-y} [ \exp (\lambda \tau_y),\tau_y <\tau_\Delta] &= E_{-y} [\exp (\lambda \tau_0)] E_0 [ \exp (\lambda \tau_{-y}),\tau_{-y}< \tau_\Delta]\\&<E_{-y} [ \exp (\lambda \tau_{-y}),\tau_{-y} < \tau_\Delta]\\
&\overset{\mbox{\scriptsize Prop. \ref{prop:weird_values}}} < 1.
\end{align*}
The following complements Proposition \ref{prop:minimal_null} by giving the minimal QSDs when $\lambda_{cr}$ is in the finite MGF regime. We write $y_+ = \max (0,y)$ and $y_- = \max (0,-y)$. 
\begin{prop}\tro{R1:M56}
\label{prop:QSDtwosided}
Let  $\alpha \in [0,1)$. Then $\lambda_{cr} = \lambda_0$ is in the finite MGF regime and the set of minimal QSDs is a two-dimensional convex cone spanned by 
    \begin{equation}
    \label{eq:QSDtwosided}
 \nu_{\pm}^{\lambda_0} (y) =\frac{(1-\rho)^2}{1-\rho \alpha}\rho^{|y|} \times  \begin{cases} 1 & y= 0 \\ 
 \frac 12 + (1-\alpha)y_{\pm}  & y\in \Z - \{0\}.
 \end{cases} 
    \end{equation}
\end{prop}
We conclude with the expression for the QSDs for the reamining absorption parameters. To do that, we need some notation.  For $\lambda \in (0,\lambda_0]$, let 
\begin{equation} 
\label{eq:thatkappa} \kappa = \kappa (\lambda ) = \sqrt{ 1-4e^{2\lambda}q(1-q)}=\sqrt{1-e^{2(\lambda-\lambda_0)}}\le 1.
\end{equation}
Next, let 
\begin{equation}
\label{eq:thatc}
\begin{split} c&=c(\lambda) = \frac{1-2 e^\lambda r }{1+\kappa -2e^\lambda r } \\
& = \frac{1 -\alpha e^{\lambda - \lambda_0}}{1+\kappa - \alpha e^{\lambda - \lambda_0}}\\
& = 1 - \frac{\kappa}{1+\kappa - \alpha e^{\lambda - \lambda_0}}.
\end{split}
\end{equation} 
\begin{prop}
\label{prop:QSDtwosided2}
 Let $\lambda \in (0,\lambda_{cr})$. Then the set of QSDs corresponding to the absorption parameter $\lambda$ is a two-dimensional convex cone spanned by 
    $$ \nu^{\lambda}_\pm(y) =
\frac{ e^\lambda -1 }{2e^\lambda \delta}\times   
\begin{cases} 
\frac{(\frac{1+\kappa}{2e^\lambda q})^{|y|} - c  (\frac{ 2e^\lambda(1-q)}{1+\kappa})^{|y|}}{1-c} & \pm y \in\N \\ 2 & y=0 \\(\frac{ 2e^\lambda(1-q)}{1+\kappa})^{|y|} & \pm y\in-\N.  \end{cases} $$
\end{prop}

\subsection{Calculations of Exponential Moments}
We begin with a sequence of useful and standard calculations that will be used to prove Proposition \ref{prop:survival para} identifying $\lambda_{cr}$ and the expressions for the QSDs given by  Propositions \ref{prop:minimal_null},\ref{prop:QSDtwosided} and \ref{prop:QSDtwosided2}. We comment that the expressions for the QSDs in each of the three propositions can be found directly by solving the respective second-order difference equation, \eqref{eq:eigen_description}. However, we present a different route based on calculating the expectations appearing in our formula for $\nu_{cr}$ in the infinite MGF regime and the one based on the Martin boundary representation from Corollary \ref{cor:twosided_subcrit} in the finite MGF regime. 

We first establish the exponential moments for the hitting time of $0$. 
\begin{lem} \tro{R1:M58}
\label{lem:tozero}
Let $\lambda \ge 0$ and let $a= a(\lambda)= E_1 [\exp (\lambda \tau_0)]\in (0,\infty]$. Then 
\begin{enumerate}
    \item $a(\lambda)<\infty$ if and only if $4e^{2\lambda} q(1-q) \le 1$. 
    \item Under the equivalent conditions of part 1,
\end{enumerate}
\begin{eqnarray} 
\label{eq:thata} 
&a(\lambda) = \frac{1-\kappa}{2e^{\lambda} (1-q)}=\frac{2e^\lambda q}{1+\kappa}\quad&\mbox{and }\\
\label{eq:tozero} &E_x [\exp (\lambda \tau_0),\tau_0 < \tau_\Delta] =&\begin{cases}  e^{\lambda}r + \frac{1-\kappa}{2} & x=0 \\ a^{|x|} & x\in \Z-\{0\}  \end{cases}
\end{eqnarray} 
\end{lem}
\begin{proof} 
 To find the value of $a$, condition on the first step to conclude that  $a = e^{\lambda} q + e^\lambda (1-q) a^2$. The solutions to this quadratic equation are 
$a_{\pm} = \frac{1\pm \kappa}{2e^{\lambda}(1-q)}, \mbox{ where } \kappa =\sqrt{1-4e^{2\lambda}q(1-q)}$.  \tro{R1:M59}Differentiation shows that $a_+$ is decreasing as a function of $\lambda$, and therefore $a$ has the form given in the lemma. To conclude the proof, the strong Markov property gives that for $x\ne 0$,    $E_x [ \exp (\lambda \tau_0)] =E_x [ \exp (\lambda \tau_0)]= a^{|x|}$.   
It remains to prove the expression for $E_0[ \exp (\lambda \tau_0)]$. Conditioning on the first step, we have $E_0 [\exp (\lambda \tau_0),\tau_0 <\tau_\Delta] = e^\lambda r +e^\lambda (1-q) a$.   
\end{proof}
Next we calculate the moment generating function of $\tau_\Delta$.
\begin{lem}
\label{lem:todelta}
Let $\lambda> 0$ and let $a$ be as in Lemma \ref{lem:tozero}. 
Then $E_0 [ \exp (\lambda \tau_\Delta)]<\infty$ if and only if $4e^{2\lambda} q(1-q)\le1$ and $1+\kappa - 2e^\lambda r>0$.   Under these conditions, for $x\in\Z$, 
$$ E_x [ \exp (\lambda \tau_\Delta)] = \frac{2e^{\lambda} \delta}{1+\kappa - 2e^\lambda r}\times a^{|x|}.$$
\end{lem}
\begin{proof}
For any $\lambda>0$, 
\begin{equation} 
\label{eq:tozero_derivation} E_0 [ \exp (\lambda \tau_\Delta)]  = e^\lambda \delta + e^{\lambda} r E_0 [\exp (\lambda \tau_\Delta)] + e^\lambda (1-q) a E_0 [\exp (\lambda \tau_\Delta)].
\end{equation}
It follows that $a(
\lambda)<\infty$ is a necessary condition for $E_0[\exp (\lambda \tau_\Delta)]<\infty$. By Lemma \ref{lem:tozero} this holds if and only if 
$4e^{2\lambda} q(1-q)\le1$.  Under this assumption, the equation \eqref{eq:tozero_derivation} has a finite solution if and only if  
$1- e^\lambda (1-q) a- e^\lambda r >0$ and using \eqref{eq:thata}, the expression on the left becomes $1+\kappa - 2e^\lambda r $. Thus, the necessary and sufficient condition for finiteness follows. Under this condition, the solution to \eqref{eq:tozero_derivation} is then 
$$E_0 [ \exp (\lambda \tau_\Delta) ]= \frac{e^{\lambda }\delta}{1- e^\lambda (1-q) a - e^\lambda r} =\frac{e^{\lambda} \delta}{\frac{1+\kappa}{2} - e^\lambda r}.$$
To complete the proof, let  $x\in \Z-\{0\}$. By the strong Markov property,  $E_x [ \exp(\lambda \tau_\Delta)] = E_x [ \exp (\lambda \tau_0)] E_0 [\exp (\lambda \tau_\Delta)]$. The result then  follows from lemma \ref{lem:tozero}. 
\end{proof}
This immediately leads to 
\begin{proof}[Proof of Proposition \ref{prop:survival para}]\tro{R1:M57}
Let $a=a(\lambda)$ be as in Lemma \ref{lem:tozero}. 
Clearly,  $\lambda_0= \sup\{ \lambda \ge 0: a(\lambda) <\infty\}$ and so the lemma gives  $e^{\lambda_0} = \frac{1}{2\sqrt{q(1-q)}}$.  

\tro{R1:M60}Under $P_1$, $\tau_0<\tau_\Delta$,and therefore $\lambda_{cr}\le \lambda_0$. Suppose then $0<\lambda < \lambda_0$. Note that this implies $\kappa>0$. From Lemma \ref{lem:todelta}, a necessary and sufficient condition for $E_0[\exp (\lambda \tau_\Delta)]<\infty$ is $\kappa > 2e^{\lambda} r-1$. As $r=\alpha \sqrt{q(1-q)}$, this condition is
\begin{equation} 
\label{eq:Deltabigger} 
\kappa > 2e^{\lambda} \alpha \sqrt{q(1-q)}-1.
\end{equation}
We continue according to the value of $\alpha$. 
\begin{itemize} 
\item $\alpha<1$. Then the righthand side of \eqref{eq:Deltabigger} strictly negative and  the inequality holds automatically, all the way up to and including  $\lambda =\lambda_0$. 
\item $\alpha = 1$.  Then the inequality holds if and only if $\kappa >0$, that is $\lambda< \lambda_0$. 
\item $\alpha>1$.  In this case the righthand side of \eqref{eq:Deltabigger} is positive and we can square both sides to rewrite it as 
$$1- 4e^{2\lambda}q(1-q) > 4e^{2\lambda}\alpha^2 q(1-q) - 4e^{\lambda}\alpha \sqrt{q(1-q)}+1.$$
Equivalently, $$e^{\lambda}\sqrt{q(1-q)} (\alpha^2+1)-\alpha <0$$
The solution to this equation is 
$$e^{\lambda}  < \frac{\alpha}{(\alpha^2+1)\sqrt{q(1-q)}}=e^{\lambda_0} \times \frac{2}{\alpha+\frac1\alpha}.$$
\end{itemize} 
\end{proof} 
We to exponential moments for times for moving ``away" from the absorbing state  $\Delta$. To this end, define 
\begin{equation} 
\label{eq:qx} q_x =q_x (\lambda) =   E_x [ \exp (\lambda \tau_{x+1}),\tau_{x+1}<\tau_\Delta],~x\in \Z_+.
\end{equation}

\begin{lem}
Suppose $\lambda \in (0,\lambda_{cr}]$. Then 
\begin{equation}
\label{eq:q0} 
q_0 =    \frac{e^\lambda  \frac{1-q}{2}}{1-e^\lambda r - \frac{1-\kappa}{4}}<\infty.
\end{equation}
\end{lem}
\begin{proof}
Suppose first $\lambda< \lambda_{cr}$. Then by considering the cases for the first step, we observe that $E_x [ \exp (\lambda \tau_\Delta)] \ge  e^\lambda q_x $, and therefor $q_x < \infty$. 
By conditioning on the first step, 
    \begin{align*} q_0 &=e^\lambda r q_0+ e^\lambda \frac{1-q}{2}\us{=a}{E_{-1} [\exp (\lambda \tau_0)]}q_0+ e^\lambda \frac{1-q}{2}\\
    &\overset{\eqref{eq:thata}}{=}e^{\lambda}\frac{1-q}{2} +q_0( e^\lambda r +\frac{1-\kappa}{4}),
    \end{align*}
    where $a$ is as in Lemma \ref{lem:tozero}. 
    This establishes \eqref{eq:q0} for $\lambda < \lambda_{cr}$. The formula for $\lambda =\lambda_{cr}$ follows from monotone convergence. As for the finiteness at $\lambda_{cr}$ consider the two cases:  
    \begin{itemize} 
    \item $\alpha>1$,  in which case $E_0 [ \exp (\lambda_{cr} \tau_\Delta)]<\infty$ and so $q_0<\infty$; or
    \item  $\alpha\le1$, in which case $\lambda_{cr}=\lambda_0$ and the denominator in \eqref{eq:q0} is equal to $1-e^{\lambda_0} r -\frac14 \ge \frac 34 - e^{\lambda_0}\sqrt{q(1-q)}=\frac 12>0$. 
    \end{itemize}
\end{proof}
\begin{lem}
\label{lem:forward}
 For   $\ell\in \Z_+\cup\{-1\}$, define 
$$r_\ell = 1-(\frac{1-\kappa}{1+\kappa})^{\ell+1} c .$$
Then for $x \in \Z_+$
\begin{align*} q_x &= \frac{1-\kappa}{2e^\lambda q}\frac{r_{x-1}}{r_x} \\
 & = \frac{1}{2e^\lambda q}(1+\kappa -\frac{2\kappa}{r_x} )\\
  & = \frac{1}{2e^\lambda q}\left (1+ \kappa - \frac{2\kappa}{1-(\frac{1-\kappa}{1+\kappa})^{x+1}c}\right).
\end{align*}
\end{lem}
\begin{proof}
Suppose first  $x\in\N$. Then
$$q_x = e^\lambda q q_{x-1} q_x  + e^\lambda (1-q).$$
Let $T$ be the M\"obius transform $T(z) = \frac{ e^{\lambda} (1-q) }{1-e^{\lambda }q z}$. Then $q_x = T(q_{x-1})$, and by induction, for $x\in \Z_+$, $q_x = T^{\circ x}(q_0)$.  A direct calculation shows that the fixed points of $T$ are 
$$z_\pm = \frac{1\pm \kappa}{2e^\lambda q},$$ and then the normal form of $T$ is given by 
$$ \frac{T(z)  -z_-}{T(z) - z_+} = k \frac{ z-z_-}{z-z_+},$$
where $k$ is a constant. By letting $z\to \infty$ we conclude that  $k =\frac{z_-}{z_+}$. By iterating, $$\frac{T^{\circ n} (z) - z_-}{T^{\circ n} (z) -z_+} =k^n \frac{z-z_-}{z-z_+}.$$
Take  $z=q_0$, and $n=x$ in the formula above  we find that 
$$\frac{q_x - z_- }{q_x- z_+} = (\frac{z_-}{z_+})^x \frac{q_0-z_-}{q_0-z_+},~x\in\Z_+.$$
Solving for $q_x$ we obtain 
\begin{equation}
\begin{split}
\label{eq:qx_simplified1} q_x &=\frac{ z_--(\frac{z_-}{z_+})^x \frac{q_0 - z_-}{q_0-z_+} z_+}{ 1-(\frac{z_-}{z_+})^x \frac{q_0 - z_-}{q_0-z_+}}\\
& = z_-  \frac{1-(\frac{z_-}{z_+})^{x-1} \frac{q_0 - z_-}{q_0-z_+}}{1-(\frac{z_-}{z_+})^{x} \frac{q_0 - z_-}{q_0-z_+}}.
\end{split}
 \end{equation} 
 We also have 
 \begin{equation}
 \begin{split}
     \label{eq:qx_simplified2} q_x &=\frac{ z_--(\frac{z_-}{z_+})^x \frac{q_0 - z_-}{q_0-z_+} z_+}{ 1-(\frac{z_-}{z_+})^x \frac{q_0 - z_-}{q_0-z_+}}\\
     &= \frac{ z_--z_++z_+-(\frac{z_-}{z_+})^x \frac{q_0 - z_-}{q_0-z_+} z_+}{ 1-(\frac{z_-}{z_+})^x \frac{q_0 - z_-}{q_0-z_+}}\\
& =  1-\frac{z_+-z_-}{ 1-(\frac{z_-}{z_+})^x \frac{q_0 - z_-}{q_0-z_+}}.
\end{split} 
 \end{equation} 
To express this in terms of the ``natural" parameters of the model and $\kappa$, observe 
$$ q_0 - z_\pm = \frac{e^\lambda \frac{1-q}{2}}{1-e^\lambda r - \frac{1-\kappa}{4}}-\frac{1\pm\kappa}{2e^\lambda q}.$$
and by taking common denominator, 
\begin{align*}  q_0 - z_\pm  &= C \left (  e^{2\lambda}q (1-q)- (1\pm\kappa)(1-\frac{1-\kappa}{4}-e^{\lambda }r)\right)\\ 
& =C\left (\frac{1-\kappa^2}{4} -(1\pm \kappa)+ \frac{(1\pm \kappa)(1-\kappa)}{4} + (1\pm \kappa)e^\lambda r \right)\\
&= C \begin{cases} (1-\kappa) (-\frac 12 + e^\lambda r)& z_- \\ 
 (1+\kappa)\left (-\frac 12(1+\kappa) + e^\lambda r\right) & z_+.
 \end{cases} 
\end{align*}
 Thus,  
 $$\frac{q_0-z_-}{q_0-z_+}= \frac{z_-}{z_+}\frac{1-2e^{\lambda} r}{1+\kappa-2e^{\lambda}r}.$$
 Plugging this into \eqref{eq:qx_simplified1} and \eqref{eq:qx_simplified2} respectively, gives the first and second representations in the lemma. 
 \end{proof}
 With the last lemma we have the following 
 \begin{cor}
 Suppose $\lambda \in (0, \lambda_{cr}]$. Then for $x\in\N$ 
     \begin{align*} E_0 [\exp (\lambda \tau_x),\tau_x < \tau_\Delta] &= (\frac{1-\kappa}{2e^\lambda q})^x \frac{r_{-1}}{r_{x-1}}
     \end{align*} 
 \end{cor}
 \begin{proof}
     As by the strong Markov property, $E_0 [ \exp (\lambda \tau_x),\tau_x < \tau_\Delta] = q_0 q_1 \cdots q_{x-1}$, it follows from the second representation of $q_x$ in Lemma \ref{lem:forward} that 
     $$E_0 [\exp (\lambda \tau_x),\tau_x < \tau_\Delta] = ( \frac{1-\kappa}{2e^\lambda q})^x \frac{r_{-1}}{r_0} \times \frac{r_0}{r_1} \times \cdots \times \frac{r_{x-2}}{r_{x-1}} =( \frac{1-\kappa}{2e^\lambda q})^x \frac{r_{-1}}{r_{x-1}}.$$
 \end{proof}
 \begin{cor}
 \label{cor:toself}
 Let $\lambda \in (0,\lambda_{cr}]$. Then 
     $$E_x [ \exp (\lambda \tau_x),\tau_x < \tau_\Delta]= \begin{cases} \frac{1-\kappa}{2} + e^\lambda r & x=0 \\  1 - \frac{\kappa}{r_{|x|-1}}& x \in \Z-\{0\}.\end{cases}$$
 \end{cor}
 \begin{proof}
 For $x\in\Z$, let  $v_x = v_x (\lambda) =  E_x [ \exp (\lambda \tau_x),\tau_x < \tau_\Delta]$. We first consider the case $x=0$. Conditioning on the first step, 
 $v_0 = e^\lambda (1-q) a  + e^\lambda r $, where here and henceforth $a$ is as in Lemma \ref{lem:tozero}. Therefore 
 $$v_0  =\frac{1-\kappa}{2} + e^{\lambda} r.$$
 As symmetry implies $v_x = v_{-x}$, we continue assuming $x\in\N$. Conditioning on the first step, $ v_x = e^\lambda (1-q) a + e^\lambda q q_{x-1}$.  Then Lemma  \ref{lem:forward} gives 
 $$v_x = \frac{1-\kappa}{2} + \frac{1+\kappa}{2} - \frac{\kappa}{r_{x-1}} =1-\frac{\kappa}{r_{x-1}}.$$
 \end{proof}
 Recall Green's function $G^\lambda$ from  Definition \ref{def:KMartin}.  
 \begin{cor}
 \label{cor:Glambdayy}
 Suppose $\lambda>0$ is in the finite MGF regime. Then 
 $$G^\lambda (y,y) =  \begin{cases} \frac{2}{1+\kappa-2e^\lambda r}=\frac{2r_{-1}}{\kappa} & y=0 \\\frac{r_{|y|-1}}{\kappa} & y\in \Z-\{0\}\end{cases}.$$
 \end{cor}
\begin{proof}
    This follows directly from Corollary \ref{cor:toself} as by conditioning on the first step $G^\lambda (y,y) = 1+ E_y [\exp (\lambda \tau_y),\tau_y <\infty]G^\lambda(y,y)$. 
\end{proof}
\subsection{Proof of Propositions \ref{prop:minimal_null},\ref{prop:QSDtwosided},\ref{prop:QSDtwosided2}}
We prove the propositions in reverse order. This is because the expressions for $\lambda \in (0,\lambda_{cr})$ are simpler, and allow to obtain the expressions for $\lambda_{cr}$ thruogh limits. 

\begin{proof}[Proof of Proposition \ref{prop:QSDtwosided2}]
We use Corollary \ref{cor:twosided_subcrit}. We will only obtain the expression for $\nu_+^\lambda$, with the derivation for $\nu_-^\lambda$ being identical due to symmetry. We continue to the calculation of $\nu_+^\lambda (y)$ according to the value of $y$. 
\begin{itemize}
\item {\bf  $y=0$}.  Then the expression we obtain for the corollary is 
$$\nu_+^\lambda (0) = \frac{e^\lambda -1}{e^\lambda\delta},$$
as the only way to realize the event $\tau_\Delta < \tau_0$ starting from $0$ is through getting absorbed in the first transition.
\item {\bf $y\in \Z-\{0\}$}. From the Strong Markov property, 
\begin{align*} E_y [ \exp (\lambda \tau_\Delta),\tau_\Delta < \tau_y] &= E_y [ \exp (\lambda \tau_\Delta)]-E_y [ \exp (\lambda \tau_y),\tau_y < \tau_\Delta]E_y [ \exp (\lambda \tau_\Delta)]\\
& = E_y [ \exp (\lambda \tau_\Delta)](1-E_y [ \exp(\lambda \tau_y),\tau_y < \tau_\Delta])\\
& \overset{\scriptsize\eqref{eq:Kalpha}}{=}  \frac{E_y [ \exp (\lambda \tau_\Delta)]}{G^\lambda (y,y)}.
\end{align*}
Therefore 
\begin{align*} \nu_+^\lambda (y) &= \frac{e^\lambda -1}{E_y [\exp (\lambda \tau_\Delta)]}G^\lambda (y,y) \times \begin{cases} 1 & y \in \N \\ E_y [  \exp(\lambda \tau_0)] E_0 [\exp (\lambda \tau_y),\tau_y<\tau_\Delta] & y \in -\N\end{cases}\\
& \overset{\scriptsize\mbox{Lem. \ref{lem:todelta}, Cor \ref{cor:Glambdayy}}}{=} \frac{e^\lambda -1}{2e^\lambda \delta}\frac{1+\kappa-2e^\lambda r}{\kappa} a^{-|y|} r_{|y|-1}\times \begin{cases} 1 & y \in \N \\ E_y [  \exp(\lambda \tau_0)] E_0 [\exp (\lambda \tau_y),\tau_y<\tau_\Delta] & y \in -\N\end{cases} \\
& \overset{\scriptsize\eqref{eq:thatc}\mbox{, Lem. \ref{lem:tozero}}}{=} \frac{e^\lambda-1}{2e^\lambda \delta}a^{-|y|}\frac{r_{|y|-1}}{1-c}\times \begin{cases} 1 & y \in \N \\   a^{|y|} (\frac{1-\kappa}{2e^\lambda q})^{|y|} \frac{1-c}{r_{|y|-1}} & y \in -\N\end{cases}
\end{align*}
To complete the calculation of the expression for$y\in\N$, observe that $$a^{-y} r_{y-1} \overset{\scriptsize\mbox{Lem. \ref{lem:forward}, \ref{lem:tozero}}}{=}  (\frac{2e^\lambda(1-q)}{1-\kappa})^y - c(\frac{2e^\lambda (1-q)}{1+\kappa})^y,$$
\end{itemize}
resulting in 
\begin{equation} 
\label{eq:intermediate_nuplus} 
\nu^\lambda_+ (y) = \frac{e^\lambda-1}{2e^\lambda \delta} \begin{cases}  \frac{(\frac{2e^\lambda(1-q)}{1-\kappa})^y - c(\frac{2e^\lambda (1-q)}{1+\kappa})^y}{1-c} & y \in\N\\  2 & y= 0\\ (\frac{1-\kappa}{2e^\lambda q})^{|y|} & y \in -\N. \end{cases}
\end{equation}
To obtain the formula in the proposition, use $(1-\kappa)(1+\kappa) = 4e^{2\lambda} q(1-q)$ to conlcude that 
$$ \frac{2e^\lambda (1-q)}{1-\kappa}=\frac{1+\kappa}{2e^\lambda q},$$
 and 
$$\frac{1- \kappa}{2e^\lambda q} = \frac{2e^\lambda (1-q)}{1+\kappa}.$$
\end{proof}
\begin{proof}[Proof of Proposition \ref{prop:QSDtwosided}]
Proposition \ref{prop:survival para} guaranteeds that $\lambda_{cr}=\lambda_0$ and is in the finite MGF regime. Therefore we can utilize the formula \eqref{eq:QSDtwosided_gen} from Corollarly \ref{cor:twosided_subcrit} for $\nu_\pm^{\lambda_0}$. We only derive the expression for $\nu_+^{\lambda_0}$, with the expression for $\nu_-^{\lambda_0}$ obtained through the symmetry of the model. 
Observe that all expressions on the righthand side of  \eqref{eq:QSDtwosided_gen} can be obtained from the respective expressions for $\lambda \in (0,\lambda_0]$, by taking the limit $\lambda \nearrow \lambda_0$ and the monotone convergence theorem. We will apply this procedure to the explicit expressions in Proposition \ref{prop:QSDtwosided2}. Observe, \eqref{eq:thatkappa}, that  $\lim_{\lambda \nearrow \lambda_0} \kappa = \kappa (\lambda_0)=0$, and, \eqref{eq:thatc}, $\lim_{\lambda \nearrow \lambda_0} c(\lambda) = 1$.  As the limits of the expression for  $\nu_+^{\lambda}(y)$ in \eqref{eq:intermediate_nuplus}  for $y\in -\N$ and for $y=0$ are trivial, we continue assuming $y\in\N$ and write the expression as a function of $\kappa$: 
$$\nu^\lambda_{+} (y) = \frac{e^\lambda -1}{2e^\lambda \delta} (2e^\lambda (1-q))^y\times \frac{ f_1(\kappa) - \bar c (\kappa) f_2 (\kappa)}{1-\bar c (\kappa)},$$
where $f_1(\kappa) = (1-\kappa)^{-y}$, $f_2 (\kappa) = (1+\kappa)^{-y}$ and $\bar c$ is $c$ as a function of $\kappa$. Note that the limits of $f_1,f_2$ and of $\bar c$ as $\kappa\searrow 0$, equivalently $\lambda \nearrow \lambda_0$,  are all $1$. Therefore,
$$ f_1 (\kappa) - \bar c (\kappa) f_2 (\kappa) = f_1 (\kappa) -1 +(1-\bar c (\kappa)) f_2 (\kappa) +1 -f_2(\kappa).$$

As  $1-\bar c (\kappa) = \frac{\kappa} {1+ \kappa -2e^\lambda r}$, it follows that 

$$ \frac{ f_1(\kappa) - \bar c (\kappa)  f_2 (\kappa)}{1-\bar c (\kappa)} = (1+\kappa - 2e^\lambda r) \times \left ( \frac{f_1 (\kappa) -1}{\kappa} - \frac{f_2 (\kappa)-1}{\kappa}\right) + f_2(\kappa).$$
By taking the limit $\kappa \searrow 0$, this converges to $2y (1-2e^{
\lambda_0} r ) +1$. We proved the following: 

$$ \nu^{\lambda_0}_+(y)=\frac{e^{\lambda_0}-1}{e^{\lambda_0}\delta}(2e^{\lambda_0}(1-q))^{|y|} \times\begin{cases} y (1-2e^{\lambda_0} r ) +\frac 12  & y>0 \\ 1 & y=0 \\ \frac 12  & y<0 \end{cases}$$
It remains to simplify this expression. We begin from left to right. First 
$$\frac{e^{\lambda_0}-1}{e^{\lambda_0}\delta } = \frac{1-e^{-\lambda_0}}{\delta}=\frac{1-2\sqrt{q(1-q)}}{q-\alpha \sqrt{q(1-q)}}=\frac{(\sqrt{q}-\sqrt{1-q})^2}{q(1- \alpha \rho)}=\frac{(1-\rho)^2}{1-\alpha \rho}.$$
Next, 
$$ 2e^{\lambda_0} (1-q) = \frac{1-q}{\sqrt{q(1-q)}}=\rho.$$
Finally, 
$$ 1-2e^{\lambda_0} r = 1-\frac{\alpha \sqrt{q(1-q)}}{\sqrt{q(1-q)}}=1-\alpha.$$
\end{proof}
\begin{proof}[Proof of Proposition \ref{prop:minimal_null}]\tro{R1:M61}
Here $\lambda_{cr}$ is in the infinite MGF regime per Proposition \ref{prop:survival para}. Next, 
\begin{equation} \label{eq:nucr_0}E_0 [ \exp (\lambda_{cr} \tau_\Delta),\tau_\Delta < \tau_0]=e^{\lambda_{cr}}\delta.
\end{equation}
Therefore from Proposition \ref{prop:weird_values}-2 we obtain that $E_0 [ \exp (\lambda_{cr} (\tau_\Delta \wedge \tau_0))]=e^{\lambda_{cr}} \delta+1 < \infty$, and therefore the necessary and sufficient condition \eqref{eq:finite_stopped} in Theorem \ref{th:nu_recurr} holds, establishing the existence and uniqueness of a minimal QSD and the first equality in \eqref{eq:eightyone}. Monotone convergence guarantees then that $\nu_{cr}(y) = \lim_{\lambda \nearrow \lambda_{cr}} \frac{e^\lambda -1}{E_y [ \exp (\lambda \tau_\Delta),\tau_\Delta < \tau_y]}$, and Corollary \ref{cor:twosided_subcrit} identifies the limit on the right  with $\lim_{\lambda\nearrow \lambda_{cr}} \nu_+^\lambda (|y|)$. Since $\lambda_{cr}<\lambda_0$, and $\lambda_{cr}$ is in the infinite MGF regime,  Lemma \ref{lem:todelta} gives $1+\kappa(\lambda_{cr})-2e^{\lambda_{cr}}r=0$. In particular \eqref{eq:thatc} gives $\lim_{\lambda nearrow \lambda_{cr}} c (\lambda) = -\infty$. Using these limits in Proposition \ref{prop:QSDtwosided2} gives 

$$ \nu_{cr}(y) = \lim_{\lambda\nearrow \lambda_{cr}} \nu_+^\lambda (|y|) = \frac{1-e^{-\lambda_{cr}}}{\delta} \times \begin{cases} \frac 12 (\frac{1-q}{r})^{|y|} & y \in \Z-\{0\} \\ 1 & y=0.\end{cases}$$

But $\frac{1-q}{r} = \frac\rho\alpha$, and the expression for $\frac{1-e^{-\lambda_{cr}}}{2\delta}$ follows from normalizing the geometric series. This expression can obtained by direct computation: 

$$ \frac{1- e^{-\lambda_{cr}}}{\delta} = \frac{1- \sqrt{q(1-q)}(\alpha+\alpha^{-1})}{q-\alpha \sqrt{q(1-q)}}=1 + \frac{1-q- \sqrt{q(1-q)}\alpha^{-1}}{q-\sqrt{q(1-q)}\alpha}=1+\frac{\rho\alpha^{-1}( \sqrt{q(1-q)}\alpha - q)}{q-\sqrt{q(1-q)}\alpha}=1-\frac\rho\alpha.$$

We turn to the second part. For $\lambda \in (0,\lambda_{cr})$, Corollary \ref{cor:toself} gives 
$$ E_0 [\exp (\lambda \tau_0)\tau_0,\tau_0 < \tau_\Delta] = \frac{d}{d\lambda} E_0 [ \exp (\lambda \tau_0),\tau_0 < \tau_\Delta] = -\frac{ \kappa'(\lambda)}{2} + e^\lambda r.$$
From monotone convergene, we can take the limit $\lambda \nearrow \lambda_{cr}$ to obtain the expression for $\lambda=\lambda_{cr}$. The corresponding limit of the righthand side is finite if and only if $\lambda_{cr} < \lambda_0$, and from Proposition \ref{prop:survival para}, this holds if and only if $\alpha \in(1,\rho^{-1})$.  
\end{proof}


\bibliographystyle{imsart-nameyear} 
\bibliography{references} 
\end{document}